\pdfoutput=1
\documentclass[leqno]{shinyart}

\usepackage[utf8]{inputenc}

\usepackage{shinybib}

\usepackage{enumitem} 
\usepackage{booktabs}
\usepackage{siunitx}
\usepackage{autonum}

\renewcommand{\phi}{\varphi}
\newcommand{\N}{\mathbb{N}}

\newcommand{\R}{\mathbb{R}}
\newcommand{\calF}{\mathcal{F}}
\newcommand{\calG}{\mathcal{G}}
\newcommand{\calE}{\mathcal{E}}

\newcommand{\calS}{\mathcal{S}}

\newcommand{\calM}{{\mathcal{M}}}
\newcommand{\calO}{{\mathcal{O}}}
\newcommand{\scalprod}[1]{\langle #1 \rangle}
\newcommand{\abs}[1]{| #1 |}
\newcommand{\norm}[1]{\| #1 \|}

\DeclareMathOperator{\sign}{\mathrm{sign}}
\DeclareMathOperator{\dom}{\mathrm{dom}}
\DeclareMathOperator{\ran}{\mathrm{ran}}
\DeclareMathOperator{\Id}{\mathrm{Id}}
\DeclareMathOperator{\co}{\mathrm{co}}
\DeclareMathOperator*{\argmin}{\mathrm{arg\,min}}

\newcommand{\prox}{\mathrm{prox}}

\DeclareMathOperator{\divergence}{div}
\DeclareMathOperator{\grad}{grad}

\DeclareMathOperator{\extR}{\overline{\R}}
\newcommand{\defgl}{\mathrel{\mathop:}=}

\newcommand{\dd}{{\mathrm d}}
\renewcommand{\u}{{\bar u}}

\setlist[enumerate]{label={(\roman*)}}
\usepackage{booktabs}
\usepackage{tikz,pgfplots}
\pgfplotsset{compat=newest}
\pgfplotsset{plot coordinates/math parser=false}
\usetikzlibrary{arrows.meta}
\usetikzlibrary{plotmarks}
\usepgfplotslibrary{polar} 

\allowdisplaybreaks

\title{Vector-valued multibang control of differential equations}
\author{%
    Christian Clason\thanks{Faculty of Mathematics, Universität Duisburg-Essen, 45117 Essen, Germany (\email{christian.clason@uni-due.de})}
    \and
    Carla Tameling\thanks{Institute for Mathematical Stochastics, Universität Göttingen, Goldschmidtstr. 7,
    37077 Göttingen, Germany (\email{carla.tameling@mathematik.uni-goettingen.de})}
    \and 
    Benedikt Wirth\thanks{Applied Mathematics, Universität Münster, Einsteinstr. 62,
    48149 Münster, Germany (\email{benedikt.wirth@uni-muenster.de})}
}
\date{September 15, 2017}

\hypersetup{
    pdftitle={Vector-valued multibang control of differential equations},
    pdfauthor={Christian Clason, Carla Tameling, Benedikt Wirth},
    pdfkeywords={multibang control, convex relaxation, Bloch equation, linear elasticity, semismooth Newton method}
}

\begin{document}

\maketitle

\begin{abstract}
    We consider a class of (ill-posed) optimal control problems in which a distributed vector-valued control is enforced to pointwise take values in a finite set $\calM\subset\R^m$.
    After convex relaxation, one obtains a well-posed optimization problem, which still promotes control values in $\calM$.
    We state the corresponding well-posedness and stability analysis and exemplify the results for two specific cases of quite general interest,
    optimal control of the Bloch equation and optimal control of an elastic deformation.
    We finally formulate a semismooth Newton method to numerically solve a regularized version of the optimal control problem and illustrate the behavior of the approach for our example cases.
\end{abstract}

\section{Introduction}

We consider the optimization problem
\begin{equation}\label{eq:optControlPb}
    \min_{u \in U} \frac12\norm{S(u) - z}_Y^2 + \int_{\Omega}  g(u(x))\,\dd x\,, 
\end{equation}
where $\Omega \subset \R^n$ is an open bounded domain, $U=L^2(\Omega;\R^m)$ for some $m\geq 2$, $Y$ is a Hilbert space, $z\in Y$, $S\colon U \to Y$ is a compact and Fr\'{e}chet differentiable (possibly nonlinear) operator, and the pointwise \emph{vector multibang penalty} $g \colon \R^m \to \R \cup \left\{ \infty\right\}$ has a convex polyhedral epigraph and superlinear growth at infinity. This extends the class of scalar problems considered in \cite{CK:2013,CK:2015} to the vector-valued case.
The problem may be viewed either as an optimal control problem, in which we try to choose the control $u$ such that the state $y=S(u)$ comes close to a prescribed desired value $z$,
or as an inverse problem, in which a measurement $z$ has been obtained via a forward operator $S$ from a physical configuration $u$, which we try to recover.
To make the problem well-posed, a regularization typically has to be incorporated, which encodes some a priori knowledge or requirement of $u$.
With the term $\int_{\Omega} g(u(x))\,\dd x$ we make a very particular regularization choice,
and our main interest in this article is the behavior and influence of this choice on the solution,
which we will study by way of example for two different operators $S$ (the solution operator of the Bloch equation and of linearized elasticity) and specific costs $g$ (whose graph is given by a polyhedral cone and a square frustum).
The basic underlying intuition is that this regularization in combination with a quadratic discrepancy term increasingly promotes values of $u$ on lower-dimensional facets, and in particular 
at the vertices, of the graph of $g$, since the linear growth away from a vertex will lead to a comparatively greater increase in the penalty than the corresponding decrease in the discrepancy term. The same mechanism is responsible for the sparsity-promoting property (i.e., the preference for $u=0$) of $L^1$ regularization; it is also related to the fact that in linear optimization, minima are always found at a vertex of the polyhedral domain.

\paragraph{Motivation}
Our regularization choice is motivated by scenarios in which $u$ is required to take values only in a prescribed finite set $\calM\subset\R^m$.
Examples include topology optimization, where the spatial material composition of a (mechanical) structure is optimized
and in which $\calM$ comprises the material parameters of the available material components,
or inverse problems in which the spatial distribution of a few known materials (or, in medical imaging, tissues with known properties) has to be identified.
This leads to the minimization of an energy
\begin{equation}
    \calE^\calM(u)= \frac12\norm{S(u) - z}_Y^2 + \int_{\Omega} \delta_\calM(u(x))\,\dd x
    \qquad\text{with }
\delta_\calM(u)=\begin{cases}0&\text{if }u\in\calM,\\\infty&\text{else.}\end{cases}
\end{equation}
Unfortunately, the energy $\calE^\calM$ is not weakly lower semi-continuous  \cite[Cor.\,2.14]{Braides:2002}
so that the problem is ill-posed (unless the inverse operator $S^{-1}$ were compact into $L^1(\Omega;\R^m)$,
in which case the energy would be strongly coercive in $L^1(\Omega;\R^m)$ and one would only require strong lower semi-continuity):
generically there are no minimizers, and controls $u$ with small energy $\calE^\calM(u)$ will rapidly oscillate between different values in $\calM$.
There are (at least) two possible ways out:
\begin{enumerate}[label=(\roman*)]
    \item
        The first approach adds a penalty of variations of $u$, for instance the total variation seminorm $\norm{u}_{TV}=\int_\Omega\dd|\nabla u|$ or a Mumford--Shah-type regularization functional,
        which has the effect of preventing oscillations and penalizing the interfaces between regions of different values of $u$.
        A disadvantage of this approach is that it quite explicitly regularizes the geometry of the material distribution, which is the sought quantity.
        For instance, such a regularization will lead to rounded-off interfaces that cannot have corners.
    \item
        The second approach considers instead the relaxation (i.e., the lower semi-continuous envelope) of $\calE^\calM$,
        thereby admitting also mixed control values $u(x)\notin\calM$ that represent mixtures of values in $\calM$.
        This is an obvious disadvantage; however, it might be alleviated by adding a convex (to ensure weak lower semi-continuity) cost $\int_\Omega c(u(x))\,\dd x$ that may for instance encode a known preference for a certain material.
        If this is done before relaxation, then mixed control values will no longer have equal costs to pure control values so that the relaxation may again lead to pure control values $u(x)\in\calM$. This has for instance been observed in \cite{CK:2013}.
\end{enumerate}
The additional cost regularization of the latter approach acts on the material amounts rather than the geometry of their distribution
and therefore is worthwhile studying as an alternative to the standard regularization via penalization of interfaces.
Specifically, the relaxation of $\int_\Omega \delta_\calM(u(x))\,\dd x+\alpha \int_\Omega c(u(x))\,\dd x$ for some $\alpha >0$ and $c:\R^m\to\R$ non-negative, strictly convex, and lower semi-continuous, is given by $\int_\Omega g(u(x))\,\dd x$ with
\begin{equation}\label{eq:reg}
    g=g_\infty^{**}\qquad\text{for }g_\infty:=\alpha c+\delta_\calM,
\end{equation}
where the double asterisk denotes the biconjugate or convex envelope.
Those functions $g$ are precisely the ones with convex polyhedral epigraph (since this epigraph is the convex hull of the finitely many points $\{(v,\alpha c(v)))\,:\,v\in\calM\}$, and any function $g$ with convex polyhedral epigraph can be obtained via $c=g/\alpha$ and an appropriate choice of $\calM$), which motivates our problem formulation \eqref{eq:optControlPb}.
While our theoretical statements will hold for any such choice of $c$, the explicit computation of $g$ and the numerical solution will be carried out as in \cite{CIK:2014,CK:2013,CK:2015} for the specific choice
\begin{equation}
    c(v)=\frac12|v|_2^2.
\end{equation}

In the case of a scalar function $u$ (i.e., for $m=1$), this optimization problem reduces to the one considered in \cite{CK:2013};
the difference in the vector-valued case is that now several (or even all) values in $\calM$ can be assigned the same control cost, therefore allowing for multiple equally preferred discrete values. 
Providing explicit and numerically implementable characterizations of the required generalized derivatives is one of the main contributions of this work. Furthermore, we provide an extended analysis of the stability and multibang properties of the optimal controls in the general case.

\paragraph{Model problems}
To illustrate the broad applicability of the proposed approach, we consider as specific examples two different forward operators $S$ and admissible sets $\calM$
(the analysis in \crefrange{sec:existence}{sec:penalty} will be independent of these models, though, beyond some general assumptions).

The first example follows \cite{Spincontrol:15}, where the authors try to drive a collection of spin systems using external electromagnetic fields to a desired spin state in the context of NMR spectroscopy or tomography.
The hardware here only allows a discrete set of control values (the radiofrequency pulse phases and amplitudes).
The underlying model is given by the Bloch equation in a rotating reference frame without relaxation (see \cite{Epstein:2006} for an introduction), which relates the magnetization vector $\mathbf M:[0,T]\to\R^3$ and the applied magnetic field $\mathbf B:[0,T]\to \R^3$ via the bilinear differential equation
\begin{equation}
    \frac{\dd}{\dd t}{\mathbf{M}}(t) = \mathbf{M}(t) \times \mathbf B(t)\,,
    \qquad\mathbf{M}(0)=\mathbf{M}_0.
\end{equation}
The goal is to shift the magnetization vector from the initial state $\mathbf{M}_0$ (e.g., aligned to a strong external field) to a desired state $\mathbf{M}_d$ (e.g., orthogonal to the external field) at time $T$.
The control $u\in L^2((0,T);\R^2)$ enters the equation as $\mathbf{B}(t) = (u_1(t),u_2(t),\omega)$, where $\omega$ is a fixed resonance frequency (which coincides with the rotation frequency of the domain),
and thus the (nonlinear) operator $S$ maps the control $u$ onto the magnetization vector $\mathbf M(T)$ at time $T$.
For details, see \cref{sec:BlochStateEq}.

The second example deals with linearized elasticity as the most basic model of coupled PDEs as state equations, i.e., we consider $S$ to be the solution operator of the elliptic problem
\begin{equation}
    \left\{\begin{aligned}
            -2\mu \divergence \epsilon(y) - \lambda \grad \divergence y &= u \text{  in } \Omega, \\ 
            y &= 0 \text{  on }  \Gamma, \\
            (2\mu \epsilon(y) + \lambda \divergence y) n &= 0 \text{  on }  \partial\Omega\setminus\Gamma
    \end{aligned}\right.
\end{equation}
with distributed control $u$,
see \cref{sec:elasticityStateEq} for details.

\bigskip

Regarding the admissible set $\calM$, we consider for the case of the Bloch equation -- again following \cite{Spincontrol:15} -- radially distributed control values together with the origin, i.e.,
\begin{equation}
\calM = \left\{ \left(\begin{smallmatrix} 0 \\ 0 \end{smallmatrix}\right), \left(\begin{smallmatrix} \omega_0 \cos\theta_1 \\ \omega_0 \sin \theta_1 \end{smallmatrix}\right), \dots, \left(\begin{smallmatrix} \omega_0 \cos\theta_M \\ \omega_0 \sin \theta_M \end{smallmatrix}\right)\right\} 
\end{equation}
for a fixed amplitude $\omega_0>0$ and $M>2$ equi-distributed phases
\begin{equation}
    0\leq\theta_1<\ldots<\theta_M<2\pi\,.
\end{equation}
In this example, all admissible control values apart from $0$ have the same magnitude; it also provides a link to classical sparsity promotion and allows a closed-form treatment of an arbitrary number of such states.

For the case of linearized elasticity, we consider in addition an admissible set containing control values of different magnitude but not the origin. For the sake of an example, we make the concrete choice
\begin{equation}
\calM = \left\{ \begin{pmatrix} 1 \\ 1 \end{pmatrix}, \begin{pmatrix} 1 \\ -1 \end{pmatrix}, \begin{pmatrix} -1 \\ 1 \end{pmatrix}, \begin{pmatrix} -1 \\ -1 \end{pmatrix}, \begin{pmatrix} 2 \\ 2 \end{pmatrix}, \begin{pmatrix} 2 \\ -2 \end{pmatrix}, \begin{pmatrix} -2 \\ 2 \end{pmatrix}, \begin{pmatrix} -2 \\ -2 \end{pmatrix}\right\} \,.
\end{equation}
Beyond illustrating the general procedure, these two examples are meant as useful prototypes that should be directly applicable.

\paragraph{Related work}
Convex relaxation of problems lacking weak lower semi-continuity has a long history; here we only mention the monograph \cite{Ekeland:1999a}. In the context of optimal control of partial differential equations, convex relaxation of discrete control constraints was discussed in \cite{CK:2013,CK:2015}; a similar approach was applied to switching control in \cite{CIK:2014}. Special cases were treated much earlier for scalar controls. In particular, if $\calM$ contains only two points, problem \eqref{eq:optControlPb} coincides with a (regularized) bang-bang control problem; see, e.g., \cite{bergounioux_optimality_1996,tiba_optimal_1990,troltzsch_minimum_1979}. For $\calM=\{0\}$, the relaxation reduces to the well-known $L^1$ norm used to promote sparse controls; see, e.g., \cite{stadler_elliptic_2007,casas_optimality_2012,ito_optimal_2014}. 

There is a vast literature concerning pulse design in magnetic resonance imaging and spectroscopy via optimal control of the Bloch equation, e.g.,
\cite{conolly_optimal_1986,smith_solvent-suppression_1991,rosenfeld_design_1996,khaneja_optimal_2005,xu_designing_2008,gershenzon_design_2009,grissom_fast_2009,skinner_new_2012}. A mathematical treatment of this problem can be found in, e.g., \cite{bonnard_geometric_2014}. Numerical methods for the computation of optimal pulses are based on conjugate gradient methods (see, e.g., \cite{mao_selective_1986}), Krotov methods \cite{vinding_fast_2012}, quasi-Newton and Newton methods with approximate second derivatives \cite{anand_designing_2012} and Newton methods using exact second derivatives computed via the adjoint approach \cite{Aigner:2015} (which was also the basis of the winning approaches in the 2015 ISMRM RF Pulse Design Challenge \cite{ISMRMChallenge}). The latter is the basis for the numerical treatment in this work.

To the best of our knowledge, there is so far only a very limited number of works dealing with the design of discrete-valued pulses, which is of interest since the hardware often allows only a finite set $\calM$ of pulses \cite{cruickshank_kilowatt_2009,prigl_theoretical_1996}.
In \cite{Spincontrol:15}, this problem is treated via an extension of the approach from \cite{khaneja_optimal_2005} together with a quantization of a continuous control field obtained via standard optimization methods. 

\paragraph{Organization}
\Cref{sec:existence} provides the abstract convex analysis framework, including existence of solutions of the optimal control problem, necessary optimality conditions, as well as an appropriate regularization for numerical purposes.
\Cref{sec:stability} then derives stability results based on rather general assumptions on the state operator and the multibang penalty.
\Cref{sec:penalty} gives an explicit characterization of the convex analysis framework for the specific examples of the multibang penalty used in this work, while \cref{sec:stateEq} gives more detail about the model state equations and in particular verifies for them the previously exploited assumptions.
\Cref{sec:semiSmoothNewton} discusses the numerical solution using a semismooth Newton method. Finally, \cref{sec:examples} presents and discusses illustrative numerical examples for both model problems.

\section{Convex analysis framework}\label{sec:existence}

To obtain existence of minimizers and numerically feasible optimality conditions, we follow the general framework of \cite{CK:2015} (stated there for the scalar case), which we briefly summarize in this section and adapt to the vector-valued case.
We recall that $U=L^2(\Omega;\R^m)$ for some bounded open domain $\Omega\subset\R^n$ and $m\geq 2$, $Y$ is a Hilbert space, and
\begin{align}
    \calF&\colon U \to \R \cup \left\{ \infty \right\}, \quad u \mapsto \tfrac12\norm{S(u) - z}_Y^2\,,\\
    \calG&\colon U \to \R \cup \left\{ \infty \right\}, \quad u \mapsto \textstyle\int_{\Omega} g(u(x))\, \dd x\,,
\end{align}
for $g:\R^m\to\R\cup\{\infty\}$ proper, convex, lower semi-continuous with $\dom g = \co\calM$ (the convex hull of $\calM$) for some finite set $\calM\subset\R^m$.
For the operator $S$ we will require
\begin{enumerate}[label=(\textsc{h}\arabic*),ref={\normalfont(\textsc{h}\arabic*)}]
    \item\label{enm:weakWeakContinuity}
        weak-to-weak continuity, i.e.,\quad
        $u_i\rightharpoonup u\text{ in }U
        \quad\Rightarrow\quad
        S(u_i)\rightharpoonup S(u)\text{ in }Y$,
    \item\label{enm:FrechetDifferentiability}
        Fréchet differentiability.
\end{enumerate}
In the following, $\calG^*:U^*\cong U \to\R\cup\{\infty\}$ denotes the Legendre--Fenchel conjugate of $\calG$, and $S'(u)^*:Y\to U$ denotes the (Hilbert-space) adjoint of the Fréchet derivative of $S:U\to Y$.

We now consider the problem
\begin{equation}\label{eq:formal_prob}
    \min_{u\in U} \calE(u)\quad\text{ for }\calE(u):=\calF(u)+\calG(u)\,.
\end{equation}
The following statements are analogous to \cite[Prop.~2.1, Prop.~2.2]{CK:2015} for the vector-valued case.
\begin{proposition}[existence of minimizers] \label{thm:existence}
    Let $S$ satisfy \ref{enm:weakWeakContinuity}. Then there exists a solution $\bar u\in U$ to \eqref{eq:formal_prob}.
\end{proposition}
\begin{proof}
    Consider a minimizing sequence $\{u_i\}_{i\in\N}$.
    Since $g$ is infinite outside of $\co\calM$, we know that $\|u_i\|_{L^\infty(\Omega)}$ is uniformly bounded so that we may extract a subsequence, again denoted by $\{u_i\}_{i\in\N}$, weakly converging  in $U$ to some $\bar u\in U$.
    Now $\int_\Omega g(u(x))\,\dd x$ is sequentially weakly lower semi-continuous by the convexity of $g$,
    while property \ref{enm:weakWeakContinuity} implies weak convergence $S(u_i)\rightharpoonup S(\bar u)$ so that
    \begin{equation}
        \frac12\norm{S(\bar u) - z}_Y^2 + \int_{\Omega}  g(\bar u(x))\,\dd x
        \leq\liminf_{i\to\infty}\frac12\norm{S(u_i) - z}_Y^2 + \int_{\Omega}  g(u_i(x))\,\dd x.
    \end{equation}
    Hence $\bar u$ must be a minimizer.
\end{proof}

\begin{proposition}[optimality conditions] \label{thm:optCond}
    Let $S$ satisfy \ref{enm:FrechetDifferentiability} and let $\bar u\in U$ be a local minimizer of \eqref{eq:formal_prob}. Then there exists a $\bar p \in U$ satisfying
    \begin{equation}
        \label{eq:optsys}
        \left\{\begin{aligned}
                -\bar p &= \calF'(\bar u)=S'(\bar u)^*(S(\bar u)-z),\\
                \bar u &\in \partial\calG^*(\bar p).
        \end{aligned}\right.
    \end{equation}
\end{proposition}
\begin{proof}
    Abbreviate $u_t=\bar u+t(u-\bar u)$ for arbitrary $t>0$ and $u\in U$.
    Due to the optimality of $\bar u$ we have
    \begin{equation}
        0\leq[\calF(u_t)+\calG(u_t)]-[\calF(\bar u)+\calG(\bar u)]\,.
    \end{equation}
    Dividing by $t$ and rearranging, we arrive at
    \begin{equation}
        0\leq\frac{\calF(u_t)-\calF(\bar u)}t+\frac{\calG(u_t)-\calG(\bar u)}t
        \leq\frac{\calF(u_t)-\calF(\bar u)}t+\frac{(1-t)\calG(\bar u)+t\calG(u)-\calG(\bar u)}t\,,
    \end{equation}
    where in the second inequality we used the convexity of $\calG$.
    Taking the limit $t\to0$ and setting $\bar p=-\calF'(\bar u)$, we arrive at
    \begin{equation}
        0\leq\langle-\bar p,u-\bar u\rangle+\calG(u)-\calG(\bar u)\,.
    \end{equation}
    As this holds for all $u\in U$, we obtain $\bar p\in\partial\calG(\bar u)$,
    which is equivalent to $\bar u\in\partial\calG^*(\bar p)$.
\end{proof}

Note that 
\begin{equation}\label{eq:subdiff_pointwise}
    \partial\calG^*(p)=\{u\in U:u(x)\in\partial g^*(p(x))\text{ for a.e. }x\in\Omega\}\,.
\end{equation}
It is readily seen that for $g$ chosen as in \eqref{eq:reg}, $g^*$ is piecewise affine and thus $\partial g^*$ is single-valued in each affine region, the values being precisely the elements of $\calM$ (see \cref{sec:penalty}).
More precisely, for each $u\in\calM$ there is an open convex polyhedron $Q(u)\subset\R^m$ such that $\R^m=\bigcup_{u\in\calM}\overline{Q(u)}$ and $\partial g^*(q)=\{u\}$ for all $q\in Q(u)$.
This property suggests that solutions to \eqref{eq:optsys} generically satisfy $u\in\calM$ almost everywhere, which will be exploited in \cref{sec:stability} to derive corresponding stability properties of optimal controls.

\bigskip

In order to apply a semismooth Newton method in function spaces, we need to apply a regularization. Here we replace the subdifferential $\partial\calG^*(\bar p)$ by its single-valued Yosida approximation
\begin{equation}
    (\partial\calG^*)_\gamma(p) =\frac1\gamma\left(p - \prox_{\gamma\calG^*}(p)\right)
\end{equation}
for some $\gamma>0$ and the proximal mapping
\begin{equation}
    \prox_{\gamma\calG^*}(p) = \left(\Id + \gamma\partial\calG\right)^{-1}(p) = \argmin_{\tilde p\in U}\frac1{2\gamma}\norm{\tilde p-p}_{U}^2 + \calG^*(\tilde p),
\end{equation}
i.e., we consider instead of \eqref{eq:optsys}  for $\gamma>0$ the regularized optimality conditions
\begin{equation}\label{eq:regsystem}
    \left\{\begin{aligned}
            -p_\gamma &= \calF'(u_\gamma),\\
            u_\gamma &= (\partial\calG^*)_\gamma(p_\gamma).
    \end{aligned}\right.
\end{equation}
As we will show in \cref{sec:semiSmoothNewton}, $H_\gamma:=(\partial\calG^*)_\gamma$ is Newton-differentiable, thereby allowing the use of semi-smooth Newton methods.
The Yosida approximation $(\partial\calG^*)_\gamma$ is linked to the Moreau envelope 
\begin{equation}
    (\calG^*)_\gamma(p)=\min_{\tilde p\in U}\frac1{2\gamma}\norm{\tilde p-p}_{U}^2+\calG^*(\tilde p)
\end{equation}
via $(\partial \calG^*)_\gamma = \partial(\calG^*)_\gamma$,
see, e.g., \cite[Prop.~12.29]{Bauschke:2011}, which justifies the term \emph{Moreau--Yosida regularization} (of $\calG^*$). Furthermore, from \cite[Prop.~13.21]{Bauschke:2011} we have that
\begin{equation}
    ((\calG^*)_\gamma)^*(u) = \calG(u) + \frac\gamma2\norm{u}_U^2,
\end{equation}
and hence \eqref{eq:regsystem} coincides with the necessary and sufficient optimality conditions for the strictly convex minimization problem
\begin{equation}
    \label{eq:problem_reg}
    \min_{u\in U}\calE_\gamma(u)\quad\text{ for }\calE_\gamma(u)= \calF(u) + \calG(u) + \frac\gamma2 \norm{u}_{U}^2.
\end{equation}
By the same arguments as in the proof of \cref{thm:existence}, we obtain the existence of a minimizer $u_\gamma\in U$ and thus of a corresponding $p_\gamma = -\calF(u_\gamma)\in U$.
\begin{remark}
    An alternative regularization leading to Newton-differentiability is to instead apply the Yosida approximation to the equivalent subdifferential inclusion $\bar p\in\partial\calG(\bar u)$ in \eqref{eq:optsys}. This would correspond to replacing $\calG$ in \eqref{eq:formal_prob} with its (Fréchet-differentiable) Moreau envelope $\calG_\gamma:u\mapsto \min_{\tilde u\in U} \frac1{2\gamma}\norm{\tilde u-u}_U^2+\calG(\tilde u)$, thus smoothing out the non-differentiability that is responsible for the structural properties of the penalty. 
    In contrast, our regularization does not remove the non-differentiability but merely makes the functional (more) strongly convex so that the structural features of the multibang regularization are preserved.
\end{remark}

The following statement is a slight generalization of \cite[Prop.~4.1]{CK:2015}.
\begin{proposition}[limit for vanishing regularization]
    Let $S$ satisfy \ref{enm:weakWeakContinuity}. Then $\Gamma\text{-}\lim_{\gamma\to0}\calE_\gamma=\calE$ with respect to weak convergence in $U$.
    As a consequence, any sequence $u_{\gamma_n}$ of global minimizers to \eqref{eq:problem_reg} for $\gamma_n\to0$ contains a subsequence converging weakly in $U$ to a global minimizer of \eqref{eq:formal_prob}.
    Moreover, this convergence is strong.
\end{proposition}
\begin{proof}
    For the $\Gamma$-limit, we first have to show that for any sequence $\gamma_n\to0$ and any weakly converging sequence $u_n\rightharpoonup u$ we have $\liminf_{\gamma_n\to0}\calE_{\gamma_n}(u_{\gamma_n})\geq\calE(u)$,
    which is an immediate consequence of the sequential weak lower semi-continuity of $\calE$ (shown in the proof of \cref{thm:existence}) and of $\|\cdot\|_{U}$.
    Second, the required recovery sequence is just the constant sequence, $u_n=u$.
    Furthermore, minimizers of $\calE_\gamma$ are uniformly bounded in $U$, since $g$ is infinite outside the convex hull $\co\calM$,
    which together with the $\Gamma$-convergence is well-known to imply the weak convergence in $U$ of minimizers of $\calE_\gamma$ to minimizers of $\calE$.
    Finally, for such a weakly converging sequence $u_{\gamma_n}\rightharpoonup u$ of minimizers of $\calE_{\gamma_n}$ we have
    \begin{equation}
        \calE(u_{\gamma_n})+\frac{\gamma_n}2\norm{u_{\gamma_n}}_{U}^2\leq\calE_{\gamma_n}(u)\leq\calE(u_{\gamma_n})+\frac{\gamma_n}2\norm{u}_{U}^2,
    \end{equation}
    which implies $\norm{u}_{U}\geq\norm{u_{\gamma_n}}_{U}$ so that the convergence $u_{\gamma_n}\to u$ is actually strong.
\end{proof}

\section{Stability properties of multibang controls}\label{sec:stability}

We now discuss stability properties of the controls by exploiting the special structure of the optimality conditions for the multibang control problem.
In particular we consider  in what sense the controls converge as the target state converges; what can be said about controls with values in $\calM$; and when exact controls (which achieve the target state) can be retrieved by the optimization.
To keep the notation concise, we set 
\begin{equation}
    \calE^z(u):=\frac12\norm{S(u) - z}_Y^2 + \int_{\Omega}  g(u(x))\,\dd x\,,
\end{equation}
where $g:\R\to\extR$ is again proper, convex, and weakly lower semi-continuous with $\dom g = \co \calM$.

\subsection{Stability with respect to target perturbations}

First, we examine how perturbations of the target $z$ influence the minimizer of \eqref{eq:formal_prob}.
We will see that as $z_n$ converges strongly to $z$ in $Y$, the corresponding minimizers converge in $U$ in the weak sense.
Strong convergence cannot be expected in general due to worst-case scenarios in which the limit minimizer $\bar u$ has a nonempty ``singular arc''
\begin{equation}
    \calS_{\bar u}=\{x\in\Omega\,|\,\bar u(x)\notin\calM\}\,,
\end{equation}
i.e., the region in which $\bar u$ does not attain any of the distinguished values $\calM$.
However, away from that singular arc one obtains strong convergence and, as a consequence, controls in $\calM$ even for perturbed targets.
In this section we use the following additional hypotheses on $S$
(which will be shown to hold for our model forward operators in \cref{sec:stateEq}).
\begin{enumerate}[label=(\textsc{h}\arabic*),ref={\normalfont(\textsc{h}\arabic*)}]\setcounter{enumi}{2}
    \item\label{enm:compactness}
        $S:U\to Y$ is compact.
    \item\label{enm:adjointConvergence}
        For some Banach space $V\hookleftarrow U$ with $V^*\hookrightarrow L^\infty(\Omega;\R^m)$, we have
        \begin{equation}
            \lim_{\tilde u\rightharpoonup u\text{ in }U}\norm{[S'(\tilde u)-S'(u)]^*y}_{V^*}=0\quad\text{ for all }y\in Y\,.
        \end{equation}
\end{enumerate}

\begin{proposition}[$\Gamma$-convergence of objective functional]\label{thm:GammaConvergence}
    Let $z_n\to z$ in $Y$ and $S$ satisfy \ref{enm:weakWeakContinuity}. Then with respect to weak convergence in $U$, we have
    \begin{equation}
        \Gamma\text{-}\lim_{n\to\infty}\calE^{z_n}=\calE^z\,.
    \end{equation}
\end{proposition}
\begin{proof}
    For the $\liminf$ inequality, let $u_n\rightharpoonup u$ weakly in $U$, then by property \ref{enm:weakWeakContinuity} and the weak lower semi-continuity of $\norm{\cdot}_Y$ and the convexity of $g$, we have
    \begin{equation}
        \begin{aligned}
            \liminf_{n\to\infty}\calE^{z_n}(u_n)
            &=\liminf_{n\to\infty}\frac12\norm{S(u_n) - z_n}_Y^2 + \int_{\Omega}  g(u_n(x))\,\dd x\\
            &\geq\frac12\norm{S(u) - z}_Y^2 + \int_{\Omega}  g(u(x))\,\dd x
            =\calE^z(u)\,.
        \end{aligned}
    \end{equation}
    For the $\limsup$ inequality, choose $u_n=u\in U$ to obtain
    \begin{equation}
        \begin{split}
            \limsup_{n\to\infty}\calE^{z_n}(u_n)
            =\limsup_{n\to\infty}\frac12\norm{S(u) - z_n}_Y^2 + \int_{\Omega}  g(u(x))\,\dd x
            =\calE^z(u)\,.
            \qedhere
        \end{split}
    \end{equation}
\end{proof}

This proposition now implies a weak stability of the control.
\begin{corollary}[stability of control and state]
    Under the conditions of \cref{thm:GammaConvergence} and \ref{enm:compactness}, any sequence $\{u_n\}_{n\in\N}$ of minimizers of $\calE^{z_n}$ contains a subsequence converging weakly in $U$ to a minimizer $\bar u$ of $\calE^z$.
    The corresponding states $y_n=S(u_n)$ converge strongly in $Y$ to $\bar y=S(\bar u)$.
\end{corollary}
\begin{proof}
    Since $g$ is infinite outside $\co\calM$ we know that $\norm{u}_{L^\infty(\Omega;\R^n)}$ is uniformly bounded among all $u\in U$ with finite energy $\calE^{z_n}(u)$,
    where the bound is independent of $n$.
    Thus, the $\calE^{z_n}$ are equi-mildly coercive so that the convergence of minimizers $u_n$ follows from the $\Gamma$-convergence of the functionals.
    The convergence of states $y_n=S(u_n)\to \bar y=S(\bar u)$ along the subsequence follows from $u_n\rightharpoonup \bar u$ together with properties \ref{enm:weakWeakContinuity} and \ref{enm:compactness} (weak-to-weak continuity and compactness of $S$, respectively).
\end{proof}

Under additional assumptions, we also obtain convergence of the dual variable.
\begin{corollary}[stability of dual]\label{thm:stability}
    Under the conditions of \cref{thm:GammaConvergence} and \ref{enm:weakWeakContinuity}--\ref{enm:adjointConvergence},
    consider the sequence of minimization problems $\min_{u\in U}\calE^{z_n}(u)$.
    The corresponding optimal controls $u_n$, states $y_n$, and dual variables $p_n$ satisfy up to a subsequence
    \begin{equation}
        u_n\rightharpoonup \bar u\;\text{in }U\,,
        \quad
        y_n\to \bar y\;\text{in }Y\,,
        \quad\text{and}\quad
        p_n\to \bar p\;\text{in }V^*\,,
    \end{equation}
    where $\bar u$ is a minimizer of $\calE^z$, $\bar y = S(\bar u)$, and $\bar p$ satisfies \eqref{eq:optsys}.
\end{corollary}
\begin{proof}
    We already know $u_n\rightharpoonup \bar u$ and $y_n\to \bar y$.
    By the Banach--Steinhaus theorem and \ref{enm:adjointConvergence}, $[S'(u_n)-S'(\bar u)]^*$ is uniformly bounded in $L(Y;V^*)$ and thus also $S'(u_n)^*$. Now
    \begin{equation}
        \begin{split}
            \begin{aligned}[b]
                \norm{p_n-\bar p}_{V^*}
                &=\norm{S'(u_n)^*(z_n-y_n)-S'(\bar u)^*(z-\bar y)}_{V^*}\\
                &\leq\norm{S'(u_n)^*(z_n-y_n)-S'(u_n)^*(z-\bar y)}_{V^*}+\norm{S'(u_n)^*(z-\bar y)-S'(\bar u)^*(z-\bar y)}_{V^*}\\
                &\leq\norm{S'(u_n)^*}_{L(Y;V^*)}\norm{z_n-y_n-(z-\bar y)}_Y+\norm{[S'(u_n)^*-S'(\bar u)^*](z-\bar y)}_{V^*}
                \to0\,.
            \end{aligned}
            \qedhere 
        \end{split}
    \end{equation}
\end{proof}

The final result shows strong convergence of controls outside the singular arc, which will be seen to correspond to the case where $\partial g^*(\bar p(x))$ is set-valued (cf.~\eqref{eq:conj_subdiff_radial} and \eqref{eq:conj_subdiff_concentric}).
\begin{proposition}[locally strong convergence of control]\label{thm:localConvergence}
    Let the conditions of \cref{thm:GammaConvergence} and \ref{enm:weakWeakContinuity}--\ref{enm:adjointConvergence} hold.
    Furthermore, let $Q$ be the set on which $\partial g^*$ is single-valued, and abbreviate $\Omega_P=\{x\in\Omega:p(x)\in P\}$ for given $P\subset\R^m$. Then we have
    \begin{enumerate}[label=(\roman*)]
        \item for any $P\subset\subset Q$ compact and $n$ large enough, $u_n(x) = \bar u(x) \in \calM$ for a.e. $x\in \Omega_P$;
        \item $u_n|_{\Omega_Q}\to \bar u|_{\Omega_Q}$ strongly in $L^2(\Omega_Q;\R^m)$ and $\bar u(x)\in \calM$ for a.e. $x\in \Omega_Q$.
    \end{enumerate}
\end{proposition}
\begin{proof}
    By \cref{thm:stability}, we have $p_n\to \bar p$ in $L^\infty(\Omega;\R^m)$.
    In particular, for $n$ large enough, for all $x\in\Omega_P$ the value $p_n(x)$ lies in the same connected component of $Q$ as $\bar p(x)$.
    Hence, $u_n(x)=\bar u(x)$ due to $u_n(x)\in\partial g^*(p_n(x))=\partial g^*(\bar p(x))$ and  $\bar u(x)\in\partial g^*(\bar p(x))$.
    Since this holds for any compact subset $P$ of $Q$, we actually have pointwise convergence $u_n(x)\to \bar u(x)$ for almost all $x\in\Omega_Q$.
    The uniform boundedness of $u_n$ (since otherwise $g(u_n(x))=\infty$) then implies strong convergence by the dominated convergence theorem.
\end{proof}

\subsection{Controls in $\calM$}

Here, we examine more closely controls taking values only in $\calM$. In the following, we refer to minimizers $\bar u\in U$ of $\calE^z$ with $\bar u(x) \in \calM$ for a.e. $x\in \Omega$ as \emph{multibang controls}.
First, we note that such controls allow to achieve an energy arbitrarily close to the optimum.
\begin{remark}[near-optimality]
    Under hypotheses \ref{enm:weakWeakContinuity} and \ref{enm:compactness}, we have
    \begin{equation}
        \min_{u\in U} \calE^z(u)= \inf_{\substack{u\in U\\ u(x)\in\calM \text{ a.e.}}} \calE^z(u)\,.
    \end{equation}
    Indeed, let $\bar u\in U$ minimize $\calE^z$.
    By the definition of $g$, there exists a sequence $\{u_n\}_{n\in\N}\subset U$ with $u_n(x)\in \calM$ a.e., $u_n\rightharpoonup \bar u$ in $U$, and $\int_\Omega g(u_n(x))\,\dd x\to\int_\Omega g(\bar u(x))\,\dd x$.
    Furthermore, $S(u_n)\to S(\bar u)$ in $Y$ so that $\calE^z(u_n)\to\calE^z(\bar u)$.
\end{remark}

In the remainder of this subsection, we shall restrict ourselves to the case that
\begin{enumerate}[label=(\textsc{h}\arabic*),ref={\normalfont(\textsc{h}\arabic*)}]\setcounter{enumi}{4}
    \item\label{enm:linearity}
        $S:U\to Y$ is linear,
\end{enumerate}
which will only apply to the elasticity example, but not to the Bloch setting.
The intuition is that the case with multibang controls is generic (or even that targets with non-multibang controls, i.e., $u(x)\notin\calM$ on a non-negligible set, are nowhere dense in $Y$).
This is consistent with \cref{thm:localConvergence}, since targets with a singular arc of zero measure (or rather with $\Omega_Q=\Omega$) can be perturbed without producing a singular arc. 
Below we will at least see that targets leading to multibang controls  are dense in $Y$,
and that the mapping $z\mapsto\argmin_{u\in U}\calE^z(u)$ is not continuous in any target $z$ for which the singular arc has positive measure.
\begin{proposition}[approximation via multibang control]\label{thm:discrControlAppr}
    Let $S$ satisfy \ref{enm:weakWeakContinuity}--\ref{enm:linearity}.
    Then for any $z\in Y$ and corresponding minimizer $\bar u\in U$ of $\calE^z$, there exists a sequence $\{z_n\}_{n\in\N} \subset Y$ with $z_n\to z$ such that the corresponding minimizers $u_n\in U$ of $\calE^{z_n}$ satisfy $u_n(x)\in\calM$ a.e.,  $u_n\rightharpoonup \bar u$, and $\calE^{z_n}(u_n)=\calE^z(\bar u)$.
\end{proposition}
\begin{proof}[Sketch of proof]
    By \eqref{eq:optsys}, we have $\bar p=S^*(z-S\bar u)$ and $\bar u(x)\in\partial g^*(\bar p(x))$ for almost all $x\in\Omega$.
    The piecewise affine structure of $g^*:\R^m\to\R$ implies that $\bar u(x)$ is a convex combination of (at most) $m+1$ values $\hat u_j\in\calM\cap\partial g^*(\bar p(x))$.
    Thus one can find $u_n\rightharpoonup \bar u$ with $u_n(x)\in\calM\cap\partial g^*(\bar p(x))$ for almost all $x\in\Omega$.
    Choosing $z_n=Su_n+(z-S\bar u)$, we have $z_n\to z$ as well as $\bar p=S^*(z_n-Su_n)$ and $u_n(x)\in\partial g^*(\bar p(x))$ for almost all $x\in \Omega$. Hence by the convexity of the energy $\calE^{z_n}$, $u_n$ is a minimizer of $\calE^{z_n}$.
    Furthermore, one can even choose $u_n$ such that $\int_\Omega g(u_n(x))\,\dd x=\int_\Omega g(\bar u(x))\,\dd x$ so that $\calE^{z_n}(u_n) =  \calE^z(\bar u)$ as claimed.
\end{proof}

\begin{corollary}[strong convergence of control]
    Let the conditions of \cref{thm:discrControlAppr} hold. Then: 
    \begin{enumerate}[label=(\roman*)]
        \item The targets $z$ admitting a multibang control $\bar u$ minimizing $\calE^z$ are dense in $Y$.
        \item If $S$ is injective and the minimizer $\bar u$ to $\calE^z$ has a singular arc of positive measure, then one cannot have strong convergence of minimizers $u_n$ of $\calE^{z_n}$ for all $z_n\to z$.
    \end{enumerate}
\end{corollary}
\begin{proof}
    The first statement is a direct consequence of \cref{thm:discrControlAppr}. The second statement follows from the strict convexity of $\calE^z$ and thus the uniqueness of its minimizers, together with the fact that strong convergence in $U$ implies pointwise convergence:
    Indeed, let $\bar u$ have a singular arc $\calS_{\bar u}$ of positive measure and choose $z_n\to z$ such that the unique minimizers $u_n$ of $\calE^{z_n}$ are multibang controls (which is possible by the first statement).
    If we had strong convergence $u_n\to\bar u$ in $U$, then (up to a subsequence) also $u_n\to\bar u$ pointwise almost everywhere, in particular on $\calS_{\bar u}$.
    This contradicts $u_n(x)\in\calM$ almost everywhere.
\end{proof}

\subsection{Retrieval of exact controls}

We now consider more specifically the consequence of the convex relaxation \eqref{eq:reg} for some non-negative and strictly convex $c:\R^m\to\R$.
A peculiar feature of the multibang control in this case is that for attainable targets -- i.e., if there exists a $\hat u\in U$ such that $z=S(\hat u)$ -- the generating control $\hat u$ can only be recovered as a minimizer $\bar u$ of the optimization problem~\eqref{eq:formal_prob} if 
$c(\hat u(x))= \min_{v\in\calM} c(v)$ almost everywhere. This demonstrates the desirability to allow multiple admissible control values of equal magnitude. \begin{proposition}[achievement of target]
    If $S$ satisfies \ref{enm:FrechetDifferentiability}, then, for any minimizer  $\bar u\in U$ of $\calE^z$ that satisfies $S(\bar u) = z$, it holds that $g(\bar u(x)) = \min_{v\in\calM} g(v)$ almost everywhere.
    In particular, if in addition $\bar u(x)\in\calM$ almost everywhere, then $c(\bar u(x))=\min_{v\in\calM} c(v)$.
\end{proposition}
\begin{proof}
    If $S(\bar u)=z$, the first relation in the optimality condition \eqref{eq:optsys} together with linearity of $S'(\bar u)$ implies $\bar p=0$. Hence, the second relation yields $\bar u\in\partial\calG^*(0)$ and therefore $0\in\partial\calG(\bar u)$. By \eqref{eq:subdiff_pointwise}, this implies $0\in\partial g(\bar u(x))$ for almost all $x\in\Omega$ and therefore 
    \begin{equation}
        g(\bar u(x)) = \min_{v\in \R^m} g(v) = \inf_{v\in\R^m} g_\infty(v) = \inf_{v\in \calM} \alpha c(v) =  \min_{v\in \calM} g(v)
    \end{equation}
    since $\min f^{**} = \inf f$ by the properties of the convex hull, see, e.g., \cite[Prop.~12.9\,(iii)]{Bauschke:2011}.
\end{proof}
If, however, $c(\hat u(x))=\min_{v\in\calM} c(v)$ is not satisfied almost everywhere, then the generating control $\hat u$ can only be recovered in the limit $\alpha\to 0$.
In fact, in this limit the best approximation is achieved, i.e., an optimal control which yields the minimum possible tracking term $\calF$. In the following, we denote by $u_\alpha$ the minimizer of $\calE^z$ (which depends on $\alpha$ via the definition \eqref{eq:reg} of $g$) for given $\alpha>0$.
\begin{proposition}[$\Gamma$-convergence for vanishing regularization]
    For given $z\in Y$, let $M:=\inf_{u\in U}\norm{S(u)-z}_Y$ and $\calO:=\{u\in U:\norm{S( u)-z}_Y=M\}$.
    If $S$ satisfies \ref{enm:weakWeakContinuity}, then with respect to weak convergence in $U$ we have
    \begin{equation}
        \Gamma\text{-}\lim_{\alpha\to0}\frac1\alpha\left(\calE^z-\frac{M^2}2\right)
        =\delta_{\calO}+\calG_1
    \end{equation}
    where
    \begin{equation}
        \calG_1(u)=\int_{\Omega}  g_1^{**}(u(x))\,\dd x
        \quad\text{ for }\quad
        g_1(u)=c(u) + \delta_{\calM}(u)\,.
    \end{equation}
\end{proposition}
\begin{proof}
    The limsup inequality is trivial using the constant sequence; for the liminf inequality we only have to consider a sequence $u_\alpha\rightharpoonup u\notin \calO$.
    In that case, 
    \begin{equation}
        \liminf_{\alpha\to0}\norm{S(u_\alpha)-z}_Y\geq\norm{S(u)-z}_Y>M
    \end{equation}
    so that 
    \begin{equation}
        \begin{split}
            \frac1\alpha\left(\min_{u \in U} \frac12\norm{S(u) - z}_Y^2 + \int_{\Omega}  g(u(x))\,\dd x-\frac{M^2}2\right)\to\infty.
            \qedhere
        \end{split}
    \end{equation}
\end{proof}

\begin{corollary}[approximation of target]
    Under the conditions of the previous proposition, if $\calO\neq\emptyset$, then any family $\{u_\alpha\}_{\alpha>0}$ of minimizers of $\calE^z$ contains a subsequence converging weakly to a minimizer $\bar u\in \calO$ of $\calG_1$.
\end{corollary}
\begin{proof}
    This follows from the equi-mild coerciveness of the energies and the $\Gamma$-convergence, see \cite[Def.\,1.19 \& Thm.\,1.21]{Braides:2002}.
\end{proof}

\section{Vector-valued multibang penalty}\label{sec:penalty}

To implement the general framework of \cref{sec:existence}, we need explicit characterizations of the Fenchel conjugate and its subdifferential as well as its Moreau--Yosida regularization.
Here we consider the specific multibang penalty \eqref{eq:reg} for the choice $c(v) = \frac12|v|_2^2$, i.e., $\calG$ is defined as an integral functional for the normal integrand 
\begin{equation}\label{eq:multibangPenalty}
    g = \left(\frac\alpha2 |\cdot|_2^2 +\delta_\calM\right)^{**}=g_\infty^{**}.
\end{equation}
We can thus proceed by pointwise computation, where we need to differentiate based on the specific choice of the admissible set $\calM$.

We first summarize the general procedure. Since $g^* = (g_\infty^{**})^* = (g_\infty^*)^{**} = g_\infty^*$, the Legendre--Fenchel conjugate of $g$ is given by
\begin{equation}\label{eq:conj}
    g^*(q) = g_\infty^*(q) = \sup_{v\in\R^m} \scalprod{v,q} - g_\infty(v) = \max_{v\in \calM}\,  \scalprod{v,q} - \tfrac\alpha2 |v|_2^2.
\end{equation}
Hence, $g^*$ is the maximum of a finite number of convex and continuous  functions, and we can thus compute the subdifferential using the maximum rule; see, e.g., \cite[Prop.~4.5.2, Rem.~4.5.3]{Schirotzek:2007}. Setting 
\begin{equation}
    g_v^*(q) \defgl \scalprod{v,q} - \tfrac\alpha2 |v|_2^2,
\end{equation}
we have
\begin{equation}\label{eq:conj_subdiff}
    \partial g^*(q) = \co \bigcup_{\substack{v\in\calM:\\ g^*(q) = g_v^*(q)}}\partial g_v^*(q) = \co \left\{v\in\calM:g^*(q) = g_v^*(q)\right\}
\end{equation}
with $\co$ denoting the convex hull.
Finally, for the proximal mapping 
\begin{equation}
    \prox_{\gamma g^*}(q) := \argmin_{w\in\R^m} \frac1{2\gamma} |w-q|_2^2 + g^*(w) =  (\Id + \gamma \partial g^*)^{-1}(q)\,,
\end{equation}
we will make use of the equivalence 
\begin{equation}\label{eq:proxMapRelation}
    w = (\Id + \gamma \partial g^*)^{-1}(q)\qquad \Leftrightarrow\qquad q \in (\Id + \gamma \partial g^*)(w) = \{w\} + \gamma \partial g^*(w)
\end{equation}
and follow the case distinction in the maximum rule \eqref{eq:conj_subdiff}.
The Moreau--Yosida regularization of $\partial g^*$ is then given by
\begin{equation}\label{eq:my_reg}
    (\partial g^*)_{\gamma}(q) = \frac{1}{\gamma}\left(q - \prox_{\gamma g^*}(q)\right).
\end{equation}
For details, we refer to, e.g., \cite{Bauschke:2011}.

\subsection{Radially distributed control values}\label{sec:multibang:bloch}

Here, we take as set $\calM\subset\R^2$ of  admissible control values  the vector $0$ together with vectors of fixed amplitude $\omega_0>0$ and $M>2$ equidistributed phases
\begin{equation}
    0\leq\theta_1<\ldots<\theta_M<2\pi
\end{equation}
(where we shall assume $\theta_{i+1}-\theta_i<\pi$ for $i=1,\ldots,M-1$ and $\theta_1-(\theta_M-2\pi)<\pi$),
that is,
\begin{equation}\label{eq:setBloch}
\calM = \left\{ \left(\begin{smallmatrix} 0 \\ 0 \end{smallmatrix}\right), \left(\begin{smallmatrix} \omega_0 \cos\theta_1 \\ \omega_0 \sin \theta_1 \end{smallmatrix}\right), \dots, \left(\begin{smallmatrix} \omega_0 \cos\theta_M \\ \omega_0 \sin \theta_M \end{smallmatrix}\right)\right\} \eqqcolon \left\{ \u_0, \u_1, \dots \u_M\right\}\,.
\end{equation}

In the following it will be helpful to identify an angle $\theta\in[0,2\pi)$ with the corresponding point $\vec\theta=(\cos\theta,\sin\theta)$ on the unit circle $S^1$.
Let $\phi_i$ denote the midpoint between $\theta_i$ and $\theta_{i+1}$ (identifying $\theta_{M+1}=\theta_1$ for simplicity), that is, $\vec\phi_i=(\vec\theta_i+\vec\theta_{i+1})/|\vec\theta_i+\vec\theta_{i+1}|_2$, and introduce the circular sectors
\begin{equation}
    C_i=\left\{\omega\vec\theta\in\R^2:\theta\in(\phi_i,\phi_{i+1}),\,\omega\geq0\right\}\,.
\end{equation}
Here, $\theta\in(\phi_i,\phi_{i+1})$ is to be understood $2\pi$-periodically, that is, $\phi_{M+1}$ shall be identified with $\phi_1$, and $(\phi_i,\phi_{i+1})$ with $\phi_{i+1}<\phi_i$ shall be interpreted as $(\phi_i,\phi_{i+1}+2\pi)$.

\paragraph{Fenchel conjugate}

Using the equivalence of angles and sectors introduced above, it is straightforward to see
\begin{equation}
    \scalprod{q, \u_i} \geq \scalprod{q, \u_j}\;\text{ for all } q\in\overline{C_i},\,j\neq0\,.
\end{equation}
Thus, inserting the concrete choice of $\calM$ into \eqref{eq:conj}, we obtain
\begin{equation}
    g^*(q) =
    \begin{cases} 
        0 &  \text{if } \scalprod{q, \u_i} \leq  \frac\alpha2\omega_0^2\text{ for all } 1\leq i\leq M,\\
        \scalprod{q,\u_i} - \frac\alpha2\omega_0^2 &  \text{if }q\in\overline{C_i}\text{ and }\scalprod{q, \u_i} \geq \frac\alpha2\omega_0^2.
    \end{cases}
\end{equation}
Let us therefore introduce the sets (cf.~\cref{fig:subdomainsbloch:subdiff})
\begin{align}
    Q_0 &\defgl  \left\{ q \in \R^2 :\scalprod{q, \u_i} < \tfrac\alpha2\omega_0^2\text{ for all } 1\leq i\leq M\right\} ,\\
    Q_i &\defgl  \left\{ q \in C_i :\scalprod{q, \u_i} > \tfrac\alpha2\omega_0^2 \right\},&1\leq i\leq M,\\
    Q_{i_1\ldots i_k} &\defgl \bigcap_{i\in\{i_1,\ldots,i_k\}}\overline{Q_i}\setminus\bigcup_{i\notin\{i_1,\ldots,i_k\}}\overline{Q_i},&0\leq i_1,\ldots,i_k\leq M.
\end{align}
With this notation we obtain
\begin{equation}
    g^*(q)=
    \begin{cases} 
        0 &  \text{if } q\in\overline{Q_0},\\
        \scalprod{q,\u_i} - \frac\alpha2\omega_0^2 &  \text{if }q\in\overline{Q_i},\quad 1\leq i\leq M.
    \end{cases}
\end{equation}

\paragraph{Subdifferential}

From the maximum rule \eqref{eq:conj_subdiff}, we directly obtain
\begin{equation}\label{eq:conj_subdiff_radial}
    \partial g^*(q)= 
    \begin{cases}
        \{\u_i\}&\text{if }q\in Q_i,\quad\qquad 0\leq i\leq M,\\
        \co\{\u_{i_1},\ldots,\u_{i_k}\}&\text{if }q\in Q_{i_1\ldots i_k},\quad 0\leq i_1,\ldots,i_k\leq M.
    \end{cases}
\end{equation}

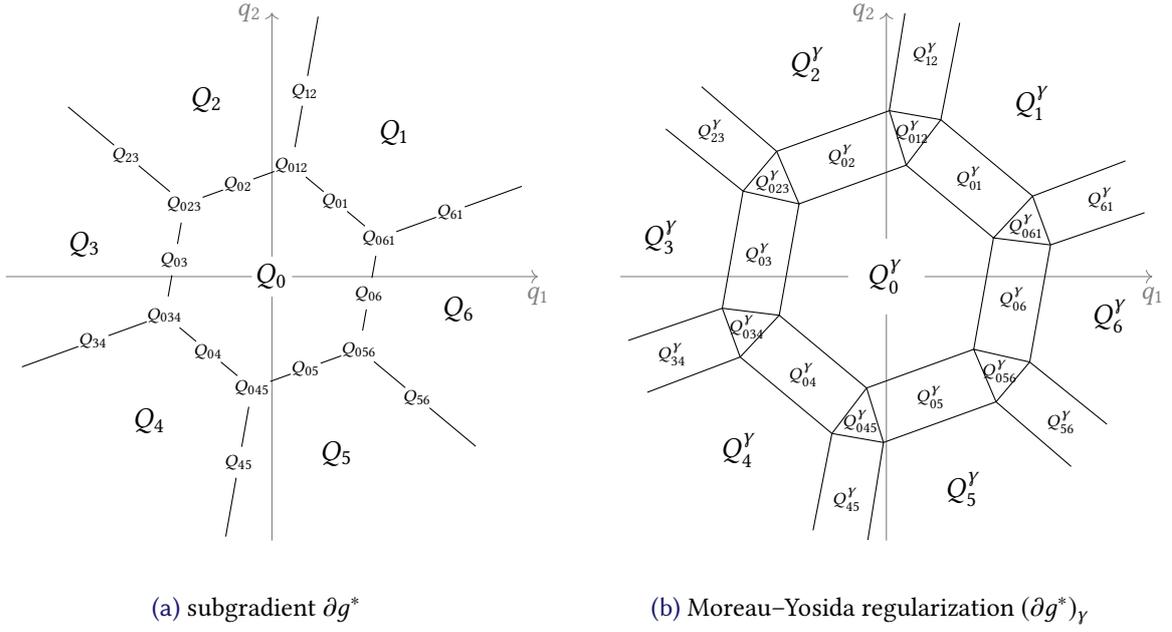
\begin{figure}[t]
    \begin{subfigure}[t]{0.45\linewidth}
        \centering
        \begin{tikzpicture}
    \draw[gray,->](0,-3.5) -- (0,3.5) node[left] {\small $q_2$};
    \draw[gray,->](-3.5,0) -- (3.5,0) node[below] {\small $q_1$};
    \draw (0,0) node[circle,fill=white,inner sep=0pt] {$Q_0$};
    \foreach \x  in {1,...,6}{
        \coordinate (q\x) at (20+\x*60:1.5cm);
        \coordinate (Q\x) at (20+\x*60:3.5cm);
        \node at (\x*60-10:2.5cm) {$Q_\x$};
        \draw (q\x) -- (Q\x);
    }
    \draw (q1) -- (q2) -- (q3) -- (q4) -- (q5) -- (q6) -- cycle;
    \foreach \x [evaluate=\x as \xnext using {int(mod(\x,6)+1)}] in {1,...,6}{
        \draw (\x*60-10:1.3cm) node[circle,fill=white,inner sep=0pt] {\tiny $Q_{0\x}$};
        \draw (20+\x*60:2.5cm) node[circle,fill=white,inner sep=0pt] {\tiny $Q_{\x\xnext}$};
        \draw (20+\x*60:1.5cm) node[circle,fill=white,inner sep=0pt] {\tiny $Q_{0\x\xnext}$};
    }
\end{tikzpicture}
        \caption{subgradient $\partial g^*$}\label{fig:subdomainsbloch:subdiff}
    \end{subfigure}
    \hfill
    \begin{subfigure}[t]{0.45\linewidth}
        \centering
        \begin{tikzpicture}
    \draw[gray,->](0,-3.5) -- (0,3.5) node[left] {\small $q_2$};
    \draw[gray,->](-3.5,0) -- (3.5,0) node[below] {\small $q_1$};
    \draw (0,0) node[circle,fill=white] {$Q^\gamma_0$};
    \foreach \x [evaluate=\x as \xnext using {int(mod(\x,6)+1)}] in {1,...,6}{
        \coordinate (q\x) at (20+\x*60:1.5cm);
        \coordinate (q\x1) at (20+\x*60-9:2.2cm);
        \coordinate (q\x2) at (20+\x*60+9:2.2cm);
        \coordinate (Q\x1) at (20+\x*60-6:3.5cm);
        \coordinate (Q\x2) at (20+\x*60+6:3.5cm);
        \node at (\x*60-10:3cm) {$Q^\gamma_\x$};
        \node at (\x*60-10:1.7cm) {\tiny $Q^\gamma_{0\x}$};
        \node at (20+\x*60:3cm) {\tiny $Q^\gamma_{\x\xnext}$};
        \node at (20+\x*60:1.95cm) {\tiny $Q^\gamma_{0\x\xnext}$};
    }
    \foreach \x [evaluate=\x as \xnext using {int(mod(\x,6)+1)}] in {1,...,6}{
        \draw (Q\x1) -- (q\x1) -- (q\x) -- (q\x2) -- (Q\x2);
        \draw (q\x) -- (q\xnext);
        \draw (q\x1) -- (q\x2) -- (q\xnext1);
    }
\end{tikzpicture}
        \caption{Moreau--Yosida regularization $(\partial g^*)_\gamma$}\label{fig:subdomainsbloch:my}
    \end{subfigure}
    \caption{Subdomains for radially distributed $\calM$}
    \label{fig:subdomainsbloch}
\end{figure}

\paragraph{Proximal mapping}

Here, we proceed as follows: For each $Q_{i_1\ldots i_k}$, we 
\begin{enumerate}
    \item compute the set $Q_{i_1\ldots i_k}^\gamma\defgl(\Id+\gamma\partial g^*)(Q_{i_1\ldots i_k})$;
    \item  solve for $w\in Q_{i_1\ldots i_k}$ the relation $q\in\{w\}+\gamma\partial g^*(w)$ for arbitrary $q\in Q_{i_1\ldots i_k}^\gamma$.
\end{enumerate}
By \eqref{eq:proxMapRelation}, we then have  $w=\prox_{\gamma g^*}(q)$.
The details are provided in \cref{tab:prox_radial}, while the sets  $Q_{i_1\ldots i_k}^\gamma$ are visualized in \cref{fig:subdomainsbloch:my}.

\begin{table}
    \caption{Computation of proximal map for radially distributed control values ($i+1$ is to be understood modulo $M$)}
    \label{tab:prox_radial}
    \begin{tabular}{llll}
        \toprule
        $Q_{i_1\ldots i_k}$ & $(\Id+\gamma\partial g^*)(w)$  & $Q_{i_1\ldots i_k}^\gamma$               & $(\Id+\gamma\partial g^*)^{-1}(q)$\\
        \midrule
        $Q_0$               & $w$                            & $Q_0$                                    & $q$\\
        $Q_i$               & $w+\gamma \u_i$                 & $Q_i+\gamma \u_i$                         & $q-\gamma \u_i$\\
        $Q_{0,i}$            & $w+\gamma\co\{0,\u_i\}$         & $Q_{0,i}+[0,\gamma]\u_i$                   & $q - \left(\frac{\scalprod{q,\u_i}}{\omega_0^2} - \frac\alpha2\right)\u_i$\\
        $Q_{i,i+1}$         & $w+\gamma\co\{\u_i,\u_{i+1}\}$   & $Q_{i,i+1}+\gamma\co\{\u_i,\u_{i+1}\}$     & $q-\frac{\gamma(\u_i + \u_{i+1})}2-\frac{\scalprod{q,\u_i-\u_{i+1}}(\u_i-\u_{i+1})}{|\u_i-\u_{i+1}|_2^2}$\\
        $Q_{0,i,i+1}$       & $w+\gamma\co\{0,\u_i,\u_{i+1}\}$ & $Q_{0,i,i+1}+\gamma\co\{0,\u_i,\u_{i+1}\}$ & $\alpha\left(\frac{\omega_0}{|\u_i + \u_{i+1}|_2}\right)^2 (\u_i + \u_{i+1})$\\
        \bottomrule
    \end{tabular}
\end{table}

To explain the case $Q_{0,i}$, note that for $q\in Q_{0,i}^\gamma$ we must have by definition of the set $Q_{0,i}^\gamma$ that 
\begin{equation}
    (\Id+\gamma\partial g^*)^{-1}(q)=q-\lambda \u_i \in Q_{0,i}\subset\left\{v\in\R^2:\scalprod{v,\u_i}=\frac\alpha2\omega_0^2\right\}
\end{equation}
for an appropriate choice of $\lambda\in[0,\gamma]$. Thus,
\begin{equation}
    \scalprod{q-\lambda \u_i,\u_i}=\tfrac\alpha2\omega_0^2
    \quad\text{ and so }\quad
    \lambda=\frac{\scalprod{q,\u_i}}{\omega_0^2} - \frac\alpha2\,.
\end{equation}

Likewise, for $q\in Q_{i,i+1}^\gamma$ we must have 
\begin{equation}
    (\Id+\gamma\partial g^*)^{-1}(q)=q-\lambda \u_i-(\gamma-\lambda) \u_{i+1}\in Q_{i,i+1}\subset(\u_i-\u_{i+1})^\perp
\end{equation}
for some $\lambda\in[0,\gamma]$. Thus,
\begin{equation}
    0=\scalprod{q-\lambda \u_i-(\gamma-\lambda) \u_{i+1},\u_i-\u_{i+1}}=\scalprod{q,\u_i-\u_{i+1}}+(\tfrac\gamma2-\lambda)|\u_i-\u_{i+1}|_2^2
\end{equation}
and so
\begin{equation}
    \lambda=\frac\gamma2+\frac{\scalprod{q,\u_i-\u_{i+1}}}{|\u_i-\u_{i+1}|_2^2}\,.
\end{equation}

Finally, note that $Q_{0,i,i+1}=\{\alpha(\frac{\omega_0}{|\u_i + \u_{i+1}|_2})^2 (\u_i + \u_{i+1})\}$ only contains a single element,
which must therefore be equal to $(\Id+\gamma\partial g^*)^{-1}(q)$ for all $q\in Q_{0,i,i+1}^\gamma$.

\paragraph{Moreau--Yosida regularization}

Inserting the above into definition \eqref{eq:my_reg} of the Moreau--Yosida regularization yields
\begin{equation}\label{eq:my_bloch}
    (\partial g^*)_{\gamma}(q) =
    \begin{cases}
        0 & \text{ if } q \in Q_0^{\gamma}, \\
        \u_i & \text{ if } q \in Q_i^{\gamma}, \\
        \left(\tfrac{\scalprod{q, \u_i}}{\gamma \omega_0^2} - \tfrac{\alpha}{2\gamma}\right) \u_i & \text{ if } q \in Q_{0,i}^{\gamma}, \\
        \frac{\u_i + \u_{i+1}}2+\frac{\scalprod{q,\u_i-\u_{i+1}}(\u_i-\u_{i+1})}{\gamma|\u_i-\u_{i+1}|_2^2}  & \text{ if } q \in Q_{i,i+1}^\gamma,\\
        \frac{q}{\gamma} - \frac\alpha{\gamma}\left(\frac{\omega_0}{|\u_i + \u_{i+1}|_2}\right)^2 \left(\u_i + \u_{i+1}\right) &\text{ if } q \in Q_{0,i,i+1}^\gamma.
    \end{cases} 
\end{equation}
Finally, in a numerical implementation it will be necessary to identify efficiently for a given $q\in\R^2$ the set $Q_{i_1\ldots i_k}^\gamma$ in which it is contained.
To this end, determine $i_q,j_q,k_q\in \{1,\ldots,M\}$ via
\begin{equation}
    q\in\overline{C_{i_q}}\,,\qquad
    q-\gamma \u_{i_q}\in\overline{C_{j_q}}\,,\qquad
    q - \left(\frac{\scalprod{q,\u_{i_q}}}{\omega_0^2} - \frac\alpha2\right)\u_{i_q}\in\overline{C_{k_q}}\,,
\end{equation}
and set 
\begin{equation}
    \rho_q:=\scalprod{q,\u_{i_q}}\,,\qquad
    \sigma_q:=\scalprod{q-\tfrac\gamma2(\u_{i_q}+\u_{j_q}),\u_{i_q}+\u_{j_q}}\,.
\end{equation}
Now it is straightforward to identify the correct subdomain via
\begin{align}
    Q_0^\gamma&=\left\{q\in\R^2:\rho_q<\tfrac\alpha2\omega_0^2\right\}\,,\\
    Q_i^\gamma&=\left\{q\in\R^2:\rho_q>(\tfrac\alpha2+\gamma)\omega_0^2,\,i_q=i,\,j_q=i\right\}\,,\\
    Q_{0,i}^\gamma&=\left\{q\in\R^2: \tfrac\alpha2\omega_0^2\leq\rho_q\leq (\tfrac\alpha2+\gamma)\omega_0^2,\,i_q=i,\,k_q=i\right\}\,,\\
    Q_{i,i+1}^\gamma&=\left\{q\in\R^2:\{i,i+1\}=\{i_q,j_q\},\,\sigma_q>\alpha\omega_0^2\right\}\,,\\
    Q_{0,i,i+1}^\gamma&=\left\{q\in\R^2:\{i,i+1\}=\{i_q,i_q+\sign(\u_{i_q}\times q)\},\,k_q\neq i_q,\,\sigma_q\leq\alpha\omega_0^2\right\}\,.
\end{align}

\paragraph{Newton derivative}

Since proximal mappings are Lipschitz continuous and we are in a finite-dimensional setting, a Newton derivative of $h_\gamma:=(\partial g^*)_\gamma$ is given by any choice
\begin{equation}
    D_Nh_\gamma(q) \in \partial_C h_\gamma(q) = \co \left\{\lim_{n\to\infty} \nabla h_\gamma(q_n)\right\},
\end{equation}
where $\partial_C$ denotes Clarke's generalized gradient which admits an explicit characterization by Rademacher's theorem; see, e.g., \cite{Clarke:2013}. We can further use that $h_\gamma$ is continuous and piecewise continuously differentiable and take
\begin{equation}\label{eq:newton_bloch}
    D_N h_{\gamma}(q) = \begin{cases}
        0  & \text{ if } q \in Q_i^{\gamma}, \\
        \frac{1}{\gamma\omega_0^2} \u_i\u_i^T & \text{ if } q \in Q_{0,i}^{\gamma},\\
        \frac{1}{\gamma\abs{\u_i-\u_{i+1}}_2^2}(\u_i - \u_{i+1})(\u_i - \u_{i+1})^T   & \text{ if } q \in Q_{i,i+1}^{\gamma},\\
        \frac{1}{\gamma}\Id & \text{ if }  q \in Q_{0,i,i+1}^{\gamma}. 
    \end{cases}
\end{equation}

\subsection{Concentric corners}\label{sec:multibang:elast}

We now address the case of admissible control values of different magnitudes, where we consider for the sake of an example the concrete set
\begin{equation}\label{eq:M_elast}
    \begin{aligned}
    \calM &= \left\{ \begin{pmatrix} 1 \\ 1 \end{pmatrix}, \begin{pmatrix} 1 \\ -1 \end{pmatrix}, \begin{pmatrix} -1 \\ 1 \end{pmatrix}, \begin{pmatrix} -1 \\ -1 \end{pmatrix}, \begin{pmatrix} 2 \\ 2 \end{pmatrix}, \begin{pmatrix} 2 \\ -2 \end{pmatrix}, \begin{pmatrix} -2 \\ 2 \end{pmatrix}, \begin{pmatrix} -2 \\ -2 \end{pmatrix}\right\} \\[0.5em]
          &= \left\{\u_{1,1}^1, \u_{1,-1}^1, \u_{-1,1}^1, \u_{-1,-1}^1, \u_{1,1}^2, \u_{1,-1}^2, \u_{-1,1}^2, \u_{-1,-1}^2\right\}.
    \end{aligned}
\end{equation}

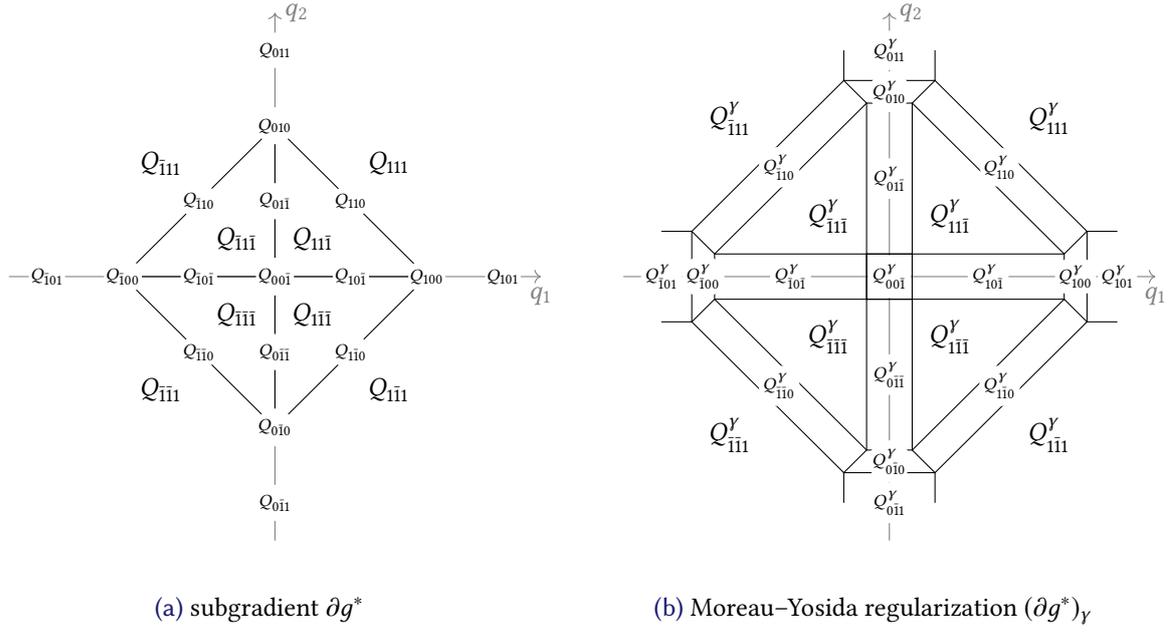
\begin{figure}[t]
    \begin{subfigure}[t]{0.45\linewidth}
        \centering
        \begin{tikzpicture}
    \draw[gray,->](0,-3.5) -- (0,3.5) node[right] {\small $q_2$};
    \draw[gray,->](-3.5,0) -- (3.5,0) node[below] {\small $q_1$};

    \draw (2,0) -- (0,2) -- (-2,0) -- (0,-2) -- cycle;
    \draw(0,-2) -- (0,2);
    \draw(-2,0) -- (2,0);

    \draw (0,0) node[circle,fill=white,inner sep=0pt] {\tiny $Q_{00\bar1}$};
    \draw (0.5,0.5) node[circle,fill=white,inner sep=0pt] {\small$Q_{11\bar1}$};
    \draw (0.5,-0.5) node[circle,fill=white,inner sep=0pt] {\small$Q_{1\bar1\bar1}$};
    \draw (-0.5,0.5) node[circle,fill=white,inner sep=0pt] {\small$Q_{\bar11\bar1}$};
    \draw (-0.5,-0.5) node[circle,fill=white,inner sep=0pt] {\small$Q_{\bar1\bar1\bar1}$};

    \draw (1.5,1.5) node[circle,fill=white,inner sep=0pt] {\small$Q_{111}$};
    \draw (1.5,-1.5) node[circle,fill=white,inner sep=0pt] {\small$Q_{1\bar11}$};
    \draw (-1.5,1.5) node[circle,fill=white,inner sep=0pt] {\small$Q_{\bar111}$};
    \draw (-1.5,-1.5) node[circle,fill=white,inner sep=0pt] {\small$Q_{\bar1\bar11}$};

    \draw (1,0) node[circle,fill=white,inner sep=0pt] {\tiny $Q_{10\bar1}$};
    \draw (0,1) node[circle,fill=white,inner sep=0pt] {\tiny $Q_{01\bar1}$};
    \draw (-1,0) node[circle,fill=white,inner sep=0pt] {\tiny $Q_{\bar10\bar1}$};
    \draw (0,-1) node[circle,fill=white,inner sep=0pt] {\tiny $Q_{0\bar1\bar1}$};

    \draw (3,0) node[circle,fill=white,inner sep=0pt] {\tiny $Q_{101}$};
    \draw (0,3) node[circle,fill=white,inner sep=0pt] {\tiny $Q_{011}$};
    \draw (-3,0) node[circle,fill=white,inner sep=0pt] {\tiny $Q_{\bar101}$};
    \draw (0,-3) node[circle,fill=white,inner sep=0pt] {\tiny $Q_{0\bar11}$};

    \draw (1,1) node[circle,fill=white,inner sep=0pt] {\tiny $Q_{110}$};
    \draw (1,-1) node[circle,fill=white,inner sep=0pt] {\tiny $Q_{1\bar10}$};
    \draw (-1,1) node[circle,fill=white,inner sep=0pt] {\tiny $Q_{\bar110}$};
    \draw (-1,-1) node[circle,fill=white,inner sep=0pt] {\tiny $Q_{\bar1\bar10}$};

    \draw (2,0) node[circle,fill=white,inner sep=0pt] {\tiny $Q_{100}$};
    \draw (0,2) node[circle,fill=white,inner sep=0pt] {\tiny $Q_{010}$};
    \draw (-2,0) node[circle,fill=white,inner sep=0pt] {\tiny $Q_{\bar100}$};
    \draw (0,-2) node[circle,fill=white,inner sep=0pt] {\tiny $Q_{0\bar10}$};
\end{tikzpicture}
        \caption{subgradient $\partial g^*$}\label{fig:subdomainselast:subdiff}
    \end{subfigure}
    \hfill
    \begin{subfigure}[t]{0.45\linewidth}
        \centering
        \begin{tikzpicture}
    \draw[gray,->](0,-3.5) -- (0,3.5) node[right] {\small $q_2$};
    \draw[gray,->](-3.5,0) -- (3.5,0) node[below] {\small $q_1$};

    \draw (0.3,0.3) -- (0.3,-0.3) -- (-0.3,-0.3) -- (-0.3,0.3) -- cycle;
    \draw (0.3,0.3) -- (0.3,2.3) -- (-0.3,2.3) -- (-0.3,0.3) -- cycle;
    \draw (0.3,-0.3) -- (0.3,-2.3) -- (-0.3,-2.3) -- (-0.3,-0.3) -- cycle;
    \draw (0.3,0.3) -- (2.3,0.3) -- (2.3,-0.3) -- (0.3,-0.3) -- cycle;
    \draw (-0.3,-0.3) -- (-2.3,-0.3) -- (-2.3,0.3) -- (-0.3,0.3) -- cycle;

    \draw (0.6,3) -- (0.6,2.6) -- (-0.6,2.6) -- (-0.6,3);
    \draw (0.6,-3) -- (0.6,-2.6) -- (-0.6,-2.6) -- (-0.6,-3);
    \draw (3,0.6) -- (2.6,0.6) -- (2.6,-0.6) -- (3,-0.6);
    \draw (-3,0.6) -- (-2.6,0.6) -- (-2.6,-0.6) -- (-3,-0.6);

    \draw (2.3,0.3) -- (0.3,2.3) -- (0.6,2.6) -- (2.6,0.6) -- cycle;
    \draw (2.3,-0.3) -- (0.3,-2.3) -- (0.6,-2.6) -- (2.6,-0.6) -- cycle;
    \draw (-2.3,0.3) -- (-0.3,2.3) -- (-0.6,2.6) -- (-2.6,0.6) -- cycle;
    \draw (-2.3,-0.3) -- (-0.3,-2.3) -- (-0.6,-2.6) -- (-2.6,-0.6) -- cycle;
    
    \draw (0,0) node[circle,fill=white,inner sep=0pt] {\tiny $Q^\gamma_{00\bar1}$};
    \draw (0.8,0.8) node[circle,fill=white,inner sep=0pt] {\small$Q^\gamma_{11\bar1}$};
    \draw (0.8,-0.8) node[circle,fill=white,inner sep=0pt] {\small$Q^\gamma_{1\bar1\bar1}$};
    \draw (-0.8,0.8) node[circle,fill=white,inner sep=0pt] {\small$Q^\gamma_{\bar11\bar1}$};
    \draw (-0.8,-0.8) node[circle,fill=white,inner sep=0pt] {\small$Q^\gamma_{\bar1\bar1\bar1}$};

    \draw (2.1,2.1) node[circle,fill=white,inner sep=0pt] {\small$Q^\gamma_{111}$};
    \draw (2.1,-2.1) node[circle,fill=white,inner sep=0pt] {\small$Q^\gamma_{1\bar11}$};
    \draw (-2.1,2.1) node[circle,fill=white,inner sep=0pt] {\small$Q^\gamma_{\bar111}$};
    \draw (-2.1,-2.1) node[circle,fill=white,inner sep=0pt] {\small$Q^\gamma_{\bar1\bar11}$};

    \draw (1.3,0) node[circle,fill=white,inner sep=0pt] {\tiny $Q^\gamma_{10\bar1}$};
    \draw (0,1.3) node[circle,fill=white,inner sep=0pt] {\tiny $Q^\gamma_{01\bar1}$};
    \draw (-1.3,0) node[circle,fill=white,inner sep=0pt] {\tiny $Q^\gamma_{\bar10\bar1}$};
    \draw (0,-1.3) node[circle,fill=white,inner sep=0pt] {\tiny $Q^\gamma_{0\bar1\bar1}$};

    \draw (3,0) node[circle,fill=white,inner sep=0pt] {\tiny $Q^\gamma_{101}$};
    \draw (0,3) node[circle,fill=white,inner sep=0pt] {\tiny $Q^\gamma_{011}$};
    \draw (-3,0) node[circle,fill=white,inner sep=0pt] {\tiny $Q^\gamma_{\bar101}$};
    \draw (0,-3) node[circle,fill=white,inner sep=0pt] {\tiny $Q^\gamma_{0\bar11}$};

    \draw (1.45,1.45) node[circle,fill=white,inner sep=0pt] {\tiny $Q^\gamma_{110}$};
    \draw (1.45,-1.45) node[circle,fill=white,inner sep=0pt] {\tiny $Q^\gamma_{1\bar10}$};
    \draw (-1.45,1.45) node[circle,fill=white,inner sep=0pt] {\tiny $Q^\gamma_{\bar110}$};
    \draw (-1.45,-1.45) node[circle,fill=white,inner sep=0pt] {\tiny $Q^\gamma_{\bar1\bar10}$};

    \draw (2.45,0) node[circle,fill=white,inner sep=0pt] {\tiny $Q^\gamma_{100}$};
    \draw (0,2.45) node[circle,fill=white,inner sep=0pt] {\tiny $Q^\gamma_{010}$};
    \draw (-2.45,0) node[circle,fill=white,inner sep=0pt] {\tiny $Q^\gamma_{\bar100}$};
    \draw (0,-2.45) node[circle,fill=white,inner sep=0pt] {\tiny $Q^\gamma_{0\bar10}$};
\end{tikzpicture}
        \caption{Moreau--Yosida regularization $(\partial g^*)_\gamma$}\label{fig:subdomainselast:my}
    \end{subfigure}
    \caption{Subdomains for concentric corners, where $\bar1$ is written for $-1$ to simplify notation (the line dimensions are provided in \cref{fig:subdiff:elast:dimensions})}
    \label{fig:subdiff:elast}
\end{figure}
\begin{figure}[t]
    \begin{subfigure}[t]{0.45\linewidth}
        \centering
        \begin{tikzpicture}
    \draw[gray,->](0,-3.5) -- (0,3.5) node[right] {\small $q_2$};
    \draw[gray,->](-3.5,0) -- (3.5,0) node[below] {\small $q_1$};

    \draw (2,0) -- (0,2) -- (-2,0) -- (0,-2) -- cycle;
    \draw(0,-2) -- (0,2);
    \draw(-2,0) -- (2,0);

    \draw (2,0) node[below] {\small\hspace*{1ex}$3\alpha$};
    \draw (0,2) node[left] {\small$3\alpha$};
    \draw (-2,0) node[below] {\small $-3\alpha$\hspace*{1ex}};
    \draw (0,-2) node[left] {\small $-3\alpha$};
\end{tikzpicture}
        \caption{subgradient $\partial g^*$}\label{fig:dimensionselast:subdiff}
    \end{subfigure}
    \hfill
    \begin{subfigure}[t]{0.45\linewidth}
        \centering
        \begin{tikzpicture}
    \draw (0.3,0.3) -- (0.3,-0.3) -- (-0.3,-0.3) -- (-0.3,0.3) -- cycle;
    \draw (0.3,0.3) -- (0.3,2.3) -- (-0.3,2.3) -- (-0.3,0.3) -- cycle;
    \draw (0.3,-0.3) -- (0.3,-2.3) -- (-0.3,-2.3) -- (-0.3,-0.3) -- cycle;
    \draw (0.3,0.3) -- (2.3,0.3) -- (2.3,-0.3) -- (0.3,-0.3) -- cycle;
    \draw (-0.3,-0.3) -- (-2.3,-0.3) -- (-2.3,0.3) -- (-0.3,0.3) -- cycle;

    \draw (0.6,3) -- (0.6,2.6) -- (-0.6,2.6) -- (-0.6,3);
    \draw (0.6,-3) -- (0.6,-2.6) -- (-0.6,-2.6) -- (-0.6,-3);
    \draw (3,0.6) -- (2.6,0.6) -- (2.6,-0.6) -- (3,-0.6);
    \draw (-3,0.6) -- (-2.6,0.6) -- (-2.6,-0.6) -- (-3,-0.6);

    \draw (2.3,0.3) -- (0.3,2.3) -- (0.6,2.6) -- (2.6,0.6) -- cycle;
    \draw (2.3,-0.3) -- (0.3,-2.3) -- (0.6,-2.6) -- (2.6,-0.6) -- cycle;
    \draw (-2.3,0.3) -- (-0.3,2.3) -- (-0.6,2.6) -- (-2.6,0.6) -- cycle;
    \draw (-2.3,-0.3) -- (-0.3,-2.3) -- (-0.6,-2.6) -- (-2.6,-0.6) -- cycle;
    
    \draw [gray,line width = 3pt] (0,0.3) -- (2.3,0.3) -- (2.6,0.6) -- (3,0.6) node [above] {\small $\eta(q_1)$};
    
    \draw [white,line width = 15pt] (0,-0.6) -- (0.3,-0.6);
    \draw [white,line width = 20pt] (0,-1.2) -- (0.6,-1.2);
    \draw [white,line width = 20pt] (0,-1.8) -- (2.3,-1.8);
    \draw [white,line width = 35pt] (0,-2.4) -- (2.6,-2.4);
    
    \draw [<->] (0,-0.6) -- (0.3,-0.6);
    \draw [<->] (0,-1.2) -- (0.6,-1.2);
    \draw [<->] (0,-1.8) -- (2.3,-1.8);
    \draw [<->] (0,-2.4) -- (2.6,-2.4);

    \draw (0.15,-0.6) node[below]{\small $\gamma$};
    \draw (0.3,-1.2) node[below]{\small $2\gamma$};
    \draw (1.15,-1.8) node[below]{\small $3\alpha+\gamma$};
    \draw (1.3,-2.4) node[below]{\small $3\alpha+2\gamma$};

    \draw [gray,->] (0,-3.5) -- (0,3.5) node[right] {\small $q_2$};
    \draw [gray,->] (-3.5,0) -- (3.5,0) node[below] {\small $q_1$};
\end{tikzpicture}
        \caption{Moreau--Yosida regularization $(\partial g^*)_\gamma$}\label{fig:dimensionselast:my}
    \end{subfigure}
    \centering
    \caption{Dimensions for \cref{fig:subdiff:elast}}
    \label{fig:subdiff:elast:dimensions}
\end{figure}

\paragraph{Fenchel conjugate}
Again inserting $\calM$ into \eqref{eq:conj}, we see that the maximum is either attained by $v=(q_1/|q_1|,q_2/|q_2|)$ or by $v=2(q_1/|q_1|,q_2/|q_2|)$,
where in the case $q_i=0$ we may define $q_i/|q_i|\in\{-1,1\}$ arbitrarily. Hence we obtain after some algebraic manipulations
\begin{equation}
    g^*(q)
    =\max\left\{\abs{q}_1-\alpha,2\abs{q}_1-4\alpha\right\}
    =
    \begin{cases}
        \abs{q}_1 - \alpha & \text{if }\abs{q}_1 \leq 3\alpha,\\
        2\abs{q}_1 - 4\alpha & \text{if }\abs{q}_1 \geq 3\alpha.
    \end{cases}
\end{equation}

\paragraph{Subdifferential}

From \eqref{eq:conj_subdiff}, we directly obtain
\begin{equation}
    \partial g^*(q) = \co\bigcup_{\substack{i\in\{1,2\}: \\ g^*(q) = g_i^*(q)}} \partial g_i^*(q)
    \quad\text{ for }\quad
    \begin{cases}
        g_1^*(q) = \abs{q}_1 - \alpha,\\
        g_2^*(q) = 2\abs{q}_1 - 4\alpha.
    \end{cases}
\end{equation}
In the above we have
\begin{equation}
\partial g_1^*(q) = \begin{pmatrix} \sign(q_1) \\ \sign(q_2)\end{pmatrix},
    \qquad
\partial g_2^*(q) = 2\begin{pmatrix} \sign(q_1) \\ \sign(q_2)\end{pmatrix},
\end{equation}
where $\sign$ denotes the set-valued sign of convex analysis, i.e., $\sign(0)=[-1,1]$.
Therefore we obtain
\begin{equation}
    \partial g^*(q) = \left.\begin{cases}
            \partial g_1^*(q) & \text{if }\abs{q}_{1} < 3\alpha\\
            \partial g_2^*(q) & \text{if }\abs{q}_{1} > 3\alpha\\
            \co\{\partial g_1^*(q),\partial g_2^*(q) \} & \text{if }\abs{q}_{1} = 3\alpha
    \end{cases}\right\}
= \begin{pmatrix} \sign(q_1) \\ \sign(q_2)\end{pmatrix}\cdot
    \begin{cases}
        1 & \text{if }\abs{q}_{1} < 3\alpha,\\
        2 & \text{if }\abs{q}_{1} > 3\alpha,\\
        [1,2] & \text{if }\abs{q}_{1} = 3\alpha.
    \end{cases}
\end{equation}
For an economic notation, let us introduce for $i,j,k\in\{-1,0,1\}$ the sets
\begin{equation}
I_k=k(0,\infty)=\begin{cases}(-\infty,0)&\text{if }k=-1\\\{0\}&\text{if }k=0\\(0,\infty)&\text{if }k=1\\\end{cases}
    \quad\text{ and }
    Q_{ijk}=\{q\in\R^2: q_1\in I_i,\,q_2\in I_j,\,\abs{q}_1-3\alpha\in I_k\}\,.
\end{equation}
A visualization is given in \cref{fig:subdomainselast:subdiff}.
Note that the index $0$ always indicates a lower-dimensional structure, in particular we have
\begin{equation}
    Q_{0jk}\subset\overline{Q_{-1,j,k}}\cap\overline{Q_{1,j,k}}\,,\quad
    Q_{i0k}\subset\overline{Q_{i,-1,k}}\cap\overline{Q_{i,1,k}}\,,\quad
    Q_{ij0}\subset\overline{Q_{i,j,-1}}\cap\overline{Q_{i,j,1}}\,.
\end{equation}
Using this notation, we can write the subdifferential as
\begin{equation}\label{eq:conj_subdiff_concentric}
    \partial g^*(q)=\begin{cases}
        \{\u_{ij}^{(k+3)/2}\}&\text{if }q\in Q_{ijk},\ i,j,k\in\{-1,1\},\\
        \co\big\{\u_{rs}^{(t+3)/2}: r,s,t\in\{-1,1\},\,|r-i|,|s-j|,|t-k|\leq1\big\}&\text{if }q\in Q_{ijk},\ 0\in\{i,j,k\},
    \end{cases}
\end{equation}
which provides more insight into its structure.
In particular, on each lower-dimensional $Q_{ijk}$ the subdifferential is the convex hull of the subdifferentials on the adjacent two-dimensional sets.

\paragraph{Proximal mapping}

To obtain the Moreau--Yosida regularization of $\partial g^*$ for $\gamma>0$, we proceed as above by first noting that $w=(\Id + \gamma\partial g^*)^{-1}(q) \in Q_{ijk}$ holds if and only if
\begin{equation}
    q \in (\Id+\gamma\partial g^*)(Q_{ijk})=: Q_{ijk}^{\gamma}.
\end{equation}
A visualization of these sets is provided in \cref{fig:subdomainselast:my}; we postpone the discussion of these sets to the end of the section and first calculate the specific value of the proximal mapping based on \eqref{eq:proxMapRelation} together with the case distinction in the subdifferential.

Let $w\in Q_{ijk}$ and correspondingly $q\in Q_{ijk}^\gamma$ for some $i,j,k\in\{-1,0,1\}$.
\begin{enumerate}[label=\roman*)]
    \item
        If $i,j,k\in\{-1,1\}$ we have $(\Id+\gamma\partial g^*)(w)=w+\gamma \u_{ij}^{(k+3)/2}$ so that
        \begin{equation}
            (\Id+\gamma\partial g^*)^{-1}(q)=q-\gamma \u_{ij}^{(k+3)/2}\qquad\text{for }q\in Q_{ijk}^\gamma\text{ with }i,j,k\in\{-1,1\}.
        \end{equation}
    \item
        If two of $i,j,k$ are zero, $(\Id+\gamma\partial g^*)^{-1}(q)$ must be the single unique element of $Q_{ijk}$, thus
        \begin{equation}
            (\Id+\gamma\partial g^*)^{-1}(q)=\begin{cases}
                0&\text{if }q\in Q_{0,0,-1}^\gamma,\\
                3\alpha(i,j)&\text{if }q\in Q_{i,j,0}^\gamma\text{ with }i=0\text{ or }j=0.
            \end{cases}
        \end{equation}
    \item
        If $i=0$ and $j,k\neq0$, then for $w\in Q_{0jk}$ we have 
        \begin{equation}
            (\Id+\gamma\partial g^*)(w)=w+\gamma\co\left\{\u_{-1,j}^{(k+3)/2},\u_{1,j}^{(k+3)/2}\right\}=w+\gamma\frac{k+3}2([-1,1],j).
        \end{equation}
        Thus for $q\in Q_{0jk}^\gamma$ we have $(\Id+\gamma\partial g^*)^{-1}(q)=q-\gamma\frac{k+3}2(\lambda,j)$, where $\lambda\in[-1,1]$ is such that $q-\gamma\frac{k+3}2(\lambda,j)\in Q_{0jk}\subset\{0\}\times\R$.
        Therefore $\lambda=\frac2{\gamma(k+3)}q_1$, and
        \begin{equation}
            (\Id+\gamma\partial g^*)^{-1}(q)=(0,q_2-\gamma\tfrac{k+3}2j)\qquad\text{for }q\in Q_{0jk}^\gamma\text{ with }j,k\in\{-1,1\}.
        \end{equation}
        Analogously,
        \begin{equation}
            (\Id+\gamma\partial g^*)^{-1}(q)=(q_1-\gamma\tfrac{k+3}2i,0)\qquad\text{for }q\in Q_{i0k}^\gamma\text{ with }i,k\in\{-1,1\}.
        \end{equation}
    \item
        If $k=0$ and $i,j\neq0$, then for $w\in Q_{ij0}$ we have
        \begin{equation}
            (\Id+\gamma\partial g^*)(w)=w+\gamma\co\{\u_{ij}^1,\u_{ij}^2\}=w+\gamma[1,2](i,j).
        \end{equation}
        Thus for $q\in Q_{ij0}^\gamma$ we have $(\Id+\gamma\partial g^*)^{-1}(q)=q-\gamma\lambda(i,j)$, where $\lambda\in[1,2]$ is such that $q-\gamma\lambda(i,j)\in Q_{ij0}\subset\{w\in\R^2:\abs{w}_1=3\alpha\}$.
        Therefore $\lambda=\frac{\abs{q}_1-3\alpha}{2\gamma}$, and
        \begin{equation}
            (\Id+\gamma\partial g^*)^{-1}(q)=q-\tfrac{\abs{q}_1-3\alpha}{2}(i,j)\qquad\text{for }q\in Q_{ij0}^\gamma\text{ with }i,j\in\{-1,1\}.
        \end{equation}
\end{enumerate}

It remains to discuss the sets $Q_{ijk}^\gamma$. 
Rather than list all sets explicitly, we instead provide a procedure for determining for a given $q\in\R^2$ the corresponding subdomain, which is what is actually required for the numerical implementation.
For that purpose, let us introduce the function (compare the illustration in \cref{fig:subdiff:elast:dimensions})
\begin{equation}
\eta(x)=\begin{cases}\gamma&\text{if }x<3\alpha+\gamma,\\x-3\alpha&\text{if }3\alpha+\gamma\leq x\leq3\alpha+2\gamma,\\2\gamma&\text{if }x>3\alpha+2\gamma.\end{cases}
\end{equation}
With this function we have $q\in Q_{ijk}^\gamma$ for $i,j,k$ given by
\begin{align}
i&=\begin{cases}0&\text{if }|q_1|\leq\eta(|q_2|),\\\sign(q_1)&\text{else,}\end{cases}\\[1em]
j&=\begin{cases}0&\text{if }|q_2|\leq\eta(|q_1|),\\\sign(q_2)&\text{else,}\end{cases}\\[1em]
    k&=\begin{cases}
        -1&\text{if }\abs{q}_\infty<3\alpha+\gamma\text{ and }\abs{q}_1<3\alpha+2\gamma,\\
        1&\text{if }\abs{q}_\infty>3\alpha+2\gamma\text{ or }\abs{q}_1>3\alpha+4\gamma,\\
    0&\text{else.}\end{cases}
\end{align}

\paragraph{Moreau--Yosida regularization}

Inserting this into the definition \eqref{eq:my_reg} of the Moreau--Yosida regularization yields
\begin{equation}\label{eq:myreg_elast}
    (\partial g^*)_\gamma(q)=\begin{cases}
        \u_{ij}^{(k+3)/2} & \text{if }q\in Q_{ijk}^\gamma\text{ with }i,j,k\neq 0,\\
        \frac1\gamma(q-3\alpha(i,j)) & \text{if }q\in Q_{ijk}^\gamma\text{ with } |i|+|j|+|k| = 1,\\
        (\frac1\gamma q_1,\frac{k+3}2j) & \text{if }q\in Q_{0jk}^\gamma\text{ with }j,k\neq 0,\\
        (\frac{k+3}2i,\frac1\gamma q_2) & \text{if }q\in Q_{i0k}^\gamma\text{ with }i,k\neq 0,\\
        \frac{\abs{q}_1-3\alpha}{2\gamma}(i,j) & \text{if }q\in Q_{ij0}^\gamma\text{ with }i,j\neq 0.
    \end{cases}
\end{equation}

\paragraph{Newton derivative}

Finally, we can again take as a Newton derivative any element of the Clarke gradient; here, we choose
\begin{equation}\label{eq:newton_elast}
    D_Nh_{\gamma}(q) = \begin{cases}
        0  & \text{ if } q\in Q_{ijk}^\gamma\text{ with }i,j,k\neq 0, \\
        \frac{1}{\gamma} \Id  & \text{ if } q\in Q_{ijk}^\gamma\text{ with } |i|+|j|+|k| = 1,\\

        \frac{1}{\gamma}(j,i)^T(j,i) 
        & \text{ if }q\in Q_{ijk}^\gamma\text{ with } |i| + |j| = 1,k\neq 0,\\

        \frac{1}{2\gamma}(i,j)^T(i,j) & \text{ if } q\in Q_{ij0}^\gamma\text{ with }i,j\neq 0. \\
    \end{cases}
\end{equation}

\section{State equation}\label{sec:stateEq}

In this section, we specify in more detail our model state operators and verify that the assumptions \ref{enm:weakWeakContinuity}--\ref{enm:adjointConvergence} of \cref{sec:existence,sec:stability} are satisfied for our model problems.

\subsection{Bloch equation}\label{sec:BlochStateEq}

As our motivating model problem, we consider the Bloch equation in a rotating reference frame without relaxation
\begin{equation}
    \frac{\dd}{\dd t}{\mathbf{M}}^{(\omega)}(t) = \mathbf{M}^{(\omega)}(t) \times \mathbf B^{(\omega)}(t)\,,
    \qquad\mathbf{M}^{(\omega)}(0)=(0,0,1)^T,
\end{equation}
which describes the temporally evolving magnetization $\mathbf M^{(\omega)}\in\R^3$ of an ensemble of spins rotating at the same resonance offset frequency $\omega$ (called \emph{isochromat}),
starting from a given equilibrium magnetization.
The time-varying effective magnetic field $\mathbf B^{(\omega)}(t)$ is of the form
\begin{equation}
    \mathbf B^{(\omega)}(t)=(\omega_x(t),\omega_y(t),\omega)^T\,,
\end{equation}
where $u(t):=(\omega_x(t),\omega_y(t))\in \R^2$ can be controlled.
The aim is to achieve a magnetization $\mathbf M^{(\omega)}(T)=\mathbf M_d$ within the time interval $\Omega=[0,T]$ for a list of offset frequencies $\omega_1,\ldots,\omega_J$.
In terms of our previous notation we thus set
\begin{equation}\label{eq:settingBloch}
    S:L^2(\Omega;\R^2) \to (\R^3)^J\qquad u\mapsto\left[\mathbf M^{(\omega_1)}(T),\ldots,\mathbf M^{(\omega_J)}(T)\right]\,.
\end{equation}
This choice of $S$ satisfies the assumptions \ref{enm:weakWeakContinuity}--\ref{enm:adjointConvergence}; see \cref{sec:BlochProperties}.

\subsection{Linear elasticity}\label{sec:elasticityStateEq}

In this case, $\Omega\subset\R^2$ represents an elastic body fixed at $\Gamma\subset\partial\Omega$ (with positive Hausdorff measure $\mathcal{H}^1(\Gamma)>0$),
where we assume $\Gamma$ and $\partial\Omega\setminus\Gamma$ to be smooth or $\Omega$ to be a convex polygon with $\Gamma$ being the union of some faces.
The elastic body is subject to a controlled body force $u:\Omega\to\R^2$.
The resulting displacement $y:\Omega\to\R^2$ is governed by the equations of linearized elasticity with Lamé parameters $\mu$ and $\lambda$,
\begin{equation}\label{eq:lame}
    \left\{\begin{aligned}
            -2\mu \divergence \epsilon(y) - \lambda \grad \divergence y &= u \text{  in } \Omega, \\ 
            y &= 0 \text{  on }  \Gamma, \\
            (2\mu \epsilon(y) + \lambda \divergence y) n &= 0 \text{  on }  \partial\Omega\setminus\Gamma, 
    \end{aligned}\right.
\end{equation}
where $n$ denotes the unit outward normal, $Dy=[\nabla y_1|\nabla y_2]^T$ is the displacement gradient, and $\epsilon(y)=\frac{Dy+Dy^T}2$ is the symmetrized gradient.
Defining
\begin{equation}
    H^1_\Gamma(\Omega) \defgl \left\{v \in H^1(\Omega;\R^2): v = 0 \text{ on } \Gamma \right\},
\end{equation}
we may take
\begin{equation}\label{eq:settingElasticity}
    S: H_{\Gamma}^1(\Omega)^* \to H_{\Gamma}^1(\Omega),\qquad u\mapsto y\text{ solving \eqref{eq:lame}}\,.
\end{equation}

The solution operator $S$ of the linear elasticity problem is well-known to be a bounded linear operator from $U=L^2(\Omega;\R^2)$ into $H^1_\Gamma(\Omega)\hookrightarrow L^2(\Omega;\R^2)=:Y$, see, e.g., \cite{Braess:2007}.
This immediately implies weak-to-weak continuity and Fréchet differentiability with $S'(u)=S$ for all $u\in U$.
Similarly, $S'(u)^*=S^*$ for all $u\in U$, and it is readily checked that actually $S$ is self-adjoint so that $S^*=S$.
As a consequence we have $\ran S'(u)^*=\ran S\hookrightarrow L^\infty(\Omega;\R^2)$.
Indeed, in case of polygonal domains $\Omega$ this follows from $\ran S\subset H^{3/2}(\Omega;\R^2)$ by \cite[Thm.~2.3]{Ni92},
and in the case of piecewise smooth domains with smooth traction boundary it follows from $\ran S\subset H^{2}(\Omega;\R^2)$ by \cite[Thm.~8]{MaNi10}.
Summarizing, this choice of $S$ satisfies assumptions \ref{enm:weakWeakContinuity}--\ref{enm:linearity}.

\section{Numerical solution}\label{sec:semiSmoothNewton}

We now discuss the numerical solution of the regularized system \eqref{eq:regsystem} via a semismooth Newton method.

\subsection{Bloch equation}

As is usual for time-dependent state equations, we avoid a full space-time discretization by following a reduced approach, i.e., we consider in place of \eqref{eq:regsystem} the equation
\begin{equation}\label{eq:my_sys_bloch}
    u_\gamma - H_\gamma(-\calF'(u_\gamma)) = 0.
\end{equation}
Recall that $H_\gamma$ is a superposition operator defined via 
\begin{equation}
    [H_\gamma(p)](x) = h_\gamma (p(x)) \qquad \text{for a.e. }x\in \Omega,
\end{equation}
with $h_\gamma = (\partial g^*)_\gamma$ given by \eqref{eq:my_bloch}.
By \cref{prop:blochAdjointContinuity}, we have $-\calF'(u_\gamma)=S'(u_\gamma)^*(z-S(u_\gamma))\in L^\infty(\Omega;\R^2)$, and hence we can consider $H_\gamma:L^r(\Omega;\R^2)\to L^2(\Omega;\R^2)$ for any $r>2$. Since $h_\gamma$ is Lipschitz continuous and piecewise differentiable, semismoothness of $H_\gamma$ follows from \cite[Thm.~3.49]{Ulbrich:2011} with Newton derivative given by
\begin{equation}
    [D_N H_\gamma(p) h](x) = D_Nh_\gamma(p(x))h(x) \qquad \text{for a.e. }x\in \Omega
\end{equation}
and $D_Nh_\gamma$ defined in \eqref{eq:newton_bloch}.

Further, note that $S$ is twice continuously differentiable.
Indeed, this follows by an analogous argument as for Fréchet differentiability in the proof of \cref{thm:operatorProperties}:
Using the same notation, the second derivative applied to test directions $\varphi,\psi\in L^2(\Omega;\R^2)$ will be given by $S''(u)(\varphi,\psi)=\mathbf{W}(T)=(\mathbf{W}^1(T),\ldots,\mathbf{W}^J(T))$ with
\begin{equation}
    \left\{\begin{aligned}
            \tfrac{\dd}{\dd t} \mathbf{W}^j(t)&=B_{u}^{\omega_j}(t)\mathbf{W}^j(t)+B_\varphi^0(t)\delta\mathbf M_\psi^{(\omega_j)}(t)+B_\psi^0(t)\delta\mathbf M_\varphi^{(\omega_j)}(t)\,,\qquad t\in[0,T],\\
            \mathbf{W}^j(0)&=0\,,
    \end{aligned}\right.
\end{equation}
where $S'(u)(\varphi)=(\delta\mathbf M_\varphi^{(\omega_1)}(T),\ldots,\delta\mathbf M_\varphi^{(\omega_J)}(T))$ with $\delta\mathbf M_\varphi^{(\omega)}$ satisfying \eqref{eq:BlochFrechetEquation}.
This equation has exactly the same structure as \eqref{eq:BlochFrechetEquation}, and thus the argument for showing 
\begin{equation}
    |S'(\tilde u)(\varphi)-S'(u)(\varphi)-S''(u)(\tilde u-u,\varphi)|_2=\|\varphi\|_{L^2(\Omega;\R^2)}O(\|\tilde u-u\|_{L^2(\Omega;\R^2)}^2)
\end{equation}
works analogously.
Since $S$ is twice continuously differentiable, we can apply the chain rule, e.g., from \cite[Thm.~3.69]{Ulbrich:2011} to obtain
\begin{equation}
    D_N (H_\gamma \circ (-\calF'))(u)\phi = - D_N H_\gamma(-\calF'(u))\calF''(u)\phi
\end{equation}
for any $\phi\in L^2(\Omega;\R^2)$. 
A semismooth Newton step is thus given by $u^{k+1} = u^k +\delta u$, where $\delta u$ is the solution to
\begin{equation}\label{eq:ssn_bloch}
    \left(\Id +  D_N H_\gamma(-\calF'(u^k))\calF''(u^k)\right)\delta u = -u^k+ H_\gamma(-\calF'(u^k))\,,
\end{equation}
which can be obtained, e.g., using a matrix-free Krylov subspace method such as GMRES.

Recall that following \cref{prop:bloch_adjoint} and \cite{Aigner:2015}, $p=-\calF'(u)$ can be evaluated by solving the adjoint equations
\begin{equation}
    \label{eq:adjoint}
    \left\{\begin{aligned}
            -\tfrac{\dd}{dt}{\mathbf{P}}^{(\omega_j)}(t)&= B^\omega_u(t) \mathbf{P}^{(\omega_j)}(t),\qquad t\in[0,T],\\
            \mathbf{P}^{(\omega_j)}(T)&=\mathbf{M}^{(\omega_j)}_u(T)-(\mathbf{M}_d)_j,
    \end{aligned}\right.
\end{equation}
for $j=1,\dots,J$ and setting
\begin{equation}
    \label{eq:gradient}
    p(t) = \sum_{j=1}^J\begin{pmatrix} 
        \big(\mathbf M_u^{(\omega_j)}(t)\big)_3{\mathbf{P}^{(\omega_j)}_2}(t) - \big(\mathbf M_u^{(\omega_j)}(t)\big)_2{\mathbf{P}^{(\omega_j)}_3}(t)\\
        \big(\mathbf M_u^{(\omega_j)}(t)\big)_3{\mathbf{P}^{(\omega_j)}_1}(t) - \big(\mathbf M_u^{(\omega_j)}(t)\big)_1{\mathbf{P}^{(\omega_j)}_3}(t)
    \end{pmatrix}
    =\sum_{j=1}^J
    \begin{pmatrix}
        \mathbf{M}^{(\omega_j)}_u(t)^T{B_1} {\mathbf{P}^{(\omega_j)}}(t)\\
        \mathbf{M}^{(\omega_j)}_u(t)^T{B_2} {\mathbf{P}^{(\omega_j)}}(t)
    \end{pmatrix}
\end{equation}
for $t\in[0,T]$, where for the sake of brevity, we have set
\begin{equation}
{B_1} :=\begin{pmatrix} 0 & 0 & 0 \\ 0 & 0 & -1 \\ 0 & 1 &0\end{pmatrix},\qquad
{B_2} :=\begin{pmatrix} 0 & 0 & -1 \\ 0 & 0 & 0 \\ 1 & 0 &0\end{pmatrix}.
\end{equation}
Similarly, the application of $\calF''(u)\phi$ for given $u,\phi\in L^2(\Omega;\R^2)$ is given by 
\begin{equation}
    \calF''(u)\phi=\sum_{j=1}^J\begin{pmatrix}
        \delta \mathbf{M}_\phi^{(\omega_j)}(t)^T{B_1} {\mathbf{P}^{(\omega_j)}}(t) +  \mathbf{M}_u^{(\omega_j)}(t)^T{B_1}\delta {\mathbf{P}^{(\omega_j)}}(t)\\
        \delta \mathbf{M}_\phi^{(\omega_j)}(t)^T{B_2} {\mathbf{P}^{(\omega_j)}}(t) +  \mathbf{M}_u^{(\omega_j)}(t)^T{B_2}\delta {\mathbf{P}^{(\omega_j)}}(t)
    \end{pmatrix},
\end{equation}
where $\delta \mathbf{M}_\phi^{(\omega)}$ (the directional derivative of $\mathbf{M}^{(\omega)}$ with respect to ${u}$) is given by the solution of the linearized state equation \eqref{eq:BlochFrechetEquation}
and $\delta {\mathbf{P}^{(\omega)}}$ (the directional derivative of ${\mathbf{P}^{(\omega)}}$ with respect to ${u}$) is given by the solution of the linearized adjoint equation
\begin{equation}
    \label{eq:adjoint_lin}
    \left\{\begin{aligned}
            -\tfrac{\dd}{\dd t}{\delta {\mathbf{P}^{(\omega)}}}(t) &= {B_u^\omega}(t) {\delta {\mathbf{P}^{(\omega)}}}(t)+ B_\phi^0(t) {\mathbf{P}^{(\omega)}}(t),\qquad t\in[0,T],\\
            \delta {\mathbf{P}^{(\omega)}}(T) &= \delta \mathbf{M}_\phi^{(\omega)}(T).
    \end{aligned}\right.
\end{equation}
This characterization can be derived using formal Lagrangian calculus and rigorously justified using the implicit function theorem; see, e.g., \cite[Chapter 1.6]{Hinze2009}.

\bigskip

Since the forward operator $S$ is nonlinear, the problem \eqref{eq:problem_reg} is nonconvex. Hence, convergence of the semismooth Newton method \eqref{eq:ssn_bloch} to a minimizer $u_\gamma$ requires a second-order sufficient (local quadratic growth) condition at $u_\gamma$ for $\gamma>0$ small, which is difficult to verify. Furthermore, we need to deal with the fact that Newton methods converge only locally, with the convergence region shrinking with $\gamma$. For this reason, we perform a continuation in $\gamma$, i.e., we solve \eqref{eq:regsystem} for a sequence $\gamma_1>\gamma_2>\ldots$ of regularization parameters, each time using the result for $\gamma_n$ as initialization for the iteration with $\gamma_{n+1}$. In addition, we include in each step of the semismooth Newton method a line search for $\delta u$ based on the residual norm of the reduced optimality condition \eqref{eq:my_sys_bloch}.
While globalization of nonsmooth Newton methods is a delicate issue that we do not want to address in this work, we remark that this heuristic approach seems to work well in practice.

\bigskip

We finally address the discretization of \eqref{eq:ssn_bloch}. The Bloch equation is discretized using a Crank--Nicolson method, where the states $\mathbf{M}^{(\omega)}$ are discretized as continuous piecewise linear functions with values $\mathbf{M}_m^{(\omega)}:=\mathbf{M}^{(\omega)}(t_m)$ for discrete time points $t_1,\ldots,t_{N_u}$, and the control ${u}$ is treated as a piecewise constant function, i.e., ${u}=\sum_{m=1}^{N_u} {u_m} \chi_{(t_{m-1},t_m]}(t)$, where $\chi_{(a,b]}$ is the characteristic function of the half-open interval~$(a,b]$. To obtain a consistent scheme, where discretization and optimization commute, the adjoint state $\mathbf{P}^{(\omega)}$ in \eqref{eq:adjoint} is discretized as piecewise constant using an appropriate time-stepping scheme \cite{BecMeiVex_07}, and the linearized state $\delta\mathbf{M}^{(\omega)}$ and the linearized adjoint state $\delta\mathbf{P}^{(\omega)}$ are discretized in the same way as the state and adjoint state, respectively; see~\cite{Aigner:2015}.

\subsection{Linearized elasticity}

For the case of linearized elasticity, we can proceed exactly as in  \cite{CIK:2014,CK:2013}. First, note that due to the embedding $H_{\Gamma}^1(\Omega) \hookrightarrow L^p(\Omega;\R^2)$ for $p>2$, the superposition operator $H_{\gamma}$ (for $h_\gamma:=(\partial g^*)_\gamma$ now given by \eqref{eq:myreg_elast}) is again semismooth with Newton derivative $D_NH_\gamma$ (for $D_N h_\gamma$ now given by \eqref{eq:newton_elast}).

To obtain a symmetric Newton system, we reduce \eqref{eq:regsystem} to the state $y_\gamma=S(u_\gamma)$ and the dual variable $p_\gamma$. Since $S$ is a bounded linear operator, we have $S'(u) = S$ and therefore by definition of $S$ obtain
\begin{equation}
    \label{eq:redregsystem}
    \left\{\begin{aligned}
            A^*p_{\gamma} &= z - y_{\gamma}, \\
            Ay_{\gamma} &= H_{\gamma}(p_{\gamma}),
    \end{aligned} \right.
\end{equation}
where $A$ denotes the elliptic linear differential operator arising from the system \eqref{eq:lame} of linearized elasticity.
Consequently, we consider
\begin{equation}
F(y,p) \defgl \begin{pmatrix} y - z +  A^*p \\ Ay - H_{\gamma}(p) \end{pmatrix} = \begin{pmatrix} 0 \\ 0 \end{pmatrix},
    \label{eq:Newtonfun}
\end{equation}
where $F:Y\times U^*\to Y\times U$. Since the regularized optimal state $y_\gamma$ and the adjoint state $p_\gamma$ are in $H^1_\Gamma(\Omega)$, we may consider $F:H^1_\Gamma(\Omega)\times H^1_\Gamma(\Omega)\to H^1_\Gamma(\Omega)^*\times H^1_\Gamma(\Omega)^*$.
For a semismooth Newton step, we obtain $(\delta y, \delta p)$ by solving
\begin{equation}
    \begin{pmatrix}
        \Id &  A^* \\
        A & -D_NH_{\gamma}(p^k)
    \end{pmatrix}
    \begin{pmatrix}
        \delta y \\
        \delta p
    \end{pmatrix} 
    = 
    \begin{pmatrix}
        z - y^k -A^*p^k \\
        -Ay^k + H_{\gamma}(p^k)
    \end{pmatrix}
    \label{eq:Newtonstep}
\end{equation}   
for given $(y^k,p^k)$, and we set $y^{k+1} = y^k + \delta y$ and $p^{k+1} = p^k + \delta p$.

\bigskip

Due to the linearity of the state equation (and hence convexity of the problem), the convergence of the semismooth Newton method for every $\gamma>0$ to a minimizer of \eqref{eq:problem_reg} can be shown exactly as in \cite{CIK:2014,CK:2013}. As in the case of the Bloch equation, we  include a continuation in $\gamma$ as well as a line search based on the residual norm in \eqref{eq:Newtonfun}.

\bigskip

For the discretization, we consider \eqref{eq:Newtonfun} in its weak form
\begin{equation}\label{eq:elasticReducedWeakSystem}
    \begin{pmatrix}
        \int_\Omega2\mu\epsilon(p):\epsilon(\phi)+\lambda\divergence(p)\,\divergence\phi+(y-z)\phi\,\dd x\\
        \int_\Omega2\mu\epsilon(y):\epsilon(\psi)+\lambda\divergence y\,\divergence\psi-h_\gamma(p)\psi\,\dd x
    \end{pmatrix} = 
    \begin{pmatrix} 
        0 \\ 0 
    \end{pmatrix}
    \qquad\text{for all }\phi,\psi\in H^1_\Gamma(\Omega).
\end{equation}
We now discretize the state $y$, the adjoint state $p$, and the test functions $\phi_h,\psi_h$ using piecewise linear finite element functions $y_h,p_h,\phi_h,\psi_h\in V_h$,
where $V_h\subset H^1_\Gamma(\Omega)$ denotes the space of piecewise linear, $\R^2$-valued functions on a uniform triangulation of $\Omega$.
Analogously to \cite{CIK:2014,CK:2013}, we employ exact quadrature for all terms
except for $\int_\Omega h_\gamma(p_h)\psi_h\,\dd x$, which we approximate by $\int_\Omega I_h(h_\gamma(p_h))\psi_h\,\dd x$ for the piecewise linear nodal interpolation operator $I_h$.
Thus, letting $\phi_1,\ldots,\phi_{N_h}$ denote a nodal basis of $V_h$ and introducing the mass and stiffness matrices
\begin{equation}
    M_h=\left(\int_\Omega\phi_i\cdot \phi_j\,\dd x\right)_{ij},\quad
    L_h=\left(\int_\Omega\epsilon(\phi_i):\epsilon(\phi_j)\,\dd x\right)_{ij},\quad
    K_h=\left(\int_\Omega\divergence\phi_i\cdot \divergence\phi_j\,\dd x\right)_{ij},
\end{equation}
as well as $A_h=2\mu L_h+\lambda K_h$ and the vector $Z_h=\left(\int_\Omega z\cdot\phi_1\,\dd x,\ldots,\int_\Omega z\cdot\phi_{N_h}\,\dd x\right)^T$, the discrete version of \eqref{eq:elasticReducedWeakSystem} reads
\begin{equation}
    \begin{pmatrix}
        A_h^T\mathbf p+M_h\mathbf y-Z_h\\
        A_h\mathbf y-M_hh_\gamma(\mathbf p)
    \end{pmatrix} 
    = 
    \begin{pmatrix} 
        0 \\ 0 
    \end{pmatrix},
\end{equation}
and \eqref{eq:Newtonstep} becomes
\begin{equation}
    \begin{pmatrix}
        M_h &  A_h^T \\
        A_h & -M_hD_Nh_{\gamma}(\mathbf p^k)
    \end{pmatrix}
    \begin{pmatrix}
        \delta\mathbf y \\
        \delta\mathbf p
    \end{pmatrix} 
    = 
    \begin{pmatrix}
        Z_h -\mathbf y^k -A_h^T\mathbf p^k \\
        -A_h\mathbf y^k + M_hh_{\gamma}(\mathbf p^k)
    \end{pmatrix}
\end{equation}
where $\mathbf y=(y_i)_i$ and $\mathbf p=(p_i)_i$ are the nodal values of $y_h$ and $p_h$, and where $h_\gamma(\mathbf p)=(h_\gamma(p_i))_i$ and $D_Nh_\gamma(\mathbf p)=(D_Nh_\gamma(p_i)\delta_{ij})_{ij}$.

\section{Numerical examples}\label{sec:examples}

We illustrate the proposed approach for the two model problems described in \cref{sec:stateEq} and the two specific multibang penalties described in \cref{sec:penalty}.
The Matlab code used to generate these examples can be downloaded from \url{http://github.com/clason/vectormultibang}. 

\subsection{Bloch equation}

The first example is based on the optimal excitation of isochromats in nuclear magnetic resonance imaging \cite{Spincontrol:15}, where the aim is to shift the magnetization vector $\mathbf M$ at time $T$ from initial alignment with a strong external magnetic field, i.e., $\mathbf M(0) = (0,0,1)^T$, to the saturated state $\mathbf M_d=(1,0,0)^T$ using a radiofrequency pulse $u(t)=(\omega_x(t),\omega_y(t))^T$. To follow the physical setup, we scale the controls as $u(t) = \bar \gamma B_1 \tilde u(t)$, where $\bar \gamma \approx 267.51$ is the gyromagnetic ratio (in MHz per Tesla) and $B_1 = 10^{-2}$ is the strength of the modulated magnetic field (in milli-Tesla); the figures always show the unscaled control $\tilde u$. The control cost parameter (which in this setting can be interpreted as a penalty on the specific absorption rate of the radio energy) is set to $\alpha=10^{-1}$.
In all examples, the Bloch equation is discretized with $N_u=1000$ time intervals; the implementation of the discrete (linearized) Bloch and adjoint equations is taken from \cite{rfcontrol}.
The semismooth Newton iteration is then applied and terminated if the relative or absolute norm of the residual in the optimality condition drops below $10^{-7}$ or if $500$ iterations are exceeded.  The Newton step is solved via GMRES without restarts and without preconditioning, which is terminated if the relative residual drops below $10^{-10}$ or if $1000$ iterations are exceeded. 
The continuation in the Moreau--Yosida regularization is started with $\gamma_0=10^2$ and reduced by a factor of $1/2$ until $\gamma_{\min}=10^{-10}$ is reached or the semismooth Newton iteration fails to convergence.
We remark that in a practical implementation, these strict fixed tolerances should be replaced as in inexact Newton methods by adaptive criteria based on residuals in the outer loops.

We begin with a single isochromat with $\omega = 10^{-2}\bar\gamma$. \Cref{fig:Bloch1} shows the resulting optimal control $\tilde u$ and magnetization evolution $\mathbf{M}^{(\omega)}(t)$ for $M=3$ equally spaced radially distributed desired control values with magnitude $\omega_0=1$ and phases
$\theta_1=-\pi$, $\theta_2=-\pi/3$, $\theta_3=\pi/3$, which are marked by colored dashed lines.
At any time $t\in[0,T]$, the optimal control $\tilde u(t)=(\omega_x(t),\omega_y(t))$ can be seen to only take values from $\calM$ as desired.
(For an easier visual comprehension, $\tilde u(t)$ is plotted as a continuous curve so that a jump from one value in $\calM$ to another is shown as a connecting line.)
Indeed, most of the time we have $\tilde u=\bar u_0=0$, periodically intermitted by short time intervals where $\tilde u$ takes the values $\bar u_1,\bar u_2,\bar u_3\in\calM$ in a periodically rotating order.
Each of these time intervals coincides in the state trajectory with a change in $M_z$, while the $M_z$ component of $\mathbf M^{(\omega)}$ stays constant during $\tilde u=0$.
The final magnetization $\mathbf M^{(\omega)}(T)$ shows a very close attainment of the target $\mathbf M_d$. The situation is very similar for $M=6$ with $\omega_0=1$ and $\theta \in \{-\pi,-2\pi/3,-\pi/3,0,\pi/3,2\pi/3\}$, see \cref{fig:Bloch2}. In both cases, all nonzero desired control values are made use of equally.

We now consider the simultaneous control of $J=4$ isochromats with $\omega = 10^{-2}\bar\gamma\cdot(1,2,3,4)$. \Cref{fig:Bloch3} shows the result if the same target $\mathbf M_d=(1,0,0)^T$ is specified for all isochromats. Again, we have very close attainment of the target,
and again the control is zero most of the time, intermitted by regularly spaced intervals in which nonzero control values from $\calM$ are used.
This time, not all nonzero values from $\calM$ occur, but just $\bar u_2$ and $\bar u_3$ (indicated by the red and turquois dashed line).
In addition there are five time points at which control values outside $\calM$ are adopted, visible in the graph as short spikes emanating from $\tilde u=0$.
(Note, though, that these values still show the desired angles, merely at smaller than desired magnitudes.) 
This may be due to the fact that in this example, the Newton method failed to converge already for $\gamma<2\cdot 10^{-6}$. 
In the more realistic case where only a single isochromat -- in this case $j=3$ -- is supposed to be excited (i.e., $\mathbf M_d=(1,0,0)^T$ for $\mathbf M^{(\omega_3)}$ and $\mathbf M_d=(0,0,1)^T$ else), we again obtain a pure multibang control (see \cref{fig:Bloch4}).

\definecolor{mycolor1}{rgb}{0.00000,0.44700,0.74100}%
\definecolor{mycolor2}{rgb}{0.85000,0.32500,0.09800}%
\definecolor{mycolor3}{rgb}{0.92900,0.69400,0.12500}%
\definecolor{mycolor4}{rgb}{0.49400,0.18400,0.55600}%
\definecolor{mycolor5}{rgb}{0.46600,0.67400,0.18800}%
\definecolor{mycolor6}{rgb}{0.30100,0.74500,0.93300}%
\definecolor{mycolor7}{rgb}{0.63500,0.07800,0.18400}%

\begin{figure}[p]
    \centering
    \begin{subfigure}[t]{0.45\linewidth}
        \centering
        \begin{tikzpicture}[trim axis left]

\begin{axis}[%
width=\linewidth,
scale only axis,
xmin=0,
xmax=8,
tick align=outside,
xlabel={$t$},
ymin=-1,
ymax=1,
ylabel={$\omega_x$},
zmin=-1,
zmax=1,
zlabel={$\omega_y$},
view={-37.5}{30},
axis background/.style={fill=white},
axis x line*=bottom,
axis y line*=left,
axis z line*=left,
xmajorgrids,
ymajorgrids,
zmajorgrids
]
\addplot3 [color=mycolor1, line width=1pt]
 table[row sep=crcr] {%
0	-0	0\\
0.504	-0	0\\
0.511	-1	-0\\
0.567	-1	-0\\
0.574	0	0\\
1.288	-0	0\\
1.295	0.5	-0.866025403784438\\
1.351	0.5	-0.866025403784438\\
1.358	0	-0\\
2.065	-0	0\\
2.072	0.389203148246079	0.67411962722797\\
2.079	0.5	0.866025403784438\\
2.135	0.5	0.866025403784438\\
2.142	-0	0\\
2.849	-0	-0\\
2.856	-1	-0\\
2.912	-1	-0\\
2.926	-0	0\\
3.633	-0	-0\\
3.64	0.5	-0.866025403784438\\
3.696	0.5	-0.866025403784438\\
3.703	-0	0\\
4.417	0	-0\\
4.424	0.5	0.866025403784438\\
4.48	0.5	0.866025403784438\\
4.487	-0	0\\
5.201	0	-0\\
5.208	-1	-0\\
5.264	-1	-0\\
5.271	0	0\\
5.985	-0	0\\
5.992	0.5	-0.866025403784438\\
6.048	0.5	-0.866025403784438\\
6.055	-0	0\\
6.762	-0	-0\\
6.769	0.439094307431502	0.760533649785632\\
6.776	0.5	0.866025403784438\\
6.832	0.500000000000001	0.866025403784441\\
6.839	-0	-0\\
6.993	0	0\\
};
 \addplot3 [color=mycolor2, dashed, line width=1.5pt]
 table[row sep=crcr] {%
0	-1	-0\\
6.993	-1	-0\\
};
 \addplot3 [color=mycolor3, dashed, line width=1.5pt]
 table[row sep=crcr] {%
0	0.5	-0.866025403784438\\
6.993	0.5	-0.866025403784438\\
};
 \addplot3 [color=mycolor4, dashed, line width=1.5pt]
 table[row sep=crcr] {%
0	0.5	0.866025403784438\\
6.993	0.5	0.866025403784438\\
};
 \end{axis}
\end{tikzpicture}%
        \caption{control $\tilde u(t)$}\label{fig:Bloch1:u}
    \end{subfigure}
    \hfill
    \begin{subfigure}[t]{0.45\linewidth}
        \centering
        \input{bloch_d3_state.tikz}
        \caption{state $\mathbf{M}_u^{(\omega)}(t)$}\label{fig:Bloch1:M}
    \end{subfigure}
    \caption{Control and state for the Bloch model problem: $M=3$}\label{fig:Bloch1}
\end{figure} 
\begin{figure}
    \centering
    \begin{subfigure}[t]{0.45\linewidth}
        \centering
        \begin{tikzpicture}[trim axis left]

\begin{axis}[%
width=\linewidth,
scale only axis,
xmin=0,
xmax=8,
tick align=outside,
xlabel={$t$},
ymin=-1,
ymax=1,
ylabel={$\omega_x$},
zmin=-1,
zmax=1,
zlabel={$\omega_y$},
view={-37.5}{30},
axis background/.style={fill=white},
axis x line*=bottom,
axis y line*=left,
axis z line*=left,
xmajorgrids,
ymajorgrids,
zmajorgrids
]
\addplot3 [color=mycolor1, line width=1pt]
 table[row sep=crcr] {%
0	-8.88178419700125e-16	0\\
0.126	0	0\\
0.133	-0.5	0.866025403784438\\
0.161	-0.5	0.866025403784436\\
0.168	-0	-0\\
0.518	-0	0\\
0.525	-1	-0\\
0.546	-1	-0\\
0.553	-0.968339124108316	-0\\
0.56	0	-0\\
0.91	-0	-0\\
0.917	-0.5	-0.866025403784438\\
0.938	-0.5	-0.866025403784438\\
0.952	-0	-0\\
1.302	-0	-0\\
1.309	0.5	-0.866025403784438\\
1.33	0.5	-0.866025403784438\\
1.337	-0	0\\
1.694	-0	0\\
1.701	1	0\\
1.722	1	0\\
1.729	0	0\\
2.086	0	0\\
2.093	0.499999999999999	0.866025403784439\\
2.114	0.499999999999999	0.866025403784439\\
2.121	0	0\\
2.478	-0	1.77635683940025e-15\\
2.485	-0.5	0.866025403784438\\
2.506	-0.5	0.866025403784438\\
2.513	0	-0\\
2.863	-0	-0\\
2.87	-0.731267125527537	-0\\
2.877	-1	-0\\
2.898	-1	-0\\
2.905	-0	0\\
3.255	-0	-0\\
3.262	-0.500000000000001	-0.86602540378443\\
3.29	-0.5	-0.866025403784438\\
3.297	0	-0\\
3.647	-0	0\\
3.654	0.499999999999999	-0.866025403784437\\
3.682	0.5	-0.866025403784438\\
3.689	-0	0\\
4.039	0	0\\
4.046	1.00000000000001	0\\
4.074	1	0\\
4.081	0	0\\
4.431	0	0\\
4.438	0.499999999999999	0.866025403784439\\
4.466	0.499999999999999	0.86602540378444\\
4.473	0	-0\\
4.823	0	-0\\
4.83	-0.5	0.866025403784438\\
4.858	-0.5	0.866025403784436\\
4.865	0	-0\\
5.215	-0	-0\\
5.222	-1	-0\\
5.25	-0.999999999999999	-0\\
5.257	-0	0\\
5.607	-0	0\\
5.614	-0.5	-0.866025403784438\\
5.642	-0.499999999999997	-0.866025403784436\\
5.649	0	0\\
5.999	0	-0\\
6.006	0.5	-0.866025403784438\\
6.027	0.5	-0.866025403784438\\
6.034	0.324766587893095	-0.562512230831794\\
6.041	0	0\\
6.391	-0	0\\
6.398	1	0\\
6.419	1	0\\
6.426	-0	0\\
6.783	-0	-0\\
6.79	0.499999999999999	0.866025403784439\\
6.811	0.499999999999999	0.866025403784439\\
6.818	-0	0\\
6.993	0	0\\
};
 \addplot3 [color=mycolor2, dashed, line width=1.5pt]
 table[row sep=crcr] {%
0	-1	-0\\
6.993	-1	-0\\
};
 \addplot3 [color=mycolor3, dashed, line width=1.5pt]
 table[row sep=crcr] {%
0	-0.5	-0.866025403784438\\
6.993	-0.5	-0.866025403784438\\
};
 \addplot3 [color=mycolor4, dashed, line width=1.5pt]
 table[row sep=crcr] {%
0	0.5	-0.866025403784438\\
6.993	0.5	-0.866025403784438\\
};
 \addplot3 [color=mycolor5, dashed, line width=1.5pt]
 table[row sep=crcr] {%
0	1	0\\
6.993	1	0\\
};
 \addplot3 [color=mycolor6, dashed, line width=1.5pt]
 table[row sep=crcr] {%
0	0.499999999999999	0.866025403784439\\
6.993	0.499999999999999	0.866025403784439\\
};
 \addplot3 [color=mycolor7, dashed, line width=1.5pt]
 table[row sep=crcr] {%
0	-0.5	0.866025403784438\\
6.993	-0.5	0.866025403784438\\
};
 \end{axis}
\end{tikzpicture}%
        \caption{control $\tilde u(t)$}\label{fig:Bloch2:u}
    \end{subfigure}
    \hfill
    \begin{subfigure}[t]{0.45\linewidth}
        \centering
        \input{bloch_d6_state.tikz}
        \caption{state $\mathbf{M}_u^{(\omega)}(t)$}\label{fig:Bloch2:M}
    \end{subfigure}
    \caption{Control and state for the Bloch model problem: $M=6$}\label{fig:Bloch2}
\end{figure} 

\begin{figure}
    \centering
    \begin{subfigure}[t]{0.45\linewidth}
        \centering
        \begin{tikzpicture}[trim axis left]

\begin{axis}[%
width=\linewidth,
scale only axis,
xmin=0,
xmax=8,
tick align=outside,
xlabel={$t$},
ymin=-1,
ymax=1,
ylabel={$\omega_x$},
zmin=-1,
zmax=1,
zlabel={$\omega_y$},
view={-37.5}{30},
axis background/.style={fill=white},
axis x line*=bottom,
axis y line*=left,
axis z line*=left,
xmajorgrids,
ymajorgrids,
zmajorgrids
]
\addplot3 [color=mycolor1, line width=1pt]
 table[row sep=crcr] {%
0	-0.5	0.866025403784438\\
0.0490000000000004	-0.5	0.866025403784438\\
0.056	-0.390135461161352	0.675734440565775\\
0.0629999999999997	-0	-0\\
0.651	-0	0\\
0.665	-0.237138774803105	-0\\
0.672	-0	-0\\
0.77	0	0\\
0.777	-0.10146386341886	-0.175740566573695\\
0.784	-0.164113039195898	-0.284252122071837\\
0.791	0	-0\\
1.428	-0	0\\
1.435	0.155726530747316	-0.269726263340789\\
1.442	0.124406663474223	-0.215478661937476\\
1.449	-0	0\\
1.547	-0	0\\
1.554	0.2517482849423	0\\
1.561	0.0485040870096149	0\\
1.568	0	0\\
2.17	0	0\\
2.177	0.00902189733162384	0.0156263845590425\\
2.184	0.499999999999999	0.866025403784439\\
2.289	0.499999999999999	0.866025403784439\\
2.296	-0	-0\\
2.303	-0	-0\\
2.31	-0.5	0.866025403784438\\
2.401	-0.5	0.866025403784438\\
2.408	-0.335528775544908	0.581152886645155\\
2.415	0	-0\\
3.129	0	0\\
3.136	-0.0084788176348729	-0.0146857429317109\\
3.143	0	0\\
4.522	0	-0\\
4.536	0.499999999999999	0.866025403784439\\
4.641	0.499999999999999	0.866025403784439\\
4.648	-0	-0\\
4.655	-0.5	0.866025403784438\\
4.753	-0.5	0.866025403784438\\
4.76	-0.349794162256198	0.605861261218727\\
4.767	0	-0\\
6.874	-0	0\\
6.881	0.358708208735582	0.62130084262205\\
6.888	0.499999999999999	0.866025403784439\\
6.986	0.499999999999999	0.866025403784439\\
6.993	-0	-0\\
};
 \addplot3 [color=mycolor2, dashed, line width=1.5pt]
 table[row sep=crcr] {%
0	-1	-0\\
6.993	-1	-0\\
};
 \addplot3 [color=mycolor3, dashed, line width=1.5pt]
 table[row sep=crcr] {%
0	-0.5	-0.866025403784438\\
6.993	-0.5	-0.866025403784438\\
};
 \addplot3 [color=mycolor4, dashed, line width=1.5pt]
 table[row sep=crcr] {%
0	0.5	-0.866025403784438\\
6.993	0.5	-0.866025403784438\\
};
 \addplot3 [color=mycolor5, dashed, line width=1.5pt]
 table[row sep=crcr] {%
0	1	0\\
6.993	1	0\\
};
 \addplot3 [color=mycolor6, dashed, line width=1.5pt]
 table[row sep=crcr] {%
0	0.499999999999999	0.866025403784439\\
6.993	0.499999999999999	0.866025403784439\\
};
 \addplot3 [color=mycolor7, dashed, line width=1.5pt]
 table[row sep=crcr] {%
0	-0.5	0.866025403784438\\
6.993	-0.5	0.866025403784438\\
};
 \end{axis}
\end{tikzpicture}%
        \caption{control $\tilde u(t)$}\label{fig:Bloch3:u}
    \end{subfigure}
    \hfill
    \begin{subfigure}[t]{0.45\linewidth}
        \centering
        \input{bloch_d6_i4_state.tikz}
        \caption{state $\mathbf{M}_u^{(\omega_j)}(t)$}\label{fig:Bloch3:M}
    \end{subfigure}
    \caption{Control and state for the Bloch model problem: $M=6$, $J=4$}\label{fig:Bloch3}
\end{figure} 
\begin{figure}
    \centering
    \begin{subfigure}[t]{0.45\linewidth}
        \centering
        \begin{tikzpicture}[trim axis left]

\begin{axis}[%
width=\linewidth,
scale only axis,
xmin=0,
xmax=8,
tick align=outside,
xlabel={$t$},
ymin=-1,
ymax=1,
ylabel={$\omega_x$},
zmin=-1,
zmax=1,
zlabel={$\omega_y$},
view={-37.5}{30},
axis background/.style={fill=white},
axis x line*=bottom,
axis y line*=left,
axis z line*=left,
xmajorgrids,
ymajorgrids,
zmajorgrids
]
\addplot3 [color=mycolor1, line width=1pt]
 table[row sep=crcr] {%
0	-0	0\\
0.00699999999999967	0	0\\
0.0140000000000002	-0.5	0.866025403784438\\
0.0209999999999999	0	0\\
0.133	0	0\\
0.14	-1	-0\\
0.147	-1	-0\\
0.154	-0	0\\
0.266	0	-0\\
0.273	-0.5	-0.866025403784438\\
0.28	-0.301760939368978	-0.522665278726782\\
0.287	-0	0\\
0.399	0	-0\\
0.406	0.5	-0.866025403784438\\
0.413	0	-0\\
0.525	-0	0\\
0.532	1	0\\
0.539	1	0\\
0.546	-0	0\\
0.658	-0	-0\\
0.665	0.499999999999999	0.866025403784439\\
0.679	0	0\\
0.784	-0	0\\
0.791	-0.124654847034217	0.215908528472991\\
0.798	-0.5	0.866025403784438\\
0.805	-0	-0\\
0.917	-0	-0\\
0.924	-1	-0\\
0.931	-1	-0\\
0.938	-0	0\\
1.05	-0	-0\\
1.057	-0.5	-0.866025403784438\\
1.064	-0.0110260488915124	-0.0190976768868376\\
1.071	-0	-0\\
1.176	0	0\\
1.183	0.194960546499996	-0.337681572009386\\
1.19	0.5	-0.866025403784438\\
1.197	-0	-0\\
1.309	0	0\\
1.316	1	0\\
1.323	1	0\\
1.33	-0	0\\
1.442	0	0\\
1.449	0.499999999999999	0.866025403784439\\
1.456	-0	-0\\
1.568	0	0\\
1.575	-0.353729988054552	0.612678311471217\\
1.582	-0.5	0.866025403784438\\
1.589	-0	-0\\
1.701	-0	-0\\
1.708	-1	-0\\
1.715	-0.8746925557469	-0\\
1.722	-0	0\\
1.834	0	0\\
1.841	-0.5	-0.866025403784438\\
1.848	-0	-0\\
1.96	-0	-0\\
1.967	0.5	-0.866025403784438\\
1.974	0.5	-0.866025403784438\\
1.981	0	-0\\
2.093	0	0\\
2.1	1	0\\
2.107	0.820583989450242	0\\
2.114	0	0\\
2.219	-0	-0\\
2.226	0.020826304043859	0.0360722167378409\\
2.233	0.499999999999999	0.866025403784439\\
2.24	-0	0\\
2.352	-0	-0\\
2.359	-0.5	0.866025403784438\\
2.366	-0.5	0.866025403784438\\
2.373	-0	0\\
2.485	-0	-0\\
2.492	-1	-0\\
2.506	-0	-0\\
2.611	-0	0\\
2.618	-0.107193807473524	-0.185665120800901\\
2.625	-0.5	-0.866025403784438\\
2.632	-0	-0\\
2.744	0	0\\
2.751	0.5	-0.866025403784438\\
2.758	0.5	-0.866025403784438\\
2.765	0	-0\\
2.877	0	0\\
2.884	1	0\\
2.891	0.199063844241419	0\\
2.898	-0	0\\
3.003	-0	-0\\
3.01	0.210361500514878	0.364356806848197\\
3.017	0.499999999999999	0.866025403784439\\
3.024	-0	0\\
3.136	0	0\\
3.143	-0.5	0.866025403784438\\
3.15	-0.5	0.866025403784438\\
3.157	-0	0\\
3.269	-0	-0\\
3.276	-1	-0\\
3.283	-0	-0\\
3.395	-0	-0\\
3.402	-0.322207954766399	-0.558080548258258\\
3.409	-0.5	-0.866025403784438\\
3.416	-0	0\\
3.528	-0	0\\
3.535	0.5	-0.866025403784438\\
3.542	0.5	-0.866025403784438\\
3.549	0	-0\\
3.661	0	0\\
3.668	1	0\\
3.675	-0	0\\
3.787	0	0\\
3.794	0.367232102962029	0.636064660500599\\
3.801	0.499999999999999	0.866025403784439\\
3.808	0	-0\\
3.92	0	0\\
3.927	-0.5	0.866025403784438\\
3.934	-0.396769649352417	0.687225191579675\\
3.941	-0	-0\\
4.053	0	0\\
4.06	-1	-0\\
4.067	-0	-0\\
4.179	-0	0\\
4.186	-0.5	-0.866025403784438\\
4.193	-0.5	-0.866025403784438\\
4.2	-0	0\\
4.312	-0	-0\\
4.319	0.5	-0.866025403784438\\
4.326	0.308492440046711	-0.534324579911799\\
4.333	-0	0\\
4.438	0	0\\
4.445	0.099686261469107	0\\
4.452	1	0\\
4.459	0	0\\
4.571	-0	0\\
4.578	0.499999999999999	0.866025403784439\\
4.585	0.499999999999999	0.866025403784439\\
4.592	0	0\\
4.704	-0	0\\
4.711	-0.5	0.866025403784438\\
4.725	0	0\\
4.83	0	0\\
4.844	-1	-0\\
4.851	-0	-0\\
4.963	0	-0\\
4.97	-0.5	-0.866025403784438\\
4.977	-0.5	-0.866025403784438\\
4.984	-0	-0\\
5.096	0	-0\\
5.103	0.5	-0.866025403784438\\
5.11	0	-0\\
5.222	0	0\\
5.236	1	0\\
5.243	-0	0\\
5.355	0	-0\\
5.362	0.499999999999999	0.866025403784439\\
5.369	0.499999999999999	0.866025403784439\\
5.376	0	-0\\
5.488	-0	0\\
5.495	-0.5	0.866025403784438\\
5.502	-0	0\\
5.614	-0	-0\\
5.621	-0.847035175587634	-0\\
5.628	-1	-0\\
5.635	0	0\\
5.747	-0	-0\\
5.754	-0.5	-0.866025403784438\\
5.761	-0.5	-0.866025403784438\\
5.768	-0	-0\\
5.88	-0	0\\
5.887	0.5	-0.866025403784438\\
5.894	-0	0\\
6.006	0	0\\
6.013	0.872181644123008	0\\
6.02	1	0\\
6.027	0	0\\
6.139	-0	0\\
6.146	0.499999999999999	0.866025403784439\\
6.153	0.396574145549109	0.686886569059273\\
6.16	-0	0\\
6.272	0	-0\\
6.279	-0.5	0.866025403784438\\
6.286	-0	0\\
6.398	0	0\\
6.405	-1	-0\\
6.412	-1	-0\\
6.419	-0	0\\
6.531	-0	-0\\
6.538	-0.5	-0.866025403784438\\
6.545	-0.17042063016699	-0.295177190107132\\
6.552	0	-0\\
6.657	0	0\\
6.664	0.0155491830246897	-0.0269319750149499\\
6.671	0.5	-0.866025403784438\\
6.678	0	-0\\
6.79	0	0\\
6.797	1	0\\
6.804	1	0\\
6.811	-0	0\\
6.923	0	0\\
6.93	0.499999999999999	0.866025403784439\\
6.937	0.0852671501867031	0.147687036339977\\
6.944	-0	0\\
6.993	0	0\\
};
 \addplot3 [color=mycolor2, dashed, line width=1.5pt]
 table[row sep=crcr] {%
0	-1	-0\\
6.993	-1	-0\\
};
 \addplot3 [color=mycolor3, dashed, line width=1.5pt]
 table[row sep=crcr] {%
0	-0.5	-0.866025403784438\\
6.993	-0.5	-0.866025403784438\\
};
 \addplot3 [color=mycolor4, dashed, line width=1.5pt]
 table[row sep=crcr] {%
0	0.5	-0.866025403784438\\
6.993	0.5	-0.866025403784438\\
};
 \addplot3 [color=mycolor5, dashed, line width=1.5pt]
 table[row sep=crcr] {%
0	1	0\\
6.993	1	0\\
};
 \addplot3 [color=mycolor6, dashed, line width=1.5pt]
 table[row sep=crcr] {%
0	0.499999999999999	0.866025403784439\\
6.993	0.499999999999999	0.866025403784439\\
};
 \addplot3 [color=mycolor7, dashed, line width=1.5pt]
 table[row sep=crcr] {%
0	-0.5	0.866025403784438\\
6.993	-0.5	0.866025403784438\\
};
 \end{axis}
\end{tikzpicture}%
        \caption{control $\tilde u(t)$}\label{fig:Bloch4:u}
    \end{subfigure}
    \hfill
    \begin{subfigure}[t]{0.45\linewidth}
        \centering
        \input{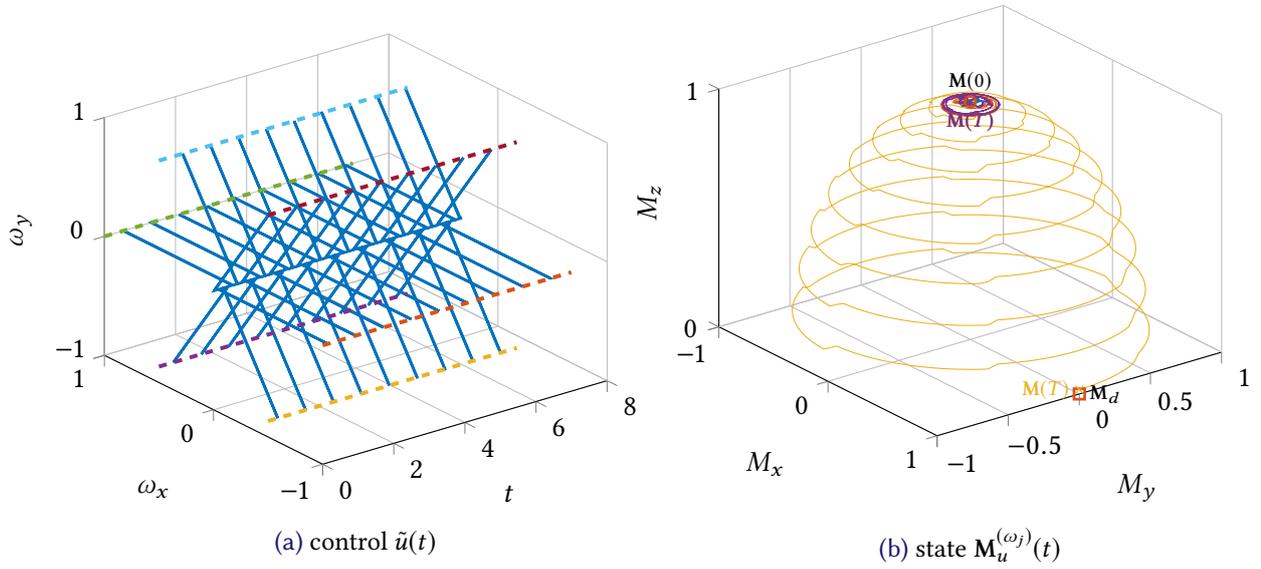}
        \caption{state $\mathbf{M}_u^{(\omega_j)}(t)$}\label{fig:Bloch4:M}
    \end{subfigure}
    \caption{Control and state for the Bloch model problem: $M=6$, $J=4$, $\mathbf M_d = (0,0,1)$ for $j=3$, $\mathbf M_0$ else}\label{fig:Bloch4}
\end{figure} 

\Cref{tab:iterBloch} summarizes the convergence behavior for the case $M=3$ and $J=1$. For a representative selection of values of $\gamma$, it shows the number of semi-smooth Newton iterations, the average number of GMRES iteration needed to solve a Newton step, the number of times a step of length less than $1$ was taken, and the number of nodes $t_m$ for which $u_\gamma(t_m)\notin\calM$. For moderate values of $\gamma$ (approximately $\gamma>10^{-6}$ in this case), very few iterations of both the semi-smooth Newton method and the inner GMRES method are required to reach the solution. If $\gamma$ is decreased further, however, the problem starts becoming significantly more difficult, requiring an increasing number of Newton iterations that in addition require a damping to lead to a decrease of the residual. These damped steps typically are taken after a few initial full steps and continue until the region of superlinear convergence is reached, after which the iteration terminates after a small number of full steps. The average number of GMRES steps, however, remains small. For $\gamma < 9.313\cdot10^{-8}$, the maximal number of semi-smooth Newton iterations are no longer sufficient to reach the given tolerance. However, the final row of the table demonstrates that already for $\gamma \approx 10^{-5}$ (where the convergence is still fast), the control is already almost perfectly multibang.
\begin{table}[t]
    \caption{Convergence behavior for the example in \cref{fig:Bloch1}: number of semi-smooth Newton steps, average number of GMRES iterations to solve a Newton step, number of times a line search was required, number of nodes $t_m$ with $u_\gamma(t_m)\notin\calM$}
    \label{tab:iterBloch}
    \centering
    \resizebox{\linewidth}{!}{%
        \begin{tabular}{lrrrrrrrrr}
            \toprule
            $\gamma$       & \num[round-precision=0]{1e2}  & \num[round-precision=0]{1.6e0} & \num[round-precision=0]{2.0e-1} & \num[round-precision=0]{1.2e-2} & \num[round-precision=0]{1.5e-3} & \num[round-precision=0]{1.9e-4} & \num[round-precision=0]{1.2e-5} & \num[round-precision=0]{1.5e-6} & \num[round-precision=0]{9.3e-8}\\
            \midrule
            \# SSN         & \num{3}    & \num{3}     & \num{4}      & \num{5}      & \num{5}      & \num{5}      & \num{4}      & \num{100}    & \num{101}\\
            avg. \# GMRES  & \num{3}    & \num{7}     & \num{7.5}    & \num{7.4}    & \num{7.8}    & \num{8.2}    & \num{3.75}   & \num{3.14}   & \num{4.3}  \\
            \# line search & \num{0}    & \num{0}     & \num{0}      & \num{0}      & \num{0}      & \num{0}      & \num{0}      & \num{98}     & \num{99} \\
            \# not MB     & \num{1000} & \num{1000}  & \num{862}    & \num{376}    & \num{191}    & \num{44}     & \num{3}      & \num{3}      & \num{3} \\
            \bottomrule
        \end{tabular}
    }
\end{table}

\subsection{Linearized elasticity}

We now address the behavior in the context of optimal control of elliptic partial differential equations for the model equations of two-dimensional linearized elasticity.
Here, we choose $\Omega=[0,1]\times[0,2]$  and $\Gamma=[0,1]\times\{0\}$, which models an elastic beam clamped at the bottom. The 
Lamé parameters are set to $\mu = \frac{E}{2(1+\nu)}$ and $\lambda = \frac{E\nu}{(1+\nu)(1-2\nu)}$ for the elastic modulus $E=20$ and the Poisson ratio $\nu =0.3$. We use a uniform structured mesh with $129$ vertices in each direction.
Since the state equation is linear, we use a direct solver for the Newton step. The Newton iteration is terminated if the active sets (i.e., the case distinctions in the definition of the Moreau--Yosida regularization) for each node coincide for two consecutive iterations, or if $50$ iterations are exceeded. The continuation in the regularization parameter $\gamma$ is performed as for the Bloch equation.

\Cref{fig:elasticExamples} shows the results for six different choices of target, multibang penalty, and control cost parameter.
In examples \ref{fig:elastic:1} to \ref{fig:elastic:4}, the target displacement $z(x)=R(x-(\frac12,1)^T)-x$ corresponds to a rotation $R\in SO(2)$ of the solid around its center.
Examples \ref{fig:elastic:1} and \ref{fig:elastic:2} use the penalty from \cref{sec:multibang:bloch} for $\alpha=10^{-3}$, while examples \ref{fig:elastic:3} and \ref{fig:elastic:4} use the penalty from \cref{sec:multibang:elast} for $\alpha=10^{-5}$ and $\alpha=10^{-3}$, respectively.
In all cases, the obtained control makes use of all control values in $\calM$ and aligns them with the rotation.
Furthermore, the center of the force vortex always lies slightly to the top right of the rotation center of the target state;
this allows a stronger overall rightward force in the lower part of the solid to compensate for the clamping at the bottom.
Note that unlike the case of (additional) gradient regularization of the control, small patches or sharp corners of the domains with homogeneous force are allowed.
\begin{figure}[t] 
    \begin{subfigure}[t]{0.15\linewidth}
        \centering
        \begin{tikzpicture}[x=\linewidth,y=\linewidth]
\path[use as bounding box] (-.5,-.55) rectangle (.5,.5);
\begin{axis}[%
    width=0.66\linewidth,
    height=0.66\linewidth,
    at={(-.33\linewidth,-.33\linewidth)},
    scale only axis,
    axis on top,
    xmin=-128,
    xmax=128,
    ymin=-128,
    ymax=128,
    axis line style={draw=none},
    ticks=none
    ]
    \addplot  graphics [xmin=-128,xmax=128,ymin=-128,ymax=128] {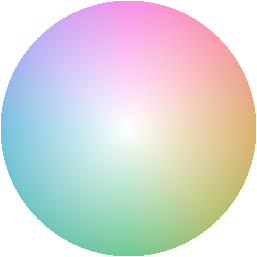};
\end{axis}
\begin{polaraxis}[%
    width=0.66\linewidth,
    height=0.66\linewidth,
    at={(-.33\linewidth,-.33\linewidth)},
    ymin=0,
    ymax=256,
    scale only axis,
    xtick={60,180,300},
    ytick={0,256},
    xticklabels={$\frac{\pi}3$,$\pi$,$-\frac{\pi}3$},
    yticklabels={$0$,$\sqrt{8}$},
    yticklabel style = {font=\tiny}
    ]
\end{polaraxis}
\end{tikzpicture}%
\\
        \begin{tikzpicture}[x=\linewidth,y=\linewidth]
\path[use as bounding box] (0,0) -- (1,0) -- (1,2.1) -- (0,2.1) -- cycle;    % defines the bounding box so one can use absolute coordinates

\begin{axis}[%
width=\linewidth,
height=2\linewidth,
scale only axis,
axis on top,
clip=false,
xmin=0,
xmax=1,
ymin=0,
ymax=2,
axis background/.style={fill=white},
ticks=none
]
\addplot [forget plot] graphics [xmin=0, xmax=1, ymin=0, ymax=2] {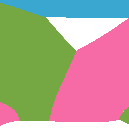};
\addplot[-Latex, 
color=black, 
point meta={sqrt((\thisrow{u})^2+(\thisrow{v})^2)}, 
point meta min=0, 
quiver={u=1.5*\thisrow{u}, v=1.5*\thisrow{v}, 
    every arrow/.append style={-{Latex[scale={1/1000*\pgfplotspointmetatransformed}]}}}]
 table[row sep=crcr] {%
x	y	u	v\\
0.0546875	0.109375	0.05	0.0866025403784439\\
0.0546875	0.359375	0.05	0.0866025403784439\\
0.0546875	0.609375	0.05	-0.0866025403784439\\
0.0546875	0.859375	0.05	-0.0866025403784439\\
0.0546875	1.109375	0.05	-0.0866025403784439\\
0.0546875	1.359375	0.05	-0.0866025403784439\\
0.0546875	1.609375	0.05	-0.0866025403784439\\
0.0546875	1.859375	0.05	-0.0866025403784439\\
% 0.1796875	0.109375	0	0\\
0.1796875	0.359375	0.05	-0.0866025403784439\\
0.1796875	0.609375	0.05	-0.0866025403784439\\
0.1796875	0.859375	0.05	-0.0866025403784439\\
0.1796875	1.109375	0.05	-0.0866025403784439\\
0.1796875	1.359375	0.05	-0.0866025403784439\\
0.1796875	1.609375	0.05	-0.0866025403784439\\
0.1796875	1.859375	-0.1	-1.22464679914735e-17\\
% 0.3046875	0.109375	0	0\\
0.3046875	0.359375	0.05	-0.0866025403784439\\
0.3046875	0.609375	0.05	-0.0866025403784439\\
0.3046875	0.859375	0.05	-0.0866025403784439\\
0.3046875	1.109375	0.05	-0.0866025403784439\\
0.3046875	1.359375	0.05	-0.0866025403784439\\
0.3046875	1.609375	0.05	-0.0866025403784439\\
0.3046875	1.859375	-0.1	-1.22464679914735e-17\\
0.4296875	0.109375	0.05	0.0866025403784439\\
0.4296875	0.359375	0.05	0.0866025403784439\\
0.4296875	0.609375	0.05	-0.0866025403784439\\
0.4296875	0.859375	0.05	-0.0866025403784439\\
0.4296875	1.109375	0.05	-0.0866025403784439\\
0.4296875	1.359375	0.05	-0.0866025403784439\\
% 0.4296875	1.609375	0	0\\
0.4296875	1.859375	-0.1	-1.22464679914735e-17\\
0.5546875	0.109375	0.05	0.0866025403784439\\
0.5546875	0.359375	0.05	0.0866025403784439\\
0.5546875	0.609375	0.05	0.0866025403784439\\
0.5546875	0.859375	0.05	0.0866025403784439\\
0.5546875	1.109375	0.05	-0.0866025403784439\\
% 0.5546875	1.359375	0	0\\
% 0.5546875	1.609375	0	0\\
0.5546875	1.859375	-0.1	-1.22464679914735e-17\\
0.6796875	0.109375	0.05	0.0866025403784439\\
0.6796875	0.359375	0.05	0.0866025403784439\\
0.6796875	0.609375	0.05	0.0866025403784439\\
0.6796875	0.859375	0.05	0.0866025403784439\\
0.6796875	1.109375	0.05	0.0866025403784439\\
% 0.6796875	1.359375	0	0\\
% 0.6796875	1.609375	0	0\\
0.6796875	1.859375	-0.1	-1.22464679914735e-17\\
0.8046875	0.109375	0.05	0.0866025403784439\\
0.8046875	0.359375	0.05	0.0866025403784439\\
0.8046875	0.609375	0.05	0.0866025403784439\\
0.8046875	0.859375	0.05	0.0866025403784439\\
0.8046875	1.109375	0.05	0.0866025403784439\\
0.8046875	1.359375	0.05	0.0866025403784439\\
% 0.8046875	1.609375	0	0\\
0.8046875	1.859375	-0.1	-1.22464679914735e-17\\
% 0.9296875	0.109375	0	0\\
0.9296875	0.359375	0.05	0.0866025403784439\\
0.9296875	0.609375	0.05	0.0866025403784439\\
0.9296875	0.859375	0.05	0.0866025403784439\\
0.9296875	1.109375	0.05	0.0866025403784439\\
0.9296875	1.359375	0.05	0.0866025403784439\\
0.9296875	1.609375	0.05	0.0866025403784439\\
0.9296875	1.859375	-0.1	-1.22464679914735e-17\\
};
\end{axis}
\end{tikzpicture}%
\\
        \input{elast_rad3_deform.tikz}
        \caption{radial, $d=3$, $\alpha = 10^{-3}$}\label{fig:elastic:1}
    \end{subfigure}
    \hfill
    \begin{subfigure}[t]{0.15\linewidth}
        \centering
        \begin{tikzpicture}[x=\linewidth,y=\linewidth]
\path[use as bounding box] (-.5,-.55) rectangle (.5,.5);
\begin{axis}[%
    width=0.66\linewidth,
    height=0.66\linewidth,
    at={(-.33\linewidth,-.33\linewidth)},
    scale only axis,
    axis on top,
    xmin=-128,
    xmax=128,
    ymin=-128,
    ymax=128,
    axis line style={draw=none},
    ticks=none
    ]
    \addplot  graphics [xmin=-128,xmax=128,ymin=-128,ymax=128] {wheel.png};
\end{axis}
\begin{polaraxis}[%
    width=0.66\linewidth,
    height=0.66\linewidth,
    at={(-.33\linewidth,-.33\linewidth)},
    ymin=0,
    ymax=256,
    scale only axis,
    xtick={36,108,180,252,324},
    ytick={0,256},
    xticklabels={$\frac{\pi}5$,$\frac{2\pi}5$,$\pi$,$-\frac{2\pi}5$,$-\frac{\pi}5$},
    yticklabels={$0$,$\sqrt{8}$},
    yticklabel style = {font=\tiny}
    ]
\end{polaraxis}
\end{tikzpicture}%
\\
        \begin{tikzpicture}[x=\linewidth,y=\linewidth]
\path[use as bounding box] (0,0) -- (1,0) -- (1,2.1) -- (0,2.1) -- cycle;    % defines the bounding box so one can use absolute coordinates

\begin{axis}[%
width=\linewidth,
height=2\linewidth,
scale only axis,
axis on top,
clip=false,
xmin=0,
xmax=1,
ymin=0,
ymax=2,
axis background/.style={fill=white},
ticks=none
]
\addplot [forget plot] graphics [xmin=0, xmax=1, ymin=0, ymax=2] {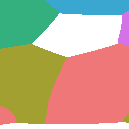};
\addplot[-Latex, 
color=black, 
point meta={sqrt((\thisrow{u})^2+(\thisrow{v})^2)}, 
point meta min=0, 
quiver={u=1.5*\thisrow{u}, v=1.5*\thisrow{v}, 
    every arrow/.append style={-{Latex[scale={1/1000*\pgfplotspointmetatransformed}]}}}]
 table[row sep=crcr] {%
x	y	u	v\\
0.0546875	0.109375	0.0809016994374947	0.0587785252292473\\
0.0546875	0.359375	0.0809016994374947	-0.0587785252292473\\
0.0546875	0.609375	0.0809016994374947	-0.0587785252292473\\
0.0546875	0.859375	0.0809016994374947	-0.0587785252292473\\
0.0546875	1.109375	0.0809016994374947	-0.0587785252292473\\
0.0546875	1.359375	-0.0309016994374947	-0.0951056516295154\\
0.0546875	1.609375	-0.0309016994374947	-0.0951056516295154\\
0.0546875	1.859375	-0.0309016994374947	-0.0951056516295154\\
0.1796875	0.109375	0.0809016994374947	-0.0587785252292473\\
0.1796875	0.359375	0.0809016994374947	-0.0587785252292473\\
0.1796875	0.609375	0.0809016994374947	-0.0587785252292473\\
0.1796875	0.859375	0.0809016994374947	-0.0587785252292473\\
0.1796875	1.109375	0.0809016994374947	-0.0587785252292473\\
0.1796875	1.359375	-0.0309016994374947	-0.0951056516295154\\
0.1796875	1.609375	-0.0309016994374947	-0.0951056516295154\\
0.1796875	1.859375	-0.0309016994374947	-0.0951056516295154\\
0.3046875	0.109375	0.0809016994374947	-0.0587785252292473\\
0.3046875	0.359375	0.0809016994374947	-0.0587785252292473\\
0.3046875	0.609375	0.0809016994374947	-0.0587785252292473\\
0.3046875	0.859375	0.0809016994374947	-0.0587785252292473\\
0.3046875	1.109375	0.0809016994374947	-0.0587785252292473\\
% 0.3046875	1.359375	0	0\\
0.3046875	1.609375	-0.0309016994374947	-0.0951056516295154\\
0.3046875	1.859375	-0.0309016994374947	-0.0951056516295154\\
0.4296875	0.109375	0.0809016994374947	0.0587785252292473\\
0.4296875	0.359375	0.0809016994374947	0.0587785252292473\\
0.4296875	0.609375	0.0809016994374947	0.0587785252292473\\
0.4296875	0.859375	0.0809016994374947	-0.0587785252292473\\
0.4296875	1.109375	0.0809016994374947	-0.0587785252292473\\
% 0.4296875	1.359375	0	0\\
% 0.4296875	1.609375	0	0\\
0.4296875	1.859375	-0.0309016994374947	-0.0951056516295154\\
0.5546875	0.109375	0.0809016994374947	0.0587785252292473\\
0.5546875	0.359375	0.0809016994374947	0.0587785252292473\\
0.5546875	0.609375	0.0809016994374947	0.0587785252292473\\
0.5546875	0.859375	0.0809016994374947	0.0587785252292473\\
0.5546875	1.109375	0.0809016994374947	0.0587785252292473\\
% 0.5546875	1.359375	0	0\\
% 0.5546875	1.609375	0	0\\
0.5546875	1.859375	-0.1	-1.22464679914735e-17\\
0.6796875	0.109375	0.0809016994374947	0.0587785252292473\\
0.6796875	0.359375	0.0809016994374947	0.0587785252292473\\
0.6796875	0.609375	0.0809016994374947	0.0587785252292473\\
0.6796875	0.859375	0.0809016994374947	0.0587785252292473\\
0.6796875	1.109375	0.0809016994374947	0.0587785252292473\\
% 0.6796875	1.359375	0	0\\
% 0.6796875	1.609375	0	0\\
0.6796875	1.859375	-0.1	-1.22464679914735e-17\\
0.8046875	0.109375	0.0809016994374947	0.0587785252292473\\
0.8046875	0.359375	0.0809016994374947	0.0587785252292473\\
0.8046875	0.609375	0.0809016994374947	0.0587785252292473\\
0.8046875	0.859375	0.0809016994374947	0.0587785252292473\\
0.8046875	1.109375	0.0809016994374947	0.0587785252292473\\
% 0.8046875	1.359375	0	0\\
% 0.8046875	1.609375	0	0\\
0.8046875	1.859375	-0.1	-1.22464679914735e-17\\
0.9296875	0.109375	0.0809016994374947	0.0587785252292473\\
0.9296875	0.359375	0.0809016994374947	0.0587785252292473\\
0.9296875	0.609375	0.0809016994374947	0.0587785252292473\\
0.9296875	0.859375	0.0809016994374947	0.0587785252292473\\
0.9296875	1.109375	0.0809016994374947	0.0587785252292473\\
0.9296875	1.359375	-0.0309016994374947	0.0951056516295154\\
% 0.9296875	1.609375	0	0\\
0.9296875	1.859375	-0.1	-1.22464679914735e-17\\
};
\end{axis}
\end{tikzpicture}%
\\
        \input{elast_rad5_deform.tikz}
        \caption{radial, $d=5$, $\alpha = 10^{-3}$}\label{fig:elastic:2}
    \end{subfigure}
    \hfill
    \begin{subfigure}[t]{0.15\linewidth}
        \centering
        \begin{tikzpicture}[x=\linewidth,y=\linewidth]
\path[use as bounding box] (-.5,-.55) rectangle (.5,.5);
\begin{axis}[%
    width=0.66\linewidth,
    height=0.66\linewidth,
    at={(-.33\linewidth,-.33\linewidth)},
    scale only axis,
    axis on top,
    xmin=-128,
    xmax=128,
    ymin=-128,
    ymax=128,
    axis line style={draw=none},
    ticks=none
    ]
    \addplot  graphics [xmin=-128,xmax=128,ymin=-128,ymax=128] {wheel.png};
\end{axis}
\begin{polaraxis}[%
    width=0.66\linewidth,
    height=0.66\linewidth,
    at={(-.33\linewidth,-.33\linewidth)},
    ymin=0,
    ymax=256,
    %     hide y axis,
    scale only axis,
    xtick={45,135,225,315},
    ytick={128,256},
    xticklabels={$\frac{\pi}4$,$\frac{3\pi}{4}\!$,$\frac{-3\pi}4\!$,$\frac{-\pi}4$},
    yticklabels={$\sqrt{2}$,$\sqrt{8}$},
    yticklabel style = {font=\tiny}
    ]
\end{polaraxis}
\end{tikzpicture}%
\\
        \begin{tikzpicture}[x=\linewidth,y=\linewidth]
\path[use as bounding box] (0,0) -- (1,0) -- (1,2.1) -- (0,2.1) -- cycle;    % defines the bounding box so one can use absolute coordinates

\begin{axis}[%
width=\linewidth,
height=2\linewidth,
scale only axis,
axis on top,
clip=false,
xmin=0,
xmax=1,
ymin=0,
ymax=2,
axis background/.style={fill=white},
ticks=none
]
\addplot [forget plot] graphics [xmin=0, xmax=1, ymin=0, ymax=2] {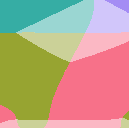};
\addplot[-Latex, 
color=black, 
point meta={sqrt((\thisrow{u})^2+(\thisrow{v})^2)}, 
point meta min=0, 
quiver={u=1.5*\thisrow{u}, v=1.5*\thisrow{v}, 
    every arrow/.append style={-{Latex[scale={1/1000*\pgfplotspointmetatransformed}]}}}]
 table[row sep=crcr] {%
x	y	u	v\\
0.0546875	0.109375	0.0353553390593274	0.0353553390593274\\
0.0546875	0.359375	0.0707106781186548	-0.0707106781186548\\
0.0546875	0.609375	0.0707106781186548	-0.0707106781186548\\
0.0546875	0.859375	0.0707106781186548	-0.0707106781186548\\
0.0546875	1.109375	0.0707106781186548	-0.0707106781186548\\
0.0546875	1.359375	0.0707106781186548	-0.0707106781186548\\
0.0546875	1.609375	-0.0707106781186548	-0.0707106781186548\\
0.0546875	1.859375	-0.0707106781186548	-0.0707106781186548\\
0.1796875	0.109375	0.0353553390593274	-0.0353553390593274\\
0.1796875	0.359375	0.0707106781186548	-0.0707106781186548\\
0.1796875	0.609375	0.0707106781186548	-0.0707106781186548\\
0.1796875	0.859375	0.0707106781186548	-0.0707106781186548\\
0.1796875	1.109375	0.0707106781186548	-0.0707106781186548\\
0.1796875	1.359375	0.0707106781186548	-0.0707106781186548\\
0.1796875	1.609375	-0.0707106781186548	-0.0707106781186548\\
0.1796875	1.859375	-0.0707106781186548	-0.0707106781186548\\
0.3046875	0.109375	0.0353553390593274	-0.0353553390593274\\
0.3046875	0.359375	0.0707106781186548	-0.0707106781186548\\
0.3046875	0.609375	0.0707106781186548	-0.0707106781186548\\
0.3046875	0.859375	0.0707106781186548	-0.0707106781186548\\
0.3046875	1.109375	0.0707106781186548	-0.0707106781186548\\
0.3046875	1.359375	0.0353553390593274	-0.0353553390593274\\
0.3046875	1.609375	-0.0353553390593274	-0.0353553390593274\\
0.3046875	1.859375	-0.0707106781186548	-0.0707106781186548\\
0.4296875	0.109375	0.0353553390593274	0.0353553390593274\\
0.4296875	0.359375	0.0707106781186548	0.0707106781186548\\
0.4296875	0.609375	0.0707106781186548	-0.0707106781186548\\
0.4296875	0.859375	0.0707106781186548	-0.0707106781186548\\
0.4296875	1.109375	0.0707106781186548	-0.0707106781186548\\
0.4296875	1.359375	0.0353553390593274	-0.0353553390593274\\
0.4296875	1.609375	-0.0353553390593274	-0.0353553390593274\\
0.4296875	1.859375	-0.0707106781186548	-0.0707106781186548\\
0.5546875	0.109375	0.0353553390593274	0.0353553390593274\\
0.5546875	0.359375	0.0707106781186548	0.0707106781186548\\
0.5546875	0.609375	0.0707106781186548	0.0707106781186548\\
0.5546875	0.859375	0.0707106781186548	0.0707106781186548\\
0.5546875	1.109375	0.0353553390593274	-0.0353553390593274\\
0.5546875	1.359375	0.0353553390593274	-0.0353553390593274\\
0.5546875	1.609375	-0.0353553390593274	-0.0353553390593274\\
0.5546875	1.859375	-0.0353553390593274	-0.0353553390593274\\
0.6796875	0.109375	0.0353553390593274	0.0353553390593274\\
0.6796875	0.359375	0.0707106781186548	0.0707106781186548\\
0.6796875	0.609375	0.0707106781186548	0.0707106781186548\\
0.6796875	0.859375	0.0707106781186548	0.0707106781186548\\
0.6796875	1.109375	0.0707106781186548	0.0707106781186548\\
0.6796875	1.359375	0.0353553390593274	0.0353553390593274\\
0.6796875	1.609375	-0.0353553390593274	-0.0353553390593274\\
0.6796875	1.859375	-0.0353553390593274	-0.0353553390593274\\
0.8046875	0.109375	0.0353553390593274	0.0353553390593274\\
0.8046875	0.359375	0.0707106781186548	0.0707106781186548\\
0.8046875	0.609375	0.0707106781186548	0.0707106781186548\\
0.8046875	0.859375	0.0707106781186548	0.0707106781186548\\
0.8046875	1.109375	0.0707106781186548	0.0707106781186548\\
0.8046875	1.359375	0.0353553390593274	0.0353553390593274\\
0.8046875	1.609375	-0.0353553390593274	0.0353553390593274\\
0.8046875	1.859375	-0.0353553390593274	0.0353553390593274\\
0.9296875	0.109375	0.0353553390593274	0.0353553390593274\\
0.9296875	0.359375	0.0707106781186548	0.0707106781186548\\
0.9296875	0.609375	0.0707106781186548	0.0707106781186548\\
0.9296875	0.859375	0.0707106781186548	0.0707106781186548\\
0.9296875	1.109375	0.0707106781186548	0.0707106781186548\\
0.9296875	1.359375	0.0353553390593274	0.0353553390593274\\
0.9296875	1.609375	-0.0353553390593274	0.0353553390593274\\
0.9296875	1.859375	-0.0353553390593274	0.0353553390593274\\
};
\end{axis}
\end{tikzpicture}%
\\
        \input{elast_box-3_deform.tikz}
        \caption{concentric, $\alpha = 10^{-3}$}\label{fig:elastic:3}
    \end{subfigure}
    \hfill
    \begin{subfigure}[t]{0.15\linewidth}
        \centering
        \begin{tikzpicture}[x=\linewidth,y=\linewidth]
\path[use as bounding box] (-.5,-.55) rectangle (.5,.5);
\begin{axis}[%
    width=0.66\linewidth,
    height=0.66\linewidth,
    at={(-.33\linewidth,-.33\linewidth)},
    scale only axis,
    axis on top,
    xmin=-128,
    xmax=128,
    ymin=-128,
    ymax=128,
    axis line style={draw=none},
    ticks=none
    ]
    \addplot  graphics [xmin=-128,xmax=128,ymin=-128,ymax=128] {wheel.png};
\end{axis}
\begin{polaraxis}[%
    width=0.66\linewidth,
    height=0.66\linewidth,
    at={(-.33\linewidth,-.33\linewidth)},
    ymin=0,
    ymax=256,
    %     hide y axis,
    scale only axis,
    xtick={45,135,225,315},
    ytick={128,256},
    xticklabels={$\frac{\pi}4$,$\frac{3\pi}{4}\!$,$\frac{-3\pi}4\!$,$\frac{-\pi}4$},
    yticklabels={$\sqrt{2}$,$\sqrt{8}$},
    yticklabel style = {font=\tiny}
    ]
\end{polaraxis}
\end{tikzpicture}%
\\
        \begin{tikzpicture}[x=\linewidth,y=\linewidth]
\path[use as bounding box] (0,0) -- (1,0) -- (1,2.1) -- (0,2.1) -- cycle;    % defines the bounding box so one can use absolute coordinates

\begin{axis}[%
width=\linewidth,
height=2\linewidth,
scale only axis,
axis on top,
clip=false,
xmin=0,
xmax=1,
ymin=0,
ymax=2,
axis background/.style={fill=white},
ticks=none
]
\addplot [forget plot] graphics [xmin=0, xmax=1, ymin=0, ymax=2] {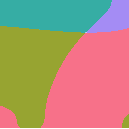};
\addplot[-Latex, 
color=black, 
point meta={sqrt((\thisrow{u})^2+(\thisrow{v})^2)}, 
point meta min=0, 
quiver={u=1.5*\thisrow{u}, v=1.5*\thisrow{v}, 
    every arrow/.append style={-{Latex[scale={1/1000*\pgfplotspointmetatransformed}]}}}]
 table[row sep=crcr] {%
x	y	u	v\\
0.0546875	0.109375	0.0707106781186548	0.0707106781186548\\
0.0546875	0.359375	0.0707106781186548	-0.0707106781186548\\
0.0546875	0.609375	0.0707106781186548	-0.0707106781186548\\
0.0546875	0.859375	0.0707106781186548	-0.0707106781186548\\
0.0546875	1.109375	0.0707106781186548	-0.0707106781186548\\
0.0546875	1.359375	0.0707106781186548	-0.0707106781186548\\
0.0546875	1.609375	-0.0707106781186548	-0.0707106781186548\\
0.0546875	1.859375	-0.0707106781186548	-0.0707106781186548\\
0.1796875	0.109375	0.0707106781186548	-0.0707106781186548\\
0.1796875	0.359375	0.0707106781186548	-0.0707106781186548\\
0.1796875	0.609375	0.0707106781186548	-0.0707106781186548\\
0.1796875	0.859375	0.0707106781186548	-0.0707106781186548\\
0.1796875	1.109375	0.0707106781186548	-0.0707106781186548\\
0.1796875	1.359375	0.0707106781186548	-0.0707106781186548\\
0.1796875	1.609375	-0.0707106781186548	-0.0707106781186548\\
0.1796875	1.859375	-0.0707106781186548	-0.0707106781186548\\
0.3046875	0.109375	0.0707106781186548	-0.0707106781186548\\
0.3046875	0.359375	0.0707106781186548	-0.0707106781186548\\
0.3046875	0.609375	0.0707106781186548	-0.0707106781186548\\
0.3046875	0.859375	0.0707106781186548	-0.0707106781186548\\
0.3046875	1.109375	0.0707106781186548	-0.0707106781186548\\
0.3046875	1.359375	0.0707106781186548	-0.0707106781186548\\
0.3046875	1.609375	-0.0707106781186548	-0.0707106781186548\\
0.3046875	1.859375	-0.0707106781186548	-0.0707106781186548\\
0.4296875	0.109375	0.0707106781186548	0.0707106781186548\\
0.4296875	0.359375	0.0707106781186548	0.0707106781186548\\
0.4296875	0.609375	0.0707106781186548	0.0707106781186548\\
0.4296875	0.859375	0.0707106781186548	-0.0707106781186548\\
0.4296875	1.109375	0.0707106781186548	-0.0707106781186548\\
0.4296875	1.359375	0.0707106781186548	-0.0707106781186548\\
0.4296875	1.609375	-0.0707106781186548	-0.0707106781186548\\
0.4296875	1.859375	-0.0707106781186548	-0.0707106781186548\\
0.5546875	0.109375	0.0707106781186548	0.0707106781186548\\
0.5546875	0.359375	0.0707106781186548	0.0707106781186548\\
0.5546875	0.609375	0.0707106781186548	0.0707106781186548\\
0.5546875	0.859375	0.0707106781186548	0.0707106781186548\\
0.5546875	1.109375	0.0707106781186548	0.0707106781186548\\
0.5546875	1.359375	0.0707106781186548	-0.0707106781186548\\
0.5546875	1.609375	-0.0707106781186548	-0.0707106781186548\\
0.5546875	1.859375	-0.0707106781186548	-0.0707106781186548\\
0.6796875	0.109375	0.0707106781186548	0.0707106781186548\\
0.6796875	0.359375	0.0707106781186548	0.0707106781186548\\
0.6796875	0.609375	0.0707106781186548	0.0707106781186548\\
0.6796875	0.859375	0.0707106781186548	0.0707106781186548\\
0.6796875	1.109375	0.0707106781186548	0.0707106781186548\\
0.6796875	1.359375	0.0707106781186548	0.0707106781186548\\
0.6796875	1.609375	-0.0707106781186548	-0.0707106781186548\\
0.6796875	1.859375	-0.0707106781186548	-0.0707106781186548\\
0.8046875	0.109375	0.0707106781186548	0.0707106781186548\\
0.8046875	0.359375	0.0707106781186548	0.0707106781186548\\
0.8046875	0.609375	0.0707106781186548	0.0707106781186548\\
0.8046875	0.859375	0.0707106781186548	0.0707106781186548\\
0.8046875	1.109375	0.0707106781186548	0.0707106781186548\\
0.8046875	1.359375	0.0707106781186548	0.0707106781186548\\
0.8046875	1.609375	-0.0707106781186548	0.0707106781186548\\
0.8046875	1.859375	-0.0707106781186548	-0.0707106781186548\\
0.9296875	0.109375	0.0707106781186548	0.0707106781186548\\
0.9296875	0.359375	0.0707106781186548	0.0707106781186548\\
0.9296875	0.609375	0.0707106781186548	0.0707106781186548\\
0.9296875	0.859375	0.0707106781186548	0.0707106781186548\\
0.9296875	1.109375	0.0707106781186548	0.0707106781186548\\
0.9296875	1.359375	0.0707106781186548	0.0707106781186548\\
0.9296875	1.609375	-0.0707106781186548	0.0707106781186548\\
0.9296875	1.859375	-0.0707106781186548	0.0707106781186548\\
};
\end{axis}
\end{tikzpicture}%
\\
        \input{elast_box-5_deform.tikz}
        \caption{concentric, $\alpha = 10^{-5}$}\label{fig:elastic:4}
    \end{subfigure}
    \hfill
    \begin{subfigure}[t]{0.15\linewidth}
        \centering
        \begin{tikzpicture}[x=\linewidth,y=\linewidth]
\path[use as bounding box] (-.5,-.55) rectangle (.5,.5);
\begin{axis}[%
    width=0.66\linewidth,
    height=0.66\linewidth,
    at={(-.33\linewidth,-.33\linewidth)},
    scale only axis,
    axis on top,
    xmin=-128,
    xmax=128,
    ymin=-128,
    ymax=128,
    axis line style={draw=none},
    ticks=none
    ]
    \addplot  graphics [xmin=-128,xmax=128,ymin=-128,ymax=128] {wheel.png};
\end{axis}
\begin{polaraxis}[%
    width=0.66\linewidth,
    height=0.66\linewidth,
    at={(-.33\linewidth,-.33\linewidth)},
    ymin=0,
    ymax=256,
    %     hide y axis,
    scale only axis,
    xtick={45,135,225,315},
    ytick={128,256},
    xticklabels={$\frac{\pi}4$,$\frac{3\pi}{4}\!$,$\frac{-3\pi}4\!$,$\frac{-\pi}4$},
    yticklabels={$\sqrt{2}$,$\sqrt{8}$},
    yticklabel style = {font=\tiny}
    ]
\end{polaraxis}
\end{tikzpicture}%
\\
        \begin{tikzpicture}[x=\linewidth,y=\linewidth]
\path[use as bounding box] (0,0) -- (1,0) -- (1,2.1) -- (0,2.1) -- cycle;    % defines the bounding box so one can use absolute coordinates

\begin{axis}[%
width=\linewidth,
height=2\linewidth,
scale only axis,
axis on top,
clip=false,
xmin=0,
xmax=1,
ymin=0,
ymax=2,
axis background/.style={fill=white},
ticks=none
]
\addplot [forget plot] graphics [xmin=0, xmax=1, ymin=0, ymax=2] {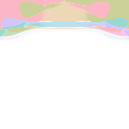};
\addplot[-Latex, 
color=black, 
point meta={sqrt((\thisrow{u})^2+(\thisrow{v})^2)}, 
point meta min=0, 
point meta max=0.1,
quiver={u=1.5*\thisrow{u}, v=1.5*\thisrow{v}, 
    every arrow/.append style={-{Latex[scale={1/1000*\pgfplotspointmetatransformed}]}}}]
 table[row sep=crcr] {%
x	y	u	v\\
% 0.0546875	0.109375	1.09627009429399e-14	1.5308996426413e-14\\
% 0.0546875	0.359375	5.39064489232603e-14	3.98711022674793e-14\\
% 0.0546875	0.609375	9.82573030064356e-14	9.16465782117809e-14\\
% 0.0546875	0.859375	3.87942386853329e-13	2.9903966734534e-12\\
% 0.0546875	1.109375	-3.54955205397894e-10	-5.79580742850944e-08\\
0.0546875	1.359375	-0.000283327586604607	0.000995558151701052\\
0.0546875	1.609375	-0.0353553390593274	0.0353553390593274\\
0.0546875	1.859375	0.0353553390593274	0.0353553390593274\\
% 0.1796875	0.109375	-2.18278488120618e-16	4.68378866275093e-15\\
% 0.1796875	0.359375	6.00638786871082e-14	1.10113738237313e-14\\
% 0.1796875	0.609375	7.92815194023522e-14	3.1858238888281e-14\\
% 0.1796875	0.859375	8.77054283628308e-13	-5.13755212057103e-12\\
% 0.1796875	1.109375	-1.45785607591114e-08	5.47996756953978e-08\\
0.1796875	1.359375	0.000263743367003969	-0.000410714927635264\\
0.1796875	1.609375	-0.0353553390593274	-0.0263140159852521\\
0.1796875	1.859375	0.0353553390593274	-0.0353553390593274\\
% 0.3046875	0.109375	5.24615194129732e-16	2.90394657082058e-15\\
% 0.3046875	0.359375	5.6002520522574e-14	1.38640476760506e-14\\
% 0.3046875	0.609375	1.51190735175495e-13	9.91209340241036e-15\\
% 0.3046875	0.859375	1.58035286165691e-12	-3.64695711697602e-12\\
% 0.3046875	1.109375	-5.45383922041285e-09	9.30405250445378e-09\\
% 0.3046875	1.359375	-1.22097041397243e-05	1.69870773305879e-05\\
0.3046875	1.609375	-0.0353553390593274	-0.0273054273281385\\
0.3046875	1.859375	0.0353553390593274	-0.0353553390593274\\
% 0.4296875	0.109375	-3.24136124078499e-15	9.97798846827815e-16\\
% 0.4296875	0.359375	6.91680811601134e-14	4.65758638530297e-15\\
% 0.4296875	0.609375	1.43060760276625e-13	-4.03899408761626e-15\\
% 0.4296875	0.859375	6.56984857913708e-16	2.83420505933784e-13\\
% 0.4296875	1.109375	8.38329866437085e-10	-1.54951010286047e-09\\
% 0.4296875	1.359375	-6.31294064434132e-07	3.85085390459356e-06\\
0.4296875	1.609375	-0.0353553390593274	-0.00560682635689849\\
0.4296875	1.859375	0.0353553390593274	-0.00167050569451843\\
% 0.5546875	0.109375	4.31114895095901e-16	-7.81128878900552e-16\\
% 0.5546875	0.359375	4.14322699353917e-14	-6.04638149505853e-15\\
% 0.5546875	0.609375	9.65223053085496e-14	-6.5965224691974e-16\\
% 0.5546875	0.859375	3.07635247941292e-13	-1.33760325358885e-13\\
% 0.5546875	1.109375	4.14115238240016e-10	8.78248164087295e-10\\
% 0.5546875	1.359375	-5.51289703854088e-07	-2.69644111370393e-06\\
0.5546875	1.609375	-0.0353553390593274	0.00406141696236012\\
0.5546875	1.859375	0.0353553390593274	0.00274745639862095\\
% 0.6796875	0.109375	-1.79690979081847e-15	-3.36728904251417e-15\\
% 0.6796875	0.359375	5.27982318484371e-14	-2.05612212074931e-14\\
% 0.6796875	0.609375	2.17025291965754e-13	2.12228550263644e-15\\
% 0.6796875	0.859375	1.02387522565234e-12	1.87914899680858e-12\\
% 0.6796875	1.109375	-2.71998354579714e-09	-4.36123056927142e-09\\
% 0.6796875	1.359375	-8.48928000863558e-06	-1.43290212026193e-05\\
0.6796875	1.609375	-0.0353553390593274	0.0230661354414778\\
0.6796875	1.859375	0.0353553390593274	0.0353553390593274\\
% 0.8046875	0.109375	-9.26976446829452e-16	-7.4992336858683e-15\\
% 0.8046875	0.359375	8.00780007102202e-14	-2.06693656559229e-14\\
% 0.8046875	0.609375	1.05167121195448e-13	-3.78233885150467e-14\\
% 0.8046875	0.859375	9.09044160860213e-13	2.28949972318017e-12\\
% 0.8046875	1.109375	-1.16932635944182e-08	-3.46079138819382e-08\\
0.8046875	1.359375	0.000175031022078319	0.000271420370308644\\
0.8046875	1.609375	-0.0353553390593274	0.00715909474315652\\
0.8046875	1.859375	0.0353553390593274	0.0353553390593274\\
% 0.9296875	0.109375	6.78729689813652e-16	-1.50486146815374e-14\\
% 0.9296875	0.359375	6.8434994239168e-14	-5.84160152369736e-14\\
% 0.9296875	0.609375	2.3428299651572e-13	-2.38118236195725e-14\\
% 0.9296875	0.859375	7.0491256199699e-13	-2.14384780773251e-12\\
% 0.9296875	1.109375	-1.58281788630802e-09	4.63413487122573e-08\\
0.9296875	1.359375	-0.000186682252276792	-0.000701918731637707\\
0.9296875	1.609375	-0.0353553390593274	-0.0353553390593274\\
0.9296875	1.859375	0.0353553390593274	-0.0353553390593274\\
};
\end{axis}
\end{tikzpicture}%
\\
        \input{elast_box_c2_deform.tikz}
        \caption{concentric, $\alpha = 10^{-5}$}\label{fig:elastic:5}
    \end{subfigure}
    \hfill
    \begin{subfigure}[t]{0.15\linewidth}
        \centering
        \begin{tikzpicture}[x=\linewidth,y=\linewidth]
\path[use as bounding box] (-.5,-.55) rectangle (.5,.5);
\begin{axis}[%
    width=0.66\linewidth,
    height=0.66\linewidth,
    at={(-.33\linewidth,-.33\linewidth)},
    scale only axis,
    axis on top,
    xmin=-128,
    xmax=128,
    ymin=-128,
    ymax=128,
    axis line style={draw=none},
    ticks=none
    ]
    \addplot  graphics [xmin=-128,xmax=128,ymin=-128,ymax=128] {wheel.png};
\end{axis}
\begin{polaraxis}[%
    width=0.66\linewidth,
    height=0.66\linewidth,
    at={(-.33\linewidth,-.33\linewidth)},
    ymin=0,
    ymax=256,
    %     hide y axis,
    scale only axis,
    xtick={45,135,225,315},
    ytick={128,256},
    xticklabels={$\frac{\pi}4$,$\frac{3\pi}{4}\!$,$\frac{-3\pi}4\!$,$\frac{-\pi}4$},
    yticklabels={$\sqrt{2}$,$\sqrt{8}$},
    yticklabel style = {font=\tiny}
    ]
\end{polaraxis}
\end{tikzpicture}%
\\
        \begin{tikzpicture}[x=\linewidth,y=\linewidth]
\path[use as bounding box] (0,0) -- (1,0) -- (1,2.1) -- (0,2.1) -- cycle;    % defines the bounding box so one can use absolute coordinates

\begin{axis}[%
width=\linewidth,
height=2\linewidth,
scale only axis,
axis on top,
clip=false,
xmin=0,
xmax=1,
ymin=0,
ymax=2,
axis background/.style={fill=white},
ticks=none
]
\addplot [forget plot] graphics [xmin=0, xmax=1, ymin=0, ymax=2] {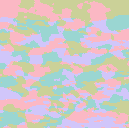};
\addplot[-Latex, 
color=black, 
point meta={sqrt((\thisrow{u})^2+(\thisrow{v})^2)}, 
point meta min=0, 
point meta max=0.1,
quiver={u=1.5*\thisrow{u}, v=1.5*\thisrow{v}, 
    every arrow/.append style={-{Latex[scale={1/1000*\pgfplotspointmetatransformed}]}}}]
 table[row sep=crcr] {%
x	y	u	v\\
0.0546875	0.109375	-0.0353553390593274	-0.0353553390593274\\
0.0546875	0.359375	0.0353553390593274	-0.0353553390593274\\
0.0546875	0.609375	-0.0353553390593274	0.0353553390593274\\
0.0546875	0.859375	0.0353553390593274	0.0353553390593274\\
0.0546875	1.109375	-0.0353553390593274	0.0353553390593274\\
0.0546875	1.359375	0.0353553390593274	-0.0353553390593274\\
0.0546875	1.609375	0.0353553390593274	0.0353553390593274\\
0.0546875	1.859375	0.0353553390593274	0.0353553390593274\\
0.1796875	0.109375	-0.0353553390593274	-0.0353553390593274\\
0.1796875	0.359375	0.0353553390593274	0.0353553390593274\\
0.1796875	0.609375	0.0353553390593274	-0.0353553390593274\\
0.1796875	0.859375	-0.0353553390593274	0.0353553390593274\\
0.1796875	1.109375	-0.0353553390593274	-0.0353553390593274\\
0.1796875	1.359375	-0.0353553390593274	0.0353553390593274\\
0.1796875	1.609375	-0.0353553390593274	-0.0353553390593274\\
0.1796875	1.859375	0.0353553390593274	0.0353553390593274\\
0.3046875	0.109375	-0.0353553390593274	-0.0353553390593274\\
0.3046875	0.359375	-0.0353553390593274	0.0353553390593274\\
0.3046875	0.609375	-0.0353553390593274	-0.0353553390593274\\
0.3046875	0.859375	0.0353553390593274	0.0353553390593274\\
0.3046875	1.109375	-0.0353553390593274	0.0353553390593274\\
0.3046875	1.359375	-0.0353553390593274	0.0353553390593274\\
0.3046875	1.609375	0.0353553390593274	-0.0353553390593274\\
0.3046875	1.859375	0.0353553390593274	0.0353553390593274\\
0.4296875	0.109375	0.0353553390593274	-0.0353553390593274\\
0.4296875	0.359375	0.0353553390593274	0.0353553390593274\\
0.4296875	0.609375	-0.0353553390593274	-0.00132554437269138\\
0.4296875	0.859375	-0.0353553390593274	-0.0353553390593274\\
0.4296875	1.109375	-0.0353553390593274	-0.0353553390593274\\
0.4296875	1.359375	0.0353553390593274	0.0353553390593274\\
0.4296875	1.609375	-0.0353553390593274	-0.0353553390593274\\
0.4296875	1.859375	0.0353553390593274	-0.0353553390593274\\
0.5546875	0.109375	-0.0353553390593274	0.0353553390593274\\
0.5546875	0.359375	0.0353553390593274	-0.0353553390593274\\
0.5546875	0.609375	-0.0353553390593274	-0.0353553390593274\\
0.5546875	0.859375	-0.0353553390593274	0.0353553390593274\\
0.5546875	1.109375	-0.0353553390593274	0.0353553390593274\\
0.5546875	1.359375	0.0353553390593274	-0.0353553390593274\\
0.5546875	1.609375	0.0353553390593274	0.0353553390593274\\
0.5546875	1.859375	0.0353553390593274	-0.0353553390593274\\
0.6796875	0.109375	0.0353553390593274	0.0353553390593274\\
0.6796875	0.359375	-0.0353553390593274	-0.0353553390593274\\
0.6796875	0.609375	0.0353553390593274	-0.0353553390593274\\
0.6796875	0.859375	-0.0353553390593274	-0.0353553390593274\\
0.6796875	1.109375	-0.0353553390593274	0.0353553390593274\\
0.6796875	1.359375	-0.0353553390593274	0.0353553390593274\\
0.6796875	1.609375	0.0353553390593274	-0.0353553390593274\\
0.6796875	1.859375	0.0353553390593274	-0.0353553390593274\\
0.8046875	0.109375	-0.0353553390593274	0.0353553390593274\\
0.8046875	0.359375	-0.0353553390593274	-0.0353553390593274\\
0.8046875	0.609375	-0.0353553390593274	0.0353553390593274\\
0.8046875	0.859375	-0.0353553390593274	0.0353553390593274\\
0.8046875	1.109375	-0.0353553390593274	-0.0353553390593274\\
0.8046875	1.359375	-0.0353553390593274	0.0353553390593274\\
0.8046875	1.609375	-0.0353553390593274	0.0353553390593274\\
0.8046875	1.859375	0.0353553390593274	0.0353553390593274\\
0.9296875	0.109375	0.0353553390593274	0.0353553390593274\\
0.9296875	0.359375	0.0353553390593274	-0.0353553390593274\\
0.9296875	0.609375	-0.0353553390593274	-0.0353553390593274\\
0.9296875	0.859375	-0.0353553390593274	-0.0353553390593274\\
0.9296875	1.109375	0.0353553390593274	0.0353553390593274\\
0.9296875	1.359375	0.0353553390593274	0.0353553390593274\\
0.9296875	1.609375	-0.0353553390593274	-0.0353553390593274\\
0.9296875	1.859375	0.0353553390593274	-0.0353553390593274\\
};
\end{axis}
\end{tikzpicture}%
\\
        \input{elast_box_c3_deform.tikz}
        \caption{concentric, $\alpha = 10^{-5}$}\label{fig:elastic:6}
    \end{subfigure}
    \caption{Control (top rows: phase and magnitude color coded as shown in color wheel with values in $\calM$ indicated, additionally indicated by arrows) and state (bottom row: target deformation in gray, achieved deformation in red) for the elasticity model problem}
    \label{fig:elasticExamples}
\end{figure}

Example \ref{fig:elastic:5} shows that the control is not guaranteed to take values in $\calM$; here, the target displacement $z$ is the displacement induced by a deadload to the left applied at the top domain boundary.
Since the target was induced by a forcing with zero load throughout the bulk material, the optimal control mainly takes the non-preferred value of zero.
However, a slight random perturbation of $z$ again leads to a pure multibang control, as shown in example \ref{fig:elastic:6}.

We again show the convergence behavior for the example in \cref{fig:elastic:3} in \cref{tab:iterElastic}. Since this example is linear, only a few Newton itertions ($2$ to $6$) are required for all values of $\gamma$, and correspondingly only few line searches are carried out for $\gamma<10^{-5}$. As before, the multibang structure is already strongly promoted for $\gamma\approx 10^{-6}$. (Let us point out that the elastic body is fixed at the bottom boundary so that the control has to be $0$ there, which for this example does not lie in $\calM$.)
\begin{table}[b]
    \caption{Convergence behavior for the example in \cref{fig:elastic:3}: number of semi-smooth Newton steps, number of times a line search was required, number of nodes with $u_\gamma(x)\notin\calM$}
    \label{tab:iterElastic}
    \centering
    \resizebox{\linewidth}{!}{%
        \begin{tabular}{lrrrrrrrrrrrrr}
            \toprule
            $\gamma$       & \num[round-precision = 0]{1.953e-1} & \num[round-precision = 0]{1.221e-2} & \num[round-precision = 0]{1.526e-3} & \num[round-precision = 0]{1.907e-4} & \num[round-precision = 0]{1.192e-5} & \num[round-precision = 0]{1.490e-6} & \num[round-precision = 0]{1.863e-7} & \num[round-precision = 0]{1.164e-8} & \num[round-precision = 0]{1.445e-9} & \num[round-precision = 0]{1.819e-10}\\
            \midrule
            \# SSN         & \num{2}    & \num{4}    & \num{5}    & \num{5}    & \num{4}   & \num{6}  & \num{4}  & \num{4}  & \num{5}  & \num{6} \\
            \# line search & \num{0}    & \num{0}    & \num{0}    & \num{0}    & \num{0}   & \num{2}  & \num{1}  & \num{2}  & \num{3}  & \num{4}\\
            \# not MB      & \num{4225} & \num{4210} & \num{3747} & \num{1245} & \num{179} & \num{84} & \num{71} & \num{68} & \num{68} & \num{68}  \\
            \bottomrule	
        \end{tabular}
    }
\end{table}

\clearpage

\section{Conclusion}

A preference for a small number of predefined discrete control values can be achieved by a piecewise affine pointwise regularization term whose corners lie at the preferred values.
In contrast to the case of scalar controls treated in \cite{CK:2013,CK:2015}, the case of vector-valued controls allows giving multiple control values an equal preference,
and numerical experiments show that this feature is indeed exploited by the optimal control.
Furthermore, the optimal control problems leading to admissible controls turn out to be dense among a family of control problems.
A more precise characterization of control problems with admissible solutions would be desirable and should be further investigated.
For instance, for certain control problems such as the elasticity-based example, one might conjecture that targets leading to non-multibang controls are nowhere dense.

\appendix

\section{Properties of Bloch equation}\label{sec:BlochProperties}
Here we verify that the state operator \eqref{eq:settingBloch} satisfies the required assumptions \ref{enm:weakWeakContinuity}--\ref{enm:adjointConvergence}.
In the following, a subscript to $\mathbf M^{(\omega)}$ and $\mathbf B^{(\omega)}$ shall always refer to the chosen control $u$.

\begin{proposition}\label{thm:operatorProperties}
    The operator $S$ as defined in \eqref{eq:settingBloch} is well-defined and satisfies \ref{enm:weakWeakContinuity}--\ref{enm:compactness}.
\end{proposition}
\begin{proof}
    Introducing the skew-symmetric matrix
    \begin{equation}
    B_u^\omega(t)=\left(\begin{smallmatrix}0&\omega&-(u(t))_2\\-\omega&0&(u(t))_1\\(u(t))_2&-(u(t))_1&0\end{smallmatrix}\right)\,,
    \end{equation}
    the homogeneous linear Bloch equation $\frac{\dd}{\dd t}{\mathbf{M}}^{(\omega)}_u(t) = B_u^\omega(t)\mathbf{M}^{(\omega)}_u(t)$ for a control $u(t)\in\R^2$ has a solution $\mathbf{M}^{(\omega)}_u(t)$ by Carathéodory's existence theorem. 
    Furthermore, 
    \begin{equation}
        \frac{\dd}{\dd t}|\mathbf{M}^{(\omega)}_u(t)|_2^2=2\mathbf{M}^{(\omega)}_u(t)\cdot\frac{\dd}{\dd t}{\mathbf{M}}^{(\omega)}_u(t)=0\,,
    \end{equation}
    and thus $|\mathbf{M}^{(\omega)}_u(t)|_2=1$ for all $t$.
    Now let $u_i\rightharpoonup u$ weakly in $L^2(\Omega;\R^2)$. Then,
    \begin{equation}
        \left\{\begin{aligned}
                \frac{\dd}{\dd t}{\mathbf{M}}^{(\omega)}_{u_i}(t)-\frac{\dd}{\dd t}{\mathbf{M}}^{(\omega)}_{u}(t)
                &=B_{u_i}^\omega(t)\left(\mathbf{M}^{(\omega)}_{u_i}(t)-\mathbf{M}^{(\omega)}_{u}(t)\right)
                +(B_{u_i}^\omega(t)-B_u^\omega(t))\mathbf{M}^{(\omega)}_{u}(t)\,,\quad t\in[0,T],\\
                \mathbf{M}^{(\omega)}_{u_i}(0) &= \mathbf{M}^{(\omega)}_{u}(0)\,.
        \end{aligned}\right.
    \end{equation}
    Upon abbreviating $\Delta M_i=\mathbf{M}^{(\omega)}_{u_i}-\mathbf{M}^{(\omega)}_{u}$ and $\Delta B_i=(B_{u_i}^\omega-B_u^\omega)\mathbf{M}^{(\omega)}_{u}$ and integrating from $0$ to $t$, we arrive at
    \begin{equation}
        \begin{aligned}
            |\Delta M_i(t)|_2
            &=\Big|\int_0^tB_{u_i}^\omega(s)\Delta M_i(s)\,\dd s+\int_0^t\Delta B_i(s)\,\dd s\,\Big|_2\\
            &\leq\int_0^t|B_{u_i}^\omega(s)|_2|\Delta M_i(s)|_2\,\dd s+\Big|\int_0^t\Delta B_i(s)\,\dd s\, \Big|_2\,.
        \end{aligned}
    \end{equation}
    Gronwall's inequality now implies that
    \begin{equation}\label{eq:Gronwall}
        \left|\Delta M_i(t)\right|_2
        \leq\Big|\int_0^t\Delta B_i(s)\,\dd s\,\Big|_2
        +\int_0^t\Big|\int_0^r\Delta B_i(s)\,\dd s\,\Big|_2|B_{u_i}^\omega(r)|_2\exp\left(\int_r^t|B_{u_i}^\omega(s)|_2\,\dd s\right)\,\dd r\,.
    \end{equation}
    The first term converges to zero due to $\Delta B_i\rightharpoonup0$ in $L^2(\Omega;\R^{3})$ (since $\mathbf{M}^{(\omega)}_{u}\in L^\infty(\Omega;\R^3)$).
    Additionally, the exponential is bounded by $\exp\left(\sqrt T\|B_{u_i}^\omega\|_{L^2(\Omega;\R^{3\times3})}\right)\leq C\in\R$ independent of $i$.
    Thus, the right-hand side converges to zero if
    \begin{equation}
        f_i\to0\text{ in }L^2(\Omega;\R)
        \qquad\text{for}\qquad
        f_i:\Omega\to\R,\quad r\mapsto\int_0^r\Delta B_i(s)\,\dd s\,.
    \end{equation}
    This is indeed the case since 
    \begin{equation}
        \|f_i\|_{L^2(\Omega)}^2=\int_{\{s\in(0,T)^3:s_1,s_2\leq s_3\}}\Delta B_i(s_1)\cdot\Delta B_i(s_2)\,\dd s
    \end{equation}
    and $s\mapsto \Delta B_i(s_1)\cdot\Delta B_i(s_2)$ converges weakly to zero in $L^2((0,T)^3;\R)$. 
    Thus $\mathbf M^{(\omega_j)}_{u_i}(T)$ converges for all $j$, and therefore $S(u_i)\to S(u)$.
    This argument also implies uniqueness of the solution.

    Moreover, $S$ is Fréchet-differentiable, and its derivative at $u\in L^2(\Omega;\R^2)$ is given by
    \begin{equation}
        S'(u):U\to Y,\quad
        \phi\mapsto \delta\mathbf M_\phi(T)=(\delta\mathbf M_\phi^{(\omega_1)}(T),\ldots,\delta\mathbf M_\phi^{(\omega_J)}(T))
        \quad
    \end{equation}
    with $\delta\mathbf M_\phi^{(\omega)}$ solving the linearized state equation (note $\partial_u(B_u^\omega)(\phi)=B_\phi^0$)
    \begin{equation}\label{eq:BlochFrechetEquation}
        \left\{\begin{aligned}
                \tfrac{\dd}{\dd t} \delta\mathbf M_\phi^{(\omega)}(t)&=B_{u}^{\omega}(t)\delta\mathbf M_\phi^{(\omega)}(t)+B_\phi^0(t)\mathbf M^{(\omega)}_{u}(t)\,,\qquad t\in[0,T],\\
                \delta\mathbf M_\phi^{(\omega)}(0)&=(0,0,0)^T\,.
        \end{aligned}\right.
    \end{equation}
    Indeed, $\delta\mathbf M_\phi(T)$ is obviously linear in $\phi$, and the unique solvability follows just like for $\mathbf M_u^{(\omega)}$.
    Furthermore, for any $\tilde u\in U$ with $\|\tilde u-u\|_{U}\leq1$ and $\phi=\tilde u-u$ we have
    \begin{equation}
        \tfrac\dd{\dd t}(\mathbf{M}^{(\omega)}_{\tilde u}-\mathbf{M}^{(\omega)}_{u}-\delta\mathbf M_\phi^{(\omega)})
        =B_{\tilde u}^{\omega}(\mathbf{M}^{(\omega)}_{\tilde u}-\mathbf{M}^{(\omega)}_{u}-\delta\mathbf M_\phi^{(\omega)})
        +(B_{\tilde u}^{\omega}-B_{u}^{\omega})\delta\mathbf M_\phi^{(\omega)}
    \end{equation}
    with zero initial condition.
    Gronwall estimates analogous to \eqref{eq:Gronwall} (now for $\delta\mathbf M_\phi^{(\omega)}$ and $\mathbf{M}^{(\omega)}_{\tilde u}-\mathbf{M}^{(\omega)}_{u}-\delta\mathbf M_\phi^{(\omega)}$, exploiting that $|B_{\tilde u}^\omega(r)|_2\exp\left(\int_r^t|B_{\tilde u}^\omega(s)|_2\,\dd s\right)$ is bounded by a constant only depending on $\|u\|_{U}$) imply that
    \begin{equation}
        |\delta\mathbf M_\phi^{(\omega)}(t)|_2\leq\tilde C\|B_{\tilde u}^{\omega}-B_{u}^{\omega}\|_{L^2(\Omega;\R^{3\times 3})}\leq2\tilde C\|\tilde u-u\|_{U}
    \end{equation}
    for a constant $\tilde C>0$ and all $t\in\Omega$ as well as
    \begin{equation}
        \begin{aligned}
            |\mathbf{M}^{(\omega)}_{\tilde u}(T)-\mathbf{M}^{(\omega)}_{u}(T)-\delta\mathbf M_\phi^{(\omega)}(T)|_2
            &\leq C\sup_{t\in\Omega}\Big|\int_0^t(B_{\tilde u}^{\omega}(s)-B_{u}^{\omega}(s))\delta\mathbf M_\phi^{(\omega)}(s)\,\dd s\,\Big|_2\\
            &\leq C\|B_{\tilde u}^{\omega}-B_{u}^{\omega}\|_{L^1(\Omega;\R^{3\times 3})}\|\delta\mathbf M_\phi^{(\omega)}\|_{L^\infty(\Omega;\R^3)}\\
            &\leq C\|\tilde u-u\|_{U}^2\,,
        \end{aligned}
    \end{equation}
    where $C$ denotes a positive constant (not necessarily the same in all inequalities).
    We thus have 
    \begin{equation}
        |S(\tilde u)-S(u)-S'(u)(\tilde u-u)|_2\leq C\|\tilde u-u\|_{U}^2 
    \end{equation}
    as required.

    The compactness follows from the finite dimensionality of $\ran S$.
\end{proof}

We will also require some regularity results for the adjoint operator $S'(u)^*$.
\begin{proposition}\label{prop:bloch_adjoint}
    For $S$ from \eqref{eq:settingBloch} and $u\in U$ we have
    \begin{equation}
        S'(u)^*:Y\to U\,,\quad
    (S'(u)^*y)(t)=\sum_{j=1}^J\left(\begin{smallmatrix}0&\left(\mathbf M_u^{(\omega_j)}(t)\right)_3&-\left(\mathbf M_u^{(\omega_j)}(t)\right)_2\\-\left(\mathbf M_u^{(\omega_j)}(t)\right)_3&0&\left(\mathbf M_u^{(\omega_j)}(t)\right)_1\end{smallmatrix}\right)\mathbf{\Psi}_{u,j}(t)\,,
    \end{equation}
    where $\mathbf{\Psi}_{u,j}$ solves the adjoint equation
    \begin{equation}
        \frac{\dd}{\dd t}\mathbf{\Psi}_{u,j}(t)=\mathbf{\Psi}_{u,j}(t)\times\mathbf B_u^{(\omega_j)}(t)\,,\quad
        \mathbf{\Psi}_{u,j}(T)=y_j\,,\quad
        j=1,\ldots,J\,.
    \end{equation}
\end{proposition}
\begin{proof}
    From \cref{thm:operatorProperties} we have   $S'(u)\phi=(\delta\mathbf M_\phi^{(\omega_1)}(T),\ldots,\delta\mathbf M_\phi^{(\omega_J)}(T))$ for any $u,\phi\in U$ with $\delta\mathbf M_\phi^{(\omega)}$ solving \eqref{eq:BlochFrechetEquation}.
    Thus we obtain for $y\in (\R^3)^J$ that
    \begin{equation}
        \begin{aligned}
            \int_\Omega \phi(t)\cdot(S'(u)^* y)(t)\,\dd t
            =\scalprod{ y,S'(u)\phi}
            &=\sum_{j=1}^J y_j^T\delta\mathbf M_\phi^{(\omega_j)}(T)
            =\sum_{j=1}^J\mathbf{\Psi}_{u,j}(T)^T\delta\mathbf M_\phi^{(\omega_j)}(T)\\
            &=\sum_{j=1}^J\int_\Omega\mathbf{\Psi}_{u,j}(t)^T\frac{\dd}{\dd t} \delta\mathbf M_\phi^{(\omega_j)}(t)+\frac{\dd}{\dd t}\mathbf{\Psi}_{u,j}(t)^T\delta\mathbf M_\phi^{(\omega_j)}(t)\,\dd t\\
            &=\sum_{j=1}^J\int_\Omega\mathbf{\Psi}_{u,j}(t)^T\left[\frac{\dd}{\dd t} \delta\mathbf M_\phi^{(\omega_j)}(t)-\delta\mathbf M_\phi^{(\omega_j)}(t)\times\mathbf B_u^{\omega_j}(t)\right]\,\dd t\\
            & =\sum_{j=1}^J\int_\Omega\mathbf{\Psi}_{u,j}(t)^T\left[B_\phi^0(t)\mathbf M_u^{(\omega_j)}(t)\right]\,\dd t\,,
        \end{aligned}
    \end{equation}
    from which the result follows.
\end{proof}

\begin{proposition}\label{prop:blochAdjointContinuity}
    For any $u\in U$, we have  $\ran S'(u)^*\hookrightarrow L^\infty(\Omega;\R^2)$.
    Moreover, $u\mapsto S'(u)^*$ is continuous in $L^\infty(\Omega;\R^2)$ under weak convergence of $u$ in $U$, thus it satisfies \ref{enm:adjointConvergence} with $V=L^1(\Omega;\R^2)$.
\end{proposition}
\begin{proof}
    By the formula for $S'(u)^*$ from \cref{prop:bloch_adjoint}, it is enough to show that $\mathbf M_{u_i}^{(\omega_j)}$ and $\mathbf{\Psi}_{u_i,j}$ converge in $L^\infty(\Omega;\R^3)$ as $u_i\rightharpoonup u$ in $U$.
    It suffices to consider $\mathbf M_{u_i}^{(\omega_j)}$, since the adjoint variable $\mathbf{\Psi}_{u,j}$ satisfies the same differential equation.
    Thus, we only have to show that the right-hand side in \eqref{eq:Gronwall} converges to zero uniformly in $t$.
    Note that the second integral is bounded above by the one for $t=T$ which has already been shown to converge to zero. Hence it suffices to show $\int_0^t\Delta B_i(s)\,\dd s\to0$ uniformly in $t$ as $i\to\infty$.
    Since $\Delta B_i\rightharpoonup0$ in $L^2(\Omega;\R^3)$, we also have weak convergence in $L^1(\Omega;\R^3)$ so that by the Dunford--Pettis criterion the $\Delta B_i$ are equi-integrable.
    Now let $t_i\in[0,T]$ be such that $\big|\int_0^{t_i}\Delta B_i(s)\,\dd s\,\big|_2\geq\sup_{t\in[0,T]}\big|\int_0^{t}\Delta B_i(s)\,\dd s\,\big|_2-\frac1i$,
    and assume that for a subsequence (still indexed by $i$) we have $\big|\int_0^{t_i}\Delta B_i(s)\,\dd s\,\big|_2\geq C>0$ for all $i$.
    Upon taking another subsequence, we can further assume that $t_i\to\hat t\in[0,T]$.
    Due to the equi-integrability, there is a $\Delta t>0$ such that $\int_{\hat t-\Delta t}^{\hat t+\Delta t}|\Delta B_i(s)|_2\,\dd s<C/2$;
    thus for $i$ large enough we have $\big|\int_0^{\hat t}\Delta B_i(s)\,\dd s\,\big|_2\geq C/2$.
    However, this contradicts the weak convergence of $\Delta B_i$ to $0$ so that indeed $\int_0^t\Delta B_i(s)\,\dd s\to0$ uniformly in $t$ as $i\to\infty$.
\end{proof}

\section*{Acknowledgments}

CC is supported by the German Science Fund (DFG) under grant CL 487/1-1.
CT gratefully acknowledges support by the DFG Research Training Group 2088 Project A1.
BW’s research was supported by the Alfried Krupp Prize for Young University Teachers awarded by the Alfried Krupp von Bohlen und Halbach-Stiftung.
The work was also supported by the Deutsche Forschungsgemeinschaft (DFG), Cells-in-Motion Cluster of Excellence (EXC1003 -- CiM), University of Münster, Germany. 

\printbibliography

\end{document}